\documentclass[reqno]{amsart}
\usepackage{amsmath, amsthm, amssymb,amsfonts,color, dsfont, bm}
\usepackage{hyperref,comment}
\usepackage{esint}
\usepackage[latin1]{inputenc}

\usepackage{tikz-cd}

\usepackage[shortlabels]{enumitem}

\usepackage{accents}

\allowdisplaybreaks

\newcommand\bH{\mathbb{H}}

\newcommand\bR{\mathbb{R}}

\newcommand\bE{\mathbf{E}}

\newcommand\bB{\mathbf{B}}

\newcommand\cC{\mathcal{C}}

\newcommand\cD{\mathcal{D}}

\newcommand\cE{\mathcal{E}}

\newcommand\cH{\mathcal{H}}

\newcommand\cI{\mathcal{I}}

\newcommand\cS{\mathcal{S}}

 \theoremstyle{definition}

\newtheorem{theorem}{Theorem}[section]
\newtheorem{lemma}[theorem]{Lemma}
\newtheorem{corollary}[theorem]{Corollary}

\newtheorem{proposition}[theorem]{Proposition}

\newtheorem{definition}{Definition}[section]

\theoremstyle{remark}
\newtheorem{remark}[theorem]{Remark}

\newtheorem{assumption}[theorem]{Assumption}

\numberwithin{equation}{section}

\newcommand\sff{\mathsf{f}}

\usepackage{imakeidx}
\indexsetup{noclearpage}
\makeindex

\makeatletter
\newcommand{\hathat}[1]{%
\begingroup%
  \let\macc@kerna\z@%
  \let\macc@kernb\z@%
  \let\macc@nucleus\@empty%
  \hat{\raisebox{.4ex}{\vphantom{\ensuremath{#1}}}\smash{\hat{#1}}}%
\endgroup%
}
\makeatother

\makeatletter
\newcommand{\dtilde}[1]{%
\begingroup%
  \let\macc@kerna\z@%
  \let\macc@kernb\z@%
  \let\macc@nucleus\@empty%
  \tilde{\raisebox{.4ex}{\vphantom{\ensuremath{#1}}}\smash{\tilde{#1}}}%
\endgroup%
}
\makeatother

\begin{document}

\title[VML with the SBC]{The Local-well-posedness of the relativistic Vlasov-Maxwell-Landau system with the specular reflection boundary condition}

\author[H. Dong]{Hongjie Dong}
\address[H. Dong]{Division of Applied Mathematics, Brown University, 182 George Street, Providence, RI 02912, USA}
\email{Hongjie\_Dong@brown.edu }
\thanks{H. Dong was partially supported by a Simons fellowship, grant no. 007638 and the NSF under agreement DMS-2055244. Y. Guo's research was supported in part by NSF Grant DMS-2106650.}

\author[Y. Guo]{Yan Guo}
\address[Y. Guo]{Division of Applied Mathematics, Brown University, 182 George Street, Providence, RI 02912, USA}
\email{Yan\_Guo@brown.edu}

\author[Z. Ouyang]{Zhimeng Ouyang}
\address[Z. Ouyang]{Department of Mathematics, University of Chicago, 5734 S. University Avenue, 
Chicago, IL, 60637}
\email{ouyangzm9386@uchicago.edu}

\author[T. Yastrzhembskiy]{Timur Yastrzhembskiy}
\address[T. Yastrzhembskiy]{Division of Applied Mathematics, Brown University, 182 George Street, Providence, RI 02912, USA}
\email{Timur\_Yastrzhembskiy@brown.edu}


\begin{abstract}
We prove the local-in-time well-posedness of the relativistic Vlasov-Maxwell-Landau system in a bounded domain $\Omega$ with the specular reflection condition. Our result covers the case when $\Omega$ is a non-convex domain, e.g., solid torus.
To the best of our knowledge, this is the first local well-posedness result for a nonlinear kinetic model with a self-consistent magnetic effect in a three-dimensional \textbf{bounded} domain. 
\end{abstract}

\maketitle

\tableofcontents

\section{Introduction}
                \label{section 1}
Let $z = (t, x, p)$, where $t \in \bR$, $x, p \in \bR^3$ are the temporal, spatial, and momentum variables, respectively. For a spatial domain $\Omega \subset \bR^3$, we denote the incoming/outgoing boundaries and the grazing sets, respectively, as follows:
\begin{align*}
  & \index{$\gamma_{\pm}$}  \gamma_{-} = \{(x, p): x \in \partial \Omega,  n_x \cdot p < 0 \},  \quad
    \gamma_{+} = \{(x, p): x \in \partial \Omega,  n_x \cdot p > 0 \}, \\
   & \index{$\gamma_0$} \gamma_0 = \{(x, p): n_x \cdot p = 0\},
\end{align*}
 where $n_x$ is an outward unit normal vector at $x \in \partial \Omega$.
Furthermore, we denote 
$$
   \index{$p_0$}  p_0 = (1+|p|^2)^{1/2}, \quad \index{$v (p)$} v (p) = \frac{p}{p_0}.
$$

We study the relativistic Vlasov-Maxwell-Landau (RVML) system  in a bounded domain:
\begin{equation}
    \label{VML}
\begin{aligned}
&
\partial_t F^{+} + v (p) \cdot \nabla_x F^{+} + (\bE + v (p) \times \bB) \cdot \nabla_p F^{+} = \cC (F^{+}, F^{+}) + \cC (F^{+}, F^{-}),\\
&
	\partial_t F^{-} + v (p) \cdot \nabla_x F^{-} - (\bE + v (p) \times \bB) \cdot \nabla_p F^{-} = \cC (F^{-}, F^{-}) + \cC (F^{-}, F^{+}),\\
			&\partial_t \bE - \nabla_x \times \bB = -  \int v (p) (F^{+} - F^{-}) \, dp,\\
 &
	\partial_t \bB +  \nabla_x \times \bE = 0, \\
 &
	\nabla_x \cdot \bE = \int (F^{+} - F^{-}) \, dp,
	   \quad \nabla_x \cdot \bB = 0,\\
	   & (\bE \times n_x)|_{\partial \Omega} = 0, \quad (\bB \cdot n_x)|_{\partial \Omega} = 0
\end{aligned}
\end{equation}
with the initial conditions 
$$
    F^{\pm} (0, \cdot) = F_0^{\pm} (\cdot), \, \,  \bE (0, \cdot) = \bE_0 (\cdot), \, \,  \bB (0, \cdot) = \bB_0 (\cdot),
$$ and the specular reflection boundary condition (SRBC)
$$
    F^{\pm} (t, x, p) = F^{\pm} (t, x, R_x p), \quad \index{$R_x p$} R_x p = p - 2 (p \cdot n_x) n_x.
$$
Here $F^{+}$ and $F^{-}$ are the density functions of ions and electrons, respectively, and $\mathcal{C}$ is the relativistic Landau collision operator given by
   \begin{equation}
            \label{eq1.1.2}
 \index{$\cC (f, g)$}	\cC (f, g) (p) =  \nabla_p \cdot \int_{\bR^3}  \Phi (P, Q) \big(\nabla_p f (p) g (q) -   f (p) \nabla_q g (q)\big) \, dq,
  \end{equation}
  where $\Phi (P, Q)$ is the Belyaev-Budker kernel defined in Section \ref{section 2} (see \eqref{eq1.11}). The RVML is a fundamental model of a hot dilute collisional plasma with magnetic and relativistic effects. Such a model is relevant for plasma fusion in tokamaks, where plasma particles may reach high velocities. For the formal derivation, see, for example,  \cite{LP_81}.
  
 For the sake of simplicity,  all the physical constants are set to $1$ (cf. \cite{GS_03}) since the exact relationships among them do not play any role in our analysis.
Our goal is to prove the local-in-time well-posedness of the RVML system near the relativistic Maxwellian
\begin{equation}
            \label{eq1.1.3}
    \index{$J (p)$}	J (p) = e^{- p_0 },
\end{equation}
which is called the J\"uttner's solution.  The global-well-posedness of the RVML system was first established in  \cite{GS_03} for the periodic boundary conditions,  and later, this result was extended to the whole space in \cite{YY_12}.
For the related studies of this model, see  \cite{LZ_14} and \cite{X_15}.


The presence of spatial boundaries is natural in kinetic models, and the study of boundary value problems is one of the foci of contemporary kinetic  PDE theory. In this context, the investigation of hyperbolic kinetic models poses a formidable challenge due to the non-uniformly characteristic nature of the grazing set $\gamma_0$ associated with the free streaming operator $\partial_t + p \cdot \nabla_x$. Near the grazing set,  the regularity of a solution is expected to deteriorate significantly, resulting in profound mathematical intricacies. The standard energy techniques, which typically rely on differentiating with respect to spatial and velocity variables, become inadequate in such a scenario.

Particularly noteworthy is the occurrence of singularities emanating from the grazing set in non-convex domains \cite{K_11}, where hyperbolic kinetic PDEs are expected to yield solutions of, at best, bounded variation \cite{GKTT_16}. Furthermore, the introduction of magnetic effects can trigger singularities even in a half-space domain. An illustrative example is the one-dimensional relativistic Vlasov-Maxwell (RVM) system subject to the perfect conductor boundary conditions  \cite{G_95} (see also \cite{G_94} for an example in a three-dimensional half-space). This specific case underscores the limited knowledge we possess, as only the global existence of a weak solution is currently known for the RVM system in a three-dimensional bounded domain \cite{G_93}.

In stark contrast, in convex domains and in the absence of magnetic effects, recent papers have demonstrated global well-posedness for several important hyperbolic plasma models such as the Vlasov-Poisson and Vlasov-Poisson-Boltzmann systems \cite{H_04}, \cite{HV_10}, \cite{CKD_19}.

Conversely, when velocity diffusion is introduced, a higher degree of regularity near the grazing set is expected, owing to a hypoelliptic gain \cite{S_22}.  The nature of this regularity depends on the specific boundary conditions imposed on the outgoing boundary $\gamma_{-}$. Notably, a linear kinetic Fokker-Planck equation (KFP) with the inflow (Dirichlet) boundary conditions is anticipated to exhibit at most H\"older regularity in both spatial and velocity variables \cite{HJL_14}.

However, in the presence of the SRBC, a unique avenue opens. Employing a flattening and extension strategy, one can extend the solution of the KFP equation to the entire space and invoke the $S_p$ theory of KFP equations, akin to the Calderon-Zygmund theory for parabolic PDEs  \cite{GHJO_20}, \cite{DGO_22}, \cite{DGY_21}. This approach yields H\"older regularity not only for the solution but also for its velocity gradient. Such an extension argument is unknown for other boundary conditions in kinetic theory.

In recent years, an $L_2$ to $L_{\infty}$ framework has been developed for the Boltzmann equation in bounded domains (see \cite{G_10} and \cite{EGKM_13}, \cite{EGKM_18}, \cite{KGH_20} for further developments). The method is based on interpolating between the 
 natural entropy or energy bound and the interplay between characteristics and velocity averaging in the collision (\cite{G_10QAM}). However, 
  this approach is less applicable to the Landau equation due to the absence of characteristic curves. We emphasize that a higher regularity of the velocity gradient is required to establish the uniqueness for the Landau equation due to the nonlinear diffusion term (see \cite{KGH_20}). To handle the  Landau and the Vlasov-Poisson-Landau equations with the SRBC, the authors of \cite{GHJO_20} and \cite{DGO_22} combined the aforementioned mirror-extension method with the $S_p$ estimate. Their results require merely $C^2$ regularity of domains and, hence, allow a solid tori domain, which resembles a tokamak.

Adapting the aforementioned framework to the  RVML system poses a formidable challenge due to the anticipated low regularity of solutions to Maxwell's equations. The intricate nature of the relativistic Landau kernel, coupled with the presence of the relativistic transport term, introduces additional mathematical complexities. Our innovative approach involves deducing the regularity of solutions to Maxwell's equations by treating them as an elliptic system of the Hodge type. This inspired the development of a delicate iteration scheme, where we propagate temporal derivatives and employ a descent argument, leveraging div-curl estimates and a relativistic adaptation of the $S_p$ estimates for KFP equations with the SRBC.

Our main result is, informally speaking, the following (see Theorem \ref{theorem 5.1}):   if $F_0 - J$, $\bE_0$, $\bB_0$ are of order  $\varepsilon$ in some sense, then, the RVML system has a unique strong solution $[F, \bE_f, \bB_f]$ on $[0, T]$ for some $T > 0$ such that $F^{\pm} - J, \bE_f, \bB_f$ are of order $\varepsilon$. 
Due to the delicate behavior of kinetic PDEs near the boundary, there have been 
 few results on well-posedness for any kinetic models with a self-consistent magnetic effect in three-dimensional domains (see \cite{CK_23} for the result on RVM in a half-space). To the best of our knowledge, Theorem \ref{theorem 5.1} provides the first well-posedness result for the system with the Vlasov-Maxwell structure in a 3d \textbf{bounded} domain.  In a separate paper \cite{GWP_23}, the first, second, and fourth authors established a global estimate and asymptotic stability for the RVML system near a global J\"uttner's solution.

\section{Notation and conventions}
                                            \label{section 2}
Before we state the main results, we introduce some notation.
Throughout the paper, $T > 0$ is a number.
\begin{itemize}
\item Geometric notation.
\begin{align}
 \label{eq1.7}
  & \index{ $P$}	P=  (p_0, p), \quad Q =  (q_0, q), \quad \index{ $P \cdot Q$} P \cdot Q = p_0 q_0 - p \cdot q,\\
   & \label{not1}
   B_r  (x_0) = \{x \in \bR^3: |x - x_0| < r\}, \quad 
  \index{$\Omega_{r} (x_0)$} \Omega_{r} (x_0) = \Omega \cap B_{r} (x_0), \\
	& 
\label{eq1.2.0}
 \bR^3_{\pm} = \{x \in \bR^3: \pm x_3 > 0\},
 \quad   \index{$\mathbb{H}_{\pm}$} \mathbb{H}_{\pm} = \{(x, p) \in \bR^3_{\pm} \times \bR^3\}, \\
&
  \index{$\mathbb{H}_{\pm}^T$}  \mathbb{H}_{\pm}^T = \{z \in (0, T) \times  \mathbb{H}_{\pm}\}, \quad  \bR^7_T = \{z \in (0, T) \times \bR^6\},\notag \\
&
 \index{$\Sigma^T$} \Sigma^T = (0, T)\times \Omega \times \bR^3,\quad \Sigma^T_{\pm} = (0,T) \times \gamma_{\pm} \notag.
\end{align}

\item Matrix notation.
\begin{align}
   & \index{$\bm{1}_3$} \bm{1}_3 = \text{diag} (1, 1, 1), \quad \index{$\bm{R}$} \bm{R} = \text{diag} (1, 1, -1), \notag\\
&  
	 \label{eq1.10}
  \index{$\bm{\xi}$}  \bm{\xi} = \text{diag} (1, -1),
\quad \index{$\bm{\xi}_0$} \bm{\xi}_0 =\begin{bmatrix} 1 \\ 1 \end{bmatrix}, \quad 	\quad \index{$\bm{\xi}_1$} \bm{\xi}_1 = \begin{bmatrix} -1 \\ 1 \end{bmatrix}.
\end{align}

\item Relativistic kinetic transport operator.
\begin{equation*}
\begin{aligned}
 &  \index{$Y$} Y  = \partial_t + v (p) \cdot \nabla_x.
\end{aligned}
\end{equation*}

\item Relativistic  Belyaev-Budker kernel. We introduce
 \begin{align}
 \label{eq1.8}
&	\index{$\Lambda (P, Q)$} \Lambda (P, Q) = (P\cdot Q)^2 \big((P\cdot Q)^2 - 1\big)^{-3/2},\\
 \label{eq1.9}
	& \index{$S (P, Q)$} S (P, Q)  =  \big((P\cdot Q)^2 - 1\big) \bm{1}_3  - (p-q) \otimes (p-q) \\
	& \quad \quad  \quad \quad+  (P\cdot Q - 1)(p \otimes q  + q \otimes p),\notag\\
&
\label{eq1.11}
\index{$\Phi (P, Q)$}	\Phi (P, Q) =\frac{ \Lambda (P, Q)}{p_0 q_0} S (P, Q).
\end{align}

\item Function spaces. Let $G \subset \bR^7$ be an open set.
\begin{itemize}[--]
\item[--] $C (\overline{G})$ be the set of all bounded continuous functions on $\overline{G}$, 
and $C^k (\overline{G}), k \in \{1, 2, \ldots\}$ be the subspace of $C (\overline{G})$  functions
with  partial derivatives up to order $k$ belonging to $C (\overline{G})$.

\item[--] \index{$C^1_0$} $C^1_0 (\overline{G})$ ($C^{0, 1}_0 (\overline{G})$) is the subset of $C^1 (\overline{G})$   (Lipschitz) functions on $G$  that vanish for large $z$. Similarly, one can define $C^k_0 (\overline{G}), k \in \{2, 3, \ldots\}$.

\item[--] $C^{\infty}_0 (G)$ is the set of all infinitely differentiable functions with the support contained in $G$.

\item[--] \textit{Anisotropic H\"older space.}
For  an open set $D \subset \bR^6$ and $\alpha \in (0, 1]$,
by $C^{\alpha/3, \alpha}_{x, p} (D)$,
we denote the set of all bounded functions
$f = f (x, p)$ such
that
\begin{align*}
        &\index{$C^{\alpha/3, \alpha}_{x, p} (D)$} [f]_{C^{\alpha/3, \alpha}_{x, p} (D)  } \\
        &:= \sup_{ (x_i, p_i) \in \overline{D}:  (x_1, p_1) \neq (x_2, p_2) } \frac{|f (x_1, p_1)- f (x_2, p_2)|}{(|x_1-x_2|^{1/3}+|p_1 - p_2|)^{\alpha}} < \infty.
\end{align*}
Furthermore, the norm is given by
\begin{equation}
			\label{1.2.0}
    \|f\|_{C^{\alpha/3, \alpha}_{ x, p } (D)  } : = \|f\|_{  L_{\infty} (D) }  +  [f]_{C^{\alpha/3, \alpha}_{x, p} (D)  }.
\end{equation}

\textit{Weighted spaces on the kinetic boundary.}
For a weight $\omega \ge 0$ on $\partial \Omega \times \bR^3$, we set
\begin{align}
        \label{1.2.0.0}
    \|f\|^2_{ L_{2} (\Sigma^T_{\pm}, \omega) }  = \int_{ \Sigma^T_{\pm}} f^2 \omega \, dS_x dp.
\end{align}

\item \textit{Traces}. 
Let  $r \in [1, \infty)$ and $f \in L_r (\Sigma^T)$  be a function such that $Y  f \in L_r (\Sigma^T)$. Then, the traces of $f$ can be defined (see the details in Section \ref{section 9}). In particular, there exist functions  $(f_T, f_0, f_{+}, f_{-})$, which we call traces of $f$, such that a variant of Green's identity holds (see Proposition \ref{proposition C.1}).

\item 
\textit{Weighted Lebesgue space.}
For $\theta \in \bR$ and $r \in [1, \infty]$, by $L_{r, \theta} (G)$
we denote the set of all Lebesgue measurable functions $u$
such that
$$
  \index{$L_{r, \theta}$}  \|u\|_{ L_{r, \theta} (G) }:= \|p_0^{\theta} u\|_{L_r (G)} < \infty.
$$

\item \textit{Weighted Sobolev spaces.}
For $r \in [1, \infty]$, by $W^1_{r, \theta} (\bR^3)$ we denote the Banach space of functions  $u  \in L_{r, \theta} (\bR^3)$ such that  the norm
$$
   \index{$W^1_{r, \theta} (\bR^3)$} \|u\|_{ W^1_{r, \theta} (\bR^3)} :=   \||u|+|\nabla_p u|\|_{ L_{r, \theta} (\bR^3) } < \infty.
$$
For $\theta = 0$, we set $W^1_r (\bR^3): = W^1_{r, 0} (\bR^3)$.

\item \textit{Dual Sobolev space.} \index{$W^{-1}_{2,   \theta} (\bR^3)$} Let
$W^{-1}_{2,   \theta} (\bR^3)$  be the space of all distributions  $u$
such that 
\begin{align}
      \label{eq1.2.15}
        u =  \partial_{p_i}  \eta_i + \xi
\end{align}
for some $\xi, \eta_i \in L_{2, \theta} (\bR^3), i = 1, 2, 3$.
Furthermore, for $u \in W^{-1}_2 (\bR^3)$ and $f \in W^{1}_2 (\bR^3)$, by
\begin{align}
      \label{eq1.2.16}
 \index{$ \langle \cdot, \cdot \rangle$}   \langle u, f \rangle =  \int_{\bR^3}  (- \eta_i \cdot \partial_{p_i} f+ \xi f) \, dp,
\end{align}
we denote the duality pairing between $W^{-1}_2 (\bR^3)$ and $W^{1}_2 (\bR^3)$, which is independent of the choice of $\eta_i$ and $\xi$.

\item \textit{Non-relativistic (Newtonian) Kinetic  Sobolev space.} 
$$
  \quad \index{unsteady $S_r^N$ space}  S_{r}^{N} (G) = \{f \in L_{r} (G): (\partial_t + p \cdot \nabla_x) f, \nabla_p f, D^2_{p} f \in L_{r} (G)\},
$$    
and the norm is defined as follows:
\begin{equation}
			\label{eq1.2.4.0}
 \quad \quad   \quad \quad \quad     \|f\|_{S_{r}^{N} (G)} = \||f| + |\nabla_p f| + |D^2_p f| + |(\partial_t +p \cdot \nabla_x) f|\|_{ L_{r} (G) }.
\end{equation}

\item \textit{Mixed-norm spaces.} For normed spaces $X$ and $Y$, we write $u = u (x, y) \in X Y$ if for each $x$, $u_x : = u (x, \cdot) \in Y$, and 
$$
   \index{Mixed-norm spaces $X Y$}  \|u\|_{ X Y  }  :=   \| \|u_x\|_{ Y } \|_X  < \infty.
$$

\item \textit{Weighted unsteady relativistic kinetic Sobolev spaces.}
Let $S_{r, \theta} (G) = \{f \in L_{r, \theta} (G): Y f, \nabla_p f, D^2_p f \in L_{r, \theta} (G)\}$
be  the Banach space  with the norm
\begin{equation}
			\label{1.2.4}
 \quad \quad   \quad \quad \quad  \quad \index{unsteady $S_r$ space} \|f\|_{S_{r, \theta} (G)} = \||f| + |\nabla_p f| + |D^2_p f| + |Y  f|\|_{ L_{r, \theta} (G) }.
\end{equation}
In the case when $\theta = 0$, we set
$S_{r} (G) = S_{r, 0} (G)$.

\item 
 \textit{Steady $S_r$ spaces.}  For $r \in [1, \infty]$, by $ S_{ r, \theta} (\Omega \times \bR^3)$, we denote the set of all functions $u$ on $\Omega \times \bR^3$ such that
\begin{equation}
                \label{eq1.2.7.0}    
  \begin{aligned}
 &\index{steady $S_r$ space}  u, v (p) \cdot  \nabla_x u, \nabla_p u, D^2_p u \in L_{r, \theta} (\Omega \times \bR^3)
 \end{aligned}
 \end{equation}
 The norm is given by
\begin{equation}
			\label{eq1.2.7}
\begin{aligned}
	& \quad \quad \quad  \|u\|_{ S_{ r, \theta} (\Omega \times \bR^3) }  =  \||u| + |v (p) \cdot \nabla_x u| + |\nabla_p u| + |D^2_p u|\|_{ L_{r, \theta} ( \Omega \times \bR^3 ) }.
\end{aligned}
\end{equation}
\end{itemize}

\textit{Vector fields.} We use boldface letters to denote vector fields. We write $\bm{u} \in  X$, where $X$ is some vector space if each component of $\bm{u}$ belongs to $X$.

\item 
\textit{Conventions.}
\begin{itemize}
\item We assume the summation with respect to repeated indexes.
\item  If functions $f$ and $g$ are defined on $D \subset \bR^3$ and $\bR^3$, respectively, and  $g$ vanishes outside $D$, 
then,  for $x \not \in D$, we set $(f g) (x) = 0$.

\item  By $N = N (\cdots)$, we denote a constant depending only on the parameters inside the
parentheses. The constants $N$ might change from line to line. Sometimes, when it is clear what parameters $N$  depends on, we omit them.
\end{itemize}

\end{itemize}

\section{Main results}
            \label{section 3}
Let  $f = (f^{+}, f^{-})$ be  perturbations  of $F^{\pm}$ near the relativistic Maxwellian given by
$$
    F^{\pm} = J + J^{1/2} f^{\pm}
$$ 
(see \eqref{eq1.1.3}).
We denote
\begin{equation}
            \label{eq5}    
\begin{aligned}
&
\index{$L$}	L f = -A f - K f,\\
 &
\index{$A$}	A_{\pm} f = 2 J^{-1/2} \cC (J^{1/2} f^{\pm}, J), \quad  A f = (A_{+} f, A_{-} f),\\
&
\index{$K$}	K f =   J^{-1/2} \cC \big(J, J^{1/2} (f^{+}+f^{-})\big) \bm{\xi}_0,\\
 &
	\Gamma_{\pm} (f, g) = J^{-1/2}\cC \big(J^{1/2} f^{\pm}, J^{1/2} (g^{+} + g^{-})\big),\\
		&
\index{$\Gamma (f, g)$}	\Gamma (f, g) = (\Gamma_{+} (f, g), \Gamma_{-} (f, g)),
\end{aligned}
\end{equation}
where $\cC$ is defined by \eqref{eq1.1.2} and \eqref{eq1.8} - \eqref{eq1.11}.
Then, the triple $[f, \bE, \bB]$ satisfies the following system (see p. 276 in \cite{GS_03}):
\begin{align}
            \label{eq5.0}
	&Y f = -\bm{\xi} (\bE + v (p) \times \bB) \cdot \nabla_p f +  \frac{\bm{\xi}_1}{2}  (v (p) \cdot \bE) f
 + A f  \\
	&+  \bm{\xi} (v (p) \cdot \bE)  J^{1/2}  + K f + \Gamma (f, f), \notag\\
		& f (0, \cdot) = f_0 (\cdot), \, \,  f (t, x, p) = f (t, x, R_x p), \, \, z \in \Sigma^T_{-}, \notag\\
             \label{eq5.0.1}
	& \index{$\bm{j}_f$} \partial_t \bE - \nabla_x \times \bB = -  \bm{j}_f: =  - \int v (p) J^{1/2} (p) f (p) \cdot \bm{\xi} \, dp,\\
 &
             \label{eq5.0.2}
	\partial_t \bB  + \nabla_x \times \bE = 0, \\
 &
              \label{eq5.0.3}
\index{$\rho_f$}	\nabla_x \cdot \bE = \rho_f: = \int J^{1/2}  f (p) \cdot \bm{\xi} \, dp, \quad \nabla_x \cdot \bB = 0,\\ 
	&
	             \label{eq5.0.4}
	   (\bE \times n_x)_{|\partial \Omega} = 0, \, \,   (\bB \cdot n_x)_{|\partial \Omega} = 0, \, \,   \bE (0, \cdot) = \bE_0 (\cdot), \, \,  \bB (0, \cdot) = \bB_0 (\cdot),
\end{align}
where   $\bm{\xi}$ and $\bm{\xi}_1$ are defined in \eqref{eq1.10}.
For the sake of convenience, we also call \eqref{eq5.0} - \eqref{eq5.0.4} the RVML system.

Before we state the definition of the  strong solution to the RVML system, we introduce the notions of finite energy  and  strong solutions to the linear relativistic Landau equation  
   \begin{align}
   			\label{eq1.0}
	&	Y f  - \nabla_p \cdot (\sigma_g \nabla_p f) + b \cdot \nabla_p f + (c + \lambda)  f = \eta, \\
\label{eq1.0.0}
 &f (t, x, p) = f (t, x, R_x p), \, z \in \Sigma^T_{-}, \, \, f (0, \cdot) = f_{0} (\cdot), 
  \end{align}
where
\begin{equation}    
            \label{eq6.0}
    \index{$\sigma_g$}    \sigma_g (t, x, p) = \int_{\bR^3} \Phi (P, Q) (2 J + J^{1/2} g (t, x, q))\, dq.
\end{equation}
\begin{remark}
The Landau equation \eqref{eq5.0} can be rewritten as \eqref{eq1.0} with
$g = f^{+}+f^{-}$ in $\eqref{eq6.0}$, $\lambda = 0$, and certain $b$, $c$,  $\eta$ depending on $f$. See the details in the proof of Proposition \ref{proposition 7.3} (cf. \eqref{eq7.48}).
\end{remark}

\begin{definition}[finite energy   solution]
    \label{definition 27.1}
    We say  that
    $$f \in C ([0, T]) L_2 (\Omega \times \bR^3) \cap L_2 ((0, T) \times \Omega) W^1_2 (\bR^3)$$ is a finite energy   solution to Eq. \eqref{eq1.0} - \eqref{eq1.0.0}
    if for any test function $\phi$ satisfying
        \begin{align}
   \label{eq27.2.1}
    & 
       \phi  \in L_2 ((0, T) \times \Omega) W^1_2 (\bR^3),   Y \phi \in  L_2 (\Sigma^T), \\
       \label{eq27.2.2}
           & \phi \in  C ([0, T]) L_2 (\Omega \times \bR^3), \\
      \label{eq27.2.3}
    &   \phi (t, x, p) = \phi (t, x, R_x p), \, (t, x, p) \in \Sigma^T_{-} \, \, \text{(in the trace sense)},
        \end{align}
and all $t \in [0, T]$, one has
\begin{align}
    \label{eq27.2}
    &   \int_{\Omega \times \bR^3} (f \phi) (t, x, p) - f_0 (x, p) \phi (0, x, p) \, dx dp
    - \int_{ (0, t)  \times \Omega \times \bR^3 }  f (Y \phi) \, dz\\
    &+ \int_{(0, t)  \times \Omega \times \bR^3} \bigg((\nabla_p \phi)^T \sigma_g \nabla_p f 
    + (b \cdot \nabla_p f) \phi + (c + \lambda) f \phi\bigg) \, dz \notag  \\
    &= \int_{(0, t)  \times \Omega }  \langle \eta (\tau, x, \cdot), \phi (\tau, x, \cdot) \rangle \notag \, dx d\tau.
\end{align}
where $\langle \cdot, \cdot \rangle$ is the duality pairing between $W^{-1}_2 (\bR^3)$  and $W^{1}_2 (\bR^3)$ (see \eqref{eq1.2.16}).

Furthermore, let $g$, $b, c$,  $f$, and $\eta$ be $t$-independent functions. Then,  we say that $f \in L_2 (\Omega) W^1_2 (\bR^3)$ is a finite energy  solution to the steady equation
\begin{equation}
\begin{aligned}
    \label{eq27.1.2}
    	&v (p) \cdot \nabla_x f  - \nabla_p \cdot (\sigma_g \nabla_p f) + b \cdot \nabla_p f + (c + \lambda)  f = \eta, \\
 &f (x, p) = f (x, R_x p), \, z \in \gamma_{-}, 
 \end{aligned}
\end{equation}
if for any test function $\phi =  \phi (x, p)$ satisfying the conditions analogous to \eqref{eq27.2.1} - \eqref{eq27.2.3}, the  `steady' counterpart of the identity \eqref{eq27.2} holds.  
\end{definition}

\begin{remark}
    \label{remark 27.7}
By Lemma \ref{lemma C.5},  if $\phi$  satisfies the conditions \eqref{eq27.2.1}, \eqref{eq27.2.3},  and $\phi_0 \in L_2 (\Omega \times \bR^3)$ in the trace sense (see Definition \ref{definition C.4}), then $\phi \in C ([0, T]) L_2 (\Omega \times \bR^3)$.
\end{remark}

\begin{definition}
    \label{definition 6.1}
    We say that $f \in S_2 (\Sigma^T)$ is a  strong solution to \eqref{eq1.0}-\eqref{eq1.0.0} if 
    \begin{itemize}
        \item the identity \eqref{eq1.0} holds in the $L_2 (\Sigma^T)$ sense, 
        \item the initial condition and the SRBC in \eqref{eq1.0.0} hold a.e. for the trace functions $f_0$ and $f_{\pm}$  (see Definition \ref{definition C.4}).
         \end{itemize}
        Similarly, we define a strong solution to the steady counterpart of \eqref{eq1.0}-\eqref{eq1.0.0}.
\end{definition}


\begin{remark}
    \label{remark 27.8}
By using the Green's identity in  \eqref{eqC.5.1}, one can show that any strong solution is also a finite energy solution.
Conversely, if $f$ is a finite energy solution such that $f \in S_2 (\Sigma^T)$, then $f$ is a strong solution.  
\end{remark}

\begin{definition}
                \label{definition 3.0}
We say that the VML system \eqref{eq5.0} - \eqref{eq5.0.4} has a  strong solution  $[f, \bE, \bB]$ on the time interval $[0, T]$ if
\begin{itemize}
    \item[--]  $f$ is a strong solution to the Landau equation  \eqref{eq5.0} (see Definition   \ref{definition 6.1}),
    \item[--]  $\bE, \bB \in C^1 \big([0, T], L_2 (\Omega))$,
  \item[--] for any $t \in [0, T]$,  $\bE (t, \cdot), \bB (t, \cdot) \in   W^1_2 (\Omega)$,  and 
  $(\bE (t, \cdot) \times n_x)_{|\partial \Omega} \equiv 0$, $(\bB (t, \cdot) \cdot n_x)_{|\partial \Omega} \equiv 0$,
\item[--] the identities \eqref{eq5.0.1} - \eqref{eq5.0.3} hold in the $L_2 ((0, T) \times \Omega)$ sense.
\end{itemize}
\end{definition}


\begin{assumption}
                                        \label{assumption 3.0}
The domain  $\Omega$ satisfies the following variant of the div-curl estimate.
  For any $r \in (1, \infty)$ and any $\bm{u} \in L_r(\Omega)$ such that 
  \begin{itemize}
     \item[--] $\nabla_x \times \bm{u} \in L_r (\Omega)$, $\nabla_x \cdot \bm{u} \in L_r (\Omega)$,
      \item[--] either $(\bm{u} \times n_x)_{|\partial \Omega} = 0$ or  $(\bm{u} \cdot n_x)_{|\partial \Omega} = 0$,
  \end{itemize}
one has $\bm{u} \in W^1_r (\Omega)$, and
\begin{equation}
                \label{eq3.0.0}
        \| \bm{u}\|_{ W^1_r (\Omega) } \le  N  \| |\nabla_x \times \bm{u}| + |\nabla_x \cdot \bm{u}|+|\bm{u}|\|_{ L_r (\Omega) },
\end{equation}
where $N = N (r, \Omega)$.
\end{assumption}

\begin{remark}
Loosely speaking,   if    $\Omega$  can be transformed into a simply connected domain of class $C^{1, 1}$ by removing a finite number of `cuts', then $\Omega$ satisfies Assumption \ref{assumption 3.0}. See Hypothesis 1.1 and Theorems 3.2 - 3.3 in \cite{AS_13}.  
We point out that a solid torus $B_1 \times S^1$  satisfies the aforementioned assumption since one needs to make a single cut to obtain a simply connected $C^{ 1, 1} $ domain.      
\end{remark}

\begin{remark}
A reader might be familiar with a variant of \eqref{eq3.0.0} where the right-hand side does not contain the term $\|\bm{u}\|_{ L_r (\Omega)}$ (see \cite{vW_92}). However, in the case when the boundary condition $(\bm{u} \cdot n_x)_{| \partial \Omega} = 0$ is imposed, such an estimate might be false if   $\Omega$ is not simply connected. We refer a reader to a beautiful counterexample in Section 9 of \cite{CDG_02}.
\end{remark}

We will construct the solution to the VML system via a Picard-type iteration argument.
 It turns out that to close such an argument, one needs to control the temporal derivatives up to order $m \ge 20$  of the particle density functions and the electromagnetic field (see \eqref{eq5.2} - \eqref{eq5.5}).  
To this end, one needs the initial data to satisfy certain regularity and compatibility conditions. Loosely speaking, those are the conditions on the temporal derivatives at $t = 0$. One can deduce the expression of such derivatives from the RVML system as follows.  Given 
$
    [\partial_t^k f (0, x, p), \partial_t^k \bE (0, x), \partial_t^k \bB (0, x)], 
$
we formally apply the operator $\partial_t^{k}$ to Eqs. \eqref{eq5.0}, \eqref{eq5.0.1}-\eqref{eq5.0.2}, plug $t = 0$, and solve for 
$
    [\partial_t^{k+1} f (0, x, p), \partial_t^{k+1} \bE (0, x), \partial_t^{k+1} \bB (0, x)].
$

\begin{definition} 
\label{definition 3.4}
\index{$[f_{0, k}, \bE_{0, k}, \bB_{0, k}]$} We set $[f_{0, 0}, \bE_{0, 0}, \bB_{0, 0}] = [f_0, \bE_0, \bB_0]$.
Furthermore, given $f_{0, j} (x, p)$, $\bE_{0, j} (x)$, $\bB_{0, j} (x), j = 0, \ldots, k$,    we set  
\begin{align}
  \label{eq3.3.2}
   &  f_{0, k+1} =   - v (p) \cdot \nabla_x f_{0, k} +  (A+K) f_{0, k}  + \bm{\xi}_1 (v (p) \cdot \bE_{0, k}) J^{1/2}\\
   &+ \sum_{j =0}^k \binom{k}{j} \bigg(-\bm{\xi} (\bE_{0, j} + v (p) \cdot \bB_{0, j}) \cdot \nabla_p  f_{0, k-j}  + \frac{\bm{\xi}_1}{2} (v (p) \cdot \bE_{0, j}) f_{0, k-j} + \Gamma (f_{0, j}, f_{0, k-j})\bigg), \notag\\
   \label{eq3.3.3}
    &  \bE_{0, k+1} (x) := \nabla_x  \times \bB_{0, k} (x) - \int_{\bR^3} v (p) J^{1/2} (p) f_{0, k} (x, p)  \cdot \bm{\xi} \, dp,\\
      \label{eq3.3.4}
    & \bB_{0, k+1} = - \nabla_x \times \bE_{0, k}.
\end{align}
\end{definition}

\begin{assumption}[Compatibility conditions]
                    \label{assumption 3.3}

We assume 
\begin{align}
  \label{eq3.3.7}
 & f_{0, k} \, \text{is a finite energy  solution to \eqref{eq3.3.2} with the SRBC}, \, k \le m-1,\\
  \label{eq3.3.7.1}
  &  f_{0, k} (x, p) = f_{0, k} (x, R_x p), \,  (x, p) \in \gamma_{-}, \, \text{(in the trace sense)} \, \, k \le m-8, \\
  \label{eq3.3.8}
  &  (\bE_{0, k} \times n_x)|_{\partial \Omega} \equiv 0, \, \, (\bB_{0, k} \cdot n_x)_{\partial \Omega} \equiv 0, \, k \le m-1,\\
          \label{eq3.10.1}
&\nabla \cdot \bB_{0, k} \equiv 0, \quad   \nabla \cdot \bE_{0, k} (x) = \int J^{1/2} (p)  f_{0, k} (x, p) \cdot \bm{\xi} \, dp, \,   k \le m-1,
\end{align}
where in  \eqref{eq3.3.7.1}, we implicitly assume that $f_{0, k}, \displaystyle \frac{p}{p_0} \cdot \nabla_x f_{0, k} \in L_2 (\Omega \times \bR^3)$, so that the trace is well defined.
\end{assumption}


\begin{remark}
        \label{remark 3.10}
Here, we show that for $k \ge 1$,  \eqref{eq3.10.1} can be derived formally from \eqref{eq3.10.1} with $k = 0$.  The first identity in \eqref{eq3.10.1} follows directly from \eqref{eq3.3.4}.
 Due to  \eqref{eq3.3.3}, to prove the second one,  it suffices to demonstrate that for $k =  0, \ldots, m$,
\begin{equation}
        \label{eq3.10.2}
      \int_{\bR^3}  J^{1/2} f_{0, k+1} \cdot \bm{\xi} \, dp   + \nabla_x \cdot \int_{\bR^3} v (p) J^{1/2} f_{0, k}  \cdot \bm{\xi} \, dp = 0.
\end{equation}
To this end, we denote $F_{0, k} = J + J^{1/2} f_{0, k}$ and note that the function  $F_{0, k+1}$ satisfies 
\begin{align}
\label{eq3.10.3}
&
F_{0, k+1}^{+}= - v (p) \cdot \nabla_x F_{0, k}^{+}\\
&+\sum_{j =0}^k \binom{k}{j} \bigg((\bE_{0, j} + v (p) \times \bB_{0, j}) \cdot \nabla_p F^{+}_{0, k-j} +\cC (F^{+}_{0, j}, F^{+}_{0, k-j}) + \cC (F^{+}_{0, j}, F^{-}_{0, k-j})\bigg)  = 0,\notag\\
\label{eq3.10.4}
&
	F^{-}_{0, k+1} = - v (p) \cdot \nabla_x F^{-}_{0, k}\\
& + \sum_{j =0}^k \binom{k}{j} \bigg(-(\bE_{0, j} + v (p) \times \bB_{0, j}) \cdot \nabla_p F^{-}_{0, k-j} + \cC (F^{-}_{0, j}, F^{-}_{0, k-j}) + \cC (F^{-}_{0, j}, F^{+}_{0, k-j})\bigg) = 0.\notag
\end{align}
The above identities can be derived by using the definition of $A, K, \Gamma$ (see \eqref{eq5}) and the fact that $\cC (J, J) = 0$.
One can also deduce  \eqref{eq3.10.3} - \eqref{eq3.10.4}   by differentiating formally the first two equations in \eqref{VML}, plugging  $t = 0$, and replacing $\partial_t^k f$ with $f_{0, k}$.
  Finally, subtracting \eqref{eq3.10.4} from \eqref{eq3.10.3}, integrating over $p \in \bR^3$, and using the definition of $\cC$ (see \eqref{eq1.1.2}),  we obtain \eqref{eq3.10.2}. 
\end{remark}

\begin{remark}
One can show  that  Assumption \ref{assumption 3.3} is satisfied if $f_0, \bE_0$, and $\bB_0$ are smooth compactly supported functions  away from $\partial \Omega$, $f_0$ decays sufficiently fast for large $p$, and \eqref{eq3.10.1} holds with $k = 0$.
\end{remark}

We introduce the key functionals that will be controlled in the proof of the local existence.  Let  $\theta, \tau > 0$ be numbers, and $f$ and  $[\bE_{  f }, \bB_{    f }]$ be  sufficiently regular functions on $\Sigma^{\tau}$ and $(0, \tau) \times \Omega$, respectively.

\textit{Instant energy functionals.} 
 We introduce the baseline instant energy
\begin{align}
\label{eq5.10}
   \index{$\cE_{||, f}$}     \cE_{||, f} (\tau) &=  \sum_{k=0}^{ m}  \bigg(\|\partial_t^k  f (\tau, \cdot) \|^2_{ L_2  (\Omega \times \bR^3) } +  \|\partial_t^k  [\bE, \bB] (\tau, \cdot)\|^2_{ L_2 (\Omega) }\bigg), \tau > 0,
\end{align}
and the   energy 
   \begin{align}
   \label{eq5.10.15}
        & \index{$\cE_f$} \cE_f (\tau)  =  \cE_{||, f} (\tau) +  \sum_{k=0}^{ m - 4 } \|\partial_t^k  f (\tau, \cdot)\|^2_{ L_{2, \theta/2^{k} } (\Omega \times \bR^3) }.
\end{align}

\textit{Higher regularity instant functional.} 
Let $\Delta r \in (0, \frac{1}{42})$   and denote  
\begin{align}
\label{eq5.10.10}
& \index{$r_i$} r_1 = 2, \quad     \frac{1}{r_{i}} = \frac{1}{r_{i-1}} -  \big(\frac{1}{6} - \Delta r\big), i = 2, 3, 4,  \\
\label{eq5.10.2}
& r_2 \in (2, 3), \quad r_3 \in (3, 6), \quad r_4 > 14,\\
\label{eq5.10.1}
& \index{$\mathcal{H}_f$} \mathcal{H}_f (\tau) =  \sum_{i=1}^4 \sum_{k=0}^{  m-4-i } \|\partial_t^k f (\tau, \cdot)\|^2_{  S_{r_i, \theta/2^{ k + 2i   } } (\Omega \times \bR^3) }  \\
& + \sum_{k=0}^{m-1} \|\partial_t^k  [\bE_f, \bB_f] (\tau, \cdot)\|^2_{ W^1_{ 2 } (\Omega) } 
+ \sum_{i=2}^3 \sum_{k=0}^{ m  - 4 -  i} \|\partial_t^k  [\bE_f, \bB_f] (\tau, \cdot)\|^2_{ W^1_{r_i} (\Omega) }. \notag
\end{align}

\textit{Total instant functional.} The total instant functional  is the sum of the total instant energy and the higher regularity  functional:
   \begin{align}
    \label{eq5.10.17}
        & \index{$\cI_f$} \cI_f (\tau)  =  \cE_{f} (\tau)   + \mathcal{H}_{f} (\tau), \tau > 0, \\
      \label{eq5.10.17.1}
        & \cI_f (0)  \, \text{is given by} \, \eqref{eq5.10.17} \, \text{with} 
        \,\partial_t^k [f, \bE_f, \bB_f] (\tau, \cdot) \, \text{replaced with} 
       \, [f_{0, k}, \bE_{0, k}, \bB_{0, k}].
        \end{align}

\textit{Dissipation functionals}.
The baseline dissipation is defined by 
\begin{align}
\index{$\cD_{||, f}$}   \cD_{||, f} (\tau) & =  \sum_{k=0}^{ m}  \|\partial_t^k  f (\tau, \cdot) \|^2_{ L_2 (\Omega) W^1_{2 } (\bR^3) }
\end{align}
and the total dissipation is 
\begin{align}
    \label{eq5.10.16}
      \index{$\cD_{f}$}  \cD_{f} (\tau) & = \cD_{||, f} (\tau) +   \sum_{k=0}^{ m - 4 }  \|\partial_t^k  f (\tau, \cdot) \|^2_{ L_2 (\Omega) W^1_{2, \theta/2^{k} } (\bR^3) }.
\end{align}

\textit{Total functional.}
\begin{equation}
    \label{eq5.10.4}
 \index{$y_f$}   y_f (\tau) =   \sup_{t \le \tau} \cI_f (\tau) + \int_0^{\tau} \cD_f (t) \, dt.
\end{equation}

Here is the main result of the present paper.
\begin{theorem}
            \label{theorem 5.1}
Let $m \ge 20$ be an integer,  $r_1, \ldots, r_4,$ be numbers satisfying \eqref{eq5.10.10} - \eqref{eq5.10.2}, and $\Omega$ be a $C^{1, 1}$ bounded domain satisfying Assumption \ref{assumption 3.0}.
Then, there exists a constant $\theta_0 = \theta_0 (r_1, \ldots, r_4) > 1$ such that for any $\theta \ge \theta_0$ there exist constants
\begin{equation}
            \label{5.1.0}
\begin{aligned}
  &   M = M (r_1, \ldots, r_4, m, \theta, \Omega) >  1, \\
  & \varepsilon_0 = \varepsilon_0 (\theta, r_1, \ldots, r_4, m,  \Omega) \in (0, 1), 
  \quad T  = T (\theta, r_1, \ldots, r_4, m, \Omega) \in (0, 1)
\end{aligned}
\end{equation}
    such that if $\cI_f (0) < \infty$ (see \eqref{eq5.10.17.1}), and
\begin{itemize}
\item[--] 
\begin{align}
            \label{eq5.1.1}
             & \cE_f (0) :=   \sum_{k = 0}^m \bigg(\|[\bE_{0, k}, \bB_{0, k}]\|^2_{ L_2 (\Omega) } + \| f_{0, k}\|^2_{ L_2 (\Omega \times \bR^3) }\bigg) 
  +  \sum_{k=0}^{m-4}  \|f_{0, k}\|^2_{ L_{2, \theta/2^k} (\Omega \times \bR^3) }
  \le  \varepsilon_0/M
\end{align}  
(see \eqref{eq3.3.2} - \eqref{eq3.3.4}),
\item[--] the compatibility conditions \eqref{eq3.3.7} - \eqref{eq3.10.1} in  Assumption  \ref{assumption 3.3}  hold,
\end{itemize} 
then the following assertions hold.

$(i)$  RVML system \eqref{eq5.0} - \eqref{eq5.0.4} has a strong solution $[f, \bE_f, \bB_f]$  (see Definition \ref{definition 3.0}) on $\Sigma^{T}$ such that
 \begin{equation}
                \label{eq5.1.4}
    y_f (T) < \varepsilon_0 \, \text{(see \eqref{eq5.10.4})}. 
 \end{equation}

    $(ii)$ For $k \le m$, $\partial_t^k f$ is a finite energy   solution (see Definition \ref{definition 27.1})  to Eq. \eqref{eq5.0} differentiated $k$ times in $t$ with the initial data $\partial_t^k f (0, \cdot) \equiv f_{0, k} (\cdot)$ and SRBC.
   
     $(iii)$ For $k \le m-1$, $\partial_t^k [\bE_f, \bB_f] \in C ([0, T]) L_2 (\Omega) \cap L_{\infty} ((0, T)) W^1_2 (\Omega)$ is a  strong solution to Maxwell's equations \eqref{eq5.0.1} - \eqref{eq5.0.2} differentiated $k$ times with the initial data $[\bE_{0, k}, \bB_{0, k}]$ and the perfect conductor BC, whereas $\partial_t^m [\bE_f, \bB_f] \in C ([0, T]) L_2 (\Omega)$ is a  weak solution (see \cite{DL_76}).
     
     $(iv)$ The identities $\nabla_x \cdot \partial_t^k \bE_f  = \partial_t^k  \rho_f$, $\nabla_x \cdot \partial_t^k  \bB_f  = 0$ hold thanks to the compatibility conditions \eqref{eq3.10.1} and the continuity equation
     $$
        \partial_t (\partial_t^k \rho_f) + \nabla_x \cdot \partial_t^k \bm{j}_f = 0, k \le m
     $$
     (see \eqref{eq5.0.1}, \eqref{eq5.0.3}).

$(v)$ In addition, if $[f_i, \bE_{f_i}, \bB_{f_i}], i = 1, 2,$ are   strong solutions to the RVML system on $\Sigma^{T}$ satisfying the bound \eqref{eq5.1.4}, then, we have $f_1 = f_2$ on $\Sigma^{T}$ and $\bE_{f_1} = \bE_{f_2}$, $\bB_{f_1} = \bB_{f_2}$ on $(0, T) \times \Omega$. 
\end{theorem}

\subsection{Iteration scheme.}
                    \label{subsection 3.1}
To prove the existence, we set up an iteration scheme (cf.  \cite{GS_03}).
Let $[f^0, \bE^0, \bB^0]   = [f_{0}, \bE_{0}, \bB_{0}]$, and, given $[f^{n}, \bE^{n}$, $\bB^{n}]$, we set $[f^{n+1}, \bE^{n+1}$, $\bB^{n+1}]$ to be the strong solution to the following linear system   (see Proposition \ref{proposition 7.2}):
\begin{align}
    \label{eq5.1}
	&Y f^{n+1}+ \bm{\xi} (\bE^n + v (p) \times \bB^n) \cdot \nabla_p f^{n+1} -  \frac{\bm{\xi}}{2}  (v (p) \cdot \bE^n) f^{n+1}
+  L f^{n+1}  \\
	&=   \bm{\xi}_1 (v (p) \cdot  \bE^{n+1})  J^{1/2}  +   \Gamma (f^{n+1}, f^n), \notag\\
  \label{eq5.1'}	
 & f^{n+1} (t, x, p) = f^{n+1} (t, x, R_x p), \, \, z \in \Sigma^T_{-}, \quad  f^{n+1} (0, \cdot) \equiv f_{0, 0},\\
	& \label{eq5.1.0*}
	\partial_t \bE^{n+1} - \nabla_x \times \bB^{n+1} = - \int v (p) J^{1/2} (p)   f^{n+1} (p) \cdot \bm{\xi} \, dp,\\
    \label{eq5.1.0}
 &
	\partial_t \bB^{n+1}  + \nabla_x \times \bE^{n+1} = 0, \\
     \label{eq5.1.0.0}
 &
	\nabla_x \cdot \bE^{n+1} = \int J^{1/2}   f^{n+1} (p) \cdot \bm{\xi} \, dp, \quad \nabla_x \cdot \bB^{n+1} = 0,\\ 
	     \label{eq5.1.0.0.0}
	   & (\bE^{n+1} \times n_x)_{|\partial \Omega} = 0, \quad (\mathbf{B}^{n+1} \cdot n_x)_{|\partial \Omega} = 0,\\
\label{eq5.1.0.1}
    & \bE^{n+1} (0, \cdot) \equiv \bE_{0, 0} (\cdot),  \quad  \bB^{n+1} (0, \cdot) \equiv \bB_{0, 0} (\cdot),
 \end{align}
where $L= -A-K$ is the linearized Landau operator (see \eqref{eq5}).

 Setting $f: = f^{n+1}$, $g: = f^n$, $\mathbf{E}_f := \mathbf{E}^{n+1}$, $\mathbf{E}_g := \mathbf{E}^n$, $\bB_f = \mathbf{B}^{n+1}$, $\mathbf{B}_g := \mathbf{B}^n$ gives
\begin{align}
    \label{eq5.2}
	&Y f + \bm{\xi} (\mathbf{E}_g + v (p)   \times \mathbf{B}_g) \cdot \nabla_p f -  \frac{\bm{\xi}}{2} (v (p) \cdot  \mathbf{E}_g)  f
+ L f  \\
	&=  \bm{\xi}_1 (v (p) \cdot  \bE_f)  J^{1/2} + \Gamma (f, g),\notag\\
		& f (0, \cdot) = f_{0, 0}, \, \, f (t, x, p) = f (t, x, R_x p), \, \, z \in \Sigma^T_{-}, \notag\\
	    \label{eq5.2.0}
	&\partial_t \mathbf{E}_f - \nabla_x \times \mathbf{B}_f = - \int v  (p) J^{1/2} (p)   f (p) \cdot \bm{\xi} \, dp,\\
 &
 \label{eq5.3}
	\partial_t \mathbf{B}_f  + \nabla_x \times \mathbf{E}_f = 0,  \\
  \label{eq5.4}
 &	\nabla_x \cdot \mathbf{E}_f = \rho_f = \int J^{1/2} (p)  f (p) \cdot \bm{\xi} \, dp, \quad  \nabla_x \cdot \mathbf{B}_f = 0, \\ 
	 \label{eq5.5}
	& (\mathbf{E}_f \times n_x)_{|\partial \Omega} = 0, \, \,  
	   (\mathbf{B}_f \cdot n_x)_{|\partial \Omega} = 0, \, \, \bE_{f} (0, \cdot) \equiv \bE_0 (\cdot), \, \, 
    \bB_{f} (0, \cdot) \equiv \bB_0 (\cdot).
 \end{align}

The following proposition is the crux of the proof of the existence part in Theorem \ref{theorem 5.1}.

\begin{proposition}[propagation of smallness]
                \label{proposition 5.2}
Invoke the assumptions of Theorem \ref{theorem 5.1} and let $\theta, M, \varepsilon_0$, and $T$   be the constants as in \eqref{5.1.0}. We assume that 
there exist a triple $g = (g^{+}, g^{-}), \bE_g, \bB_g$ such that
for each $k \in \{0, \ldots, m\}$,  
    \begin{align}
   \label{eq5.2.12.1}
   &  \partial_t^k g \in C ([0, T]) L_2 (\Omega \times \bR^3), \, \,   \partial_t^k [\bE_g, \bB_g] \in C ([0, T]) L_2 (\Omega), \, k \le m,\\
& \label{eq5.2.16}
 g (t, x, p) =   g (t, x, R_x p), \,  (t, x, p) \in \Sigma^T_{-},  \\
\label{eq5.2.19}
&
    \partial_t^k g (0, \cdot) = f_{0, k} (\cdot) \, (\text{see} \, \eqref{eq3.3.2}), \\
    \label{eq5.2.21}
    &\partial_t^k  [\bE_g, \bB_g] (0, \cdot) = [\bE_{0, k} \bB_{0, k}] (\cdot) \, \,  (\text{see} \, \eqref{eq3.3.3} - \eqref{eq3.3.4}).
\end{align}
Then, if $\theta$, and $M$ are sufficiently large and $\varepsilon_0$ is sufficiently small, and
\begin{equation}
                \label{eq5.2.4.2}
    y_g (T) < \varepsilon_0, 
    \quad \text{\eqref{eq5.1.1} holds},
\end{equation}
  then, the 
  linear RVML system \eqref{eq5.2} - \eqref{eq5.5} with the initial conditions $[f_{0}, \bE_{0}, \bB_{0}]$ has a unique  strong  solution $[f, \bE_f, \bB_f]$ (see Definition \ref{definition 3.0}).
 Furthermore, we have 
\begin{equation}
                \label{eq5.2.4.1}
y_f (T) < \varepsilon_{ 0 }, 
\end{equation}
and in addition, 
the assertions analogous to $(ii) - (v)$ hold for $\partial_t^k [f, \bE_f, \bB_f]$. Moreover,
the conditions \eqref{eq5.2.12.1} - \eqref{eq5.2.21} hold with $[g, \bE_g, \bB_g]$ replaced with $[f, \bE_f, \bB_f]$.
\end{proposition}

\section{Method of the proof and organization of the paper}
                    \label{section 4}
                    
The goal of this section is to highlight the key difficulties and to delineate the main ideas in the proof of Theorems \ref{theorem 5.1} and Proposition \ref{proposition 5.2}. For the sake of clarity,  we will omit some technical details. 

\subsection{Unique solvability and the velocity Hessian estimates for the linear Landau equation}
First, we need to show that for each $n$, the iteration scheme \eqref{eq5.1} - \eqref{eq5.1.0.1} is well-posed. We will focus on the case when  $n = 0$ and will only consider the equation \eqref{eq5.2} with   $g = f_0$. We want to show that it has a unique strong solution $f$ (see Definition \ref{definition 6.1}), under the assumption that  $f_0$, $\bE_0$, and $\bB_0$ are sufficiently regular functions. In addition, our argument will enable us to deduce that $f$ and $\nabla_p f$ are bounded functions, which is important for proving both the existence and the uniqueness parts of Theorem \ref{theorem 5.1}.

\textit{Uniqueness and $S_2$ regularity.} The basic difficulty in establishing the uniqueness of the boundary value problems for the  velocity diffusive kinetic equations lies in the fact that for the natural solution class
$$
    f \in L_2 ((0, T) \times \Omega) W^1_2 (\bR^3), \, \,  Y f \in L_2 ((0, T) \times \Omega) W^{-1}_2  (\bR^3),
$$
it is unknown if the traces are well defined and if the energy identity for the transport operator $Y$ holds.
On the other hand, if $f, Y f \in L_2 (\Sigma^T)$, then the traces are well defined (see Section \ref{section 9}), and if, additionally, $f$ satisfies the SRBC, then a variant of the energy identity does hold (see Lemma \ref{lemma C.5}).
To summarize, we first construct a solution to \eqref{eq5.2} in the natural energy class. We show the  uniqueness by establishing the $S_2$ regularity, i.e.,
\begin{equation}
                \label{eq10.1}
    Y f, D^2_p f \in L_2 (\Sigma^T).
\end{equation}

\textit{Mirror-extension argument and $S_2$ regularity.} To prove \eqref{eq10.1}, we use an extension argument, which first appeared in \cite{GHJO_20} and was later used in the studies of the Vlasov-Poisson-Landau \cite{DGO_22} and linear Landau equations \cite{DGY_21}.
First, we localize in the spatial and momentum variables by deriving an equation for $f$ multiplied by a suitable cutoff function (see \eqref{eq2.36} - \eqref{eq2.60}).
 By using a flattening and extension argument (see the proof of Lemma \ref{lemma 7.4}), near the boundary, we reduce Eq. \eqref{eq5.2}  to a parabolic PDE  on the whole space \eqref{eq2.5.0} with discontinuous drift coefficients $\mathcal{X}$ (see  \eqref{eq2.0.1}, \eqref{eq2.11.1}). We point out that such a drift term is absent when the boundary is flat. Thus, one needs to use a Calderon-Zygmund type result to obtain \eqref{eq10.1}.  However, in contrast to \cite{GHJO_20}, \cite{DGO_22},   \cite{DGY_21}, our new equation \eqref{eq2.5.0} is quite different from the Newtonian kinetic Fokker-Planck (KFP) type equation (see \eqref{eq10.3}) since the coefficient in the transport term depends on both spatial and momentum variables. We are not aware of any Calderon-Zygmund-type result for such an equation.
To overcome this difficulty, we make a change of variables in the momentum variable, which enables us to reduce Eq. \eqref{eq2.5.0} to a Newtonian KFP  equation on the whole space (see \eqref{eq2.44}). Finally, we use  the $S_2^{N}$ (see \eqref{eq1.2.4.0}) estimate of \cite{DY_21a} to deduce \eqref{eq10.1}.
We remark that for other prominent boundary conditions in kinetic theory, e.g., inflow and diffuse boundary conditions, such an extension argument does not work.

\textit{Higher regularity}. Near the spatial boundary, we work with the Newtonian KFP equation  \eqref{eq2.44}, which we derived from Eq \eqref{eq5.2}. By using the Sobolev embedding theorem for $S_r^{N}$ spaces (see \cite{PR_98}) and the $S_r^{N}$ regularity theory developed in \cite{DY_21a}, we conclude that
$$
    f, \nabla_p f \in L_{\infty} (\Sigma^T).
$$

\textit{$S_r^{N}$ theory on the whole space for a Newtonian KFP equation with rough coefficients}. Here, we want to highlight one of the main ingredients of the present paper, that is, the Calderon-Zygmund theory for nonrelativistic KFP equations established in \cite{DY_21a}. We explain the importance of this theory by considering the equation 
$$
\partial_t f + p \cdot \nabla_x f - \Delta_p f = \eta
$$
in $\Sigma^T$ with the initial condition $f_0 \equiv 0$ and the  SRBC. Near the boundary, one can use a flattening and an extension argument as in \cite{GHJO_20}, \cite{DGO_22},   \cite{DGY_21} to derive the following equation for the   ``mirror extension" $\overline{f}$ on $(0, T) \times \bR^3_y \times \bR^3_w$ (see Section 2.1 in \cite{DGY_21}):
$$
     \partial_t \overline{f} + w \cdot \nabla_y \overline{f} - a^{i j} (y)  \partial_{w_i w_j}  \overline{f} - \nabla_w \cdot(X \overline{f}) = \overline{\eta} \,\,  \text{in} \, \,  \bR^7_T,
$$
where $X$ is the `geometric' term which is quadratic $w$, depends on the curvature of $\Omega$, and is discontinuous across the the hyperplane $\{y_3 = 0\} \times \bR^3_w$. Before the work \cite{DY_21a}, the unique solvability in the class of strong solutions and the global $L_r$ estimate of $D^2_w \overline{f}, (\partial_t + w \cdot \nabla_y) \overline{f}$ was unknown for the equation
\begin{equation}
                \label{eq10.3}
\begin{aligned}                
    & \partial_t u + w \cdot \nabla_y u - a^{i j} (t, y, w)  \partial_{w_i w_j} u \\
    &+ b^i \partial_{w_i} u + c u = \eta\,\,  \text{in} \, \,  \bR^7_T, \quad u (0, \cdot) \equiv 0,
\end{aligned}     
\end{equation}
with $a \in L_{\infty} ((0, T)) C^{\alpha/3, \alpha}_{y, w} (\bR^{2d})$.  
In particular, in the papers \cite{BCLP_13} and \cite{NZ_20}, the authors  imposed the uniform continuity assumption with respect to the following `kinetic distance':
$$
    d_{\text{kin}} ((t, y, w), (t', y', w')) = \max\{|t-t'|^{1/2}, |y-y'-(t-t')w'|^{1/3}, |w-w'|\}.
$$
It is easy to see that even in dimension $1$, the function $a^{i j} (t, y, w) = 2 + \sin (y)$ is not uniformly continuous on $\bR^7$ with respect to $d_{\text{kin}}$.
In contrast, the theory developed in \cite{DY_21a} covers Eq. \eqref{eq10.3}  with more general leading coefficients $a^{i j}$ including 
the ones satisfying the uniform continuity with respect to the metric $d ((y, w), (y', w')) = |y-y'|^{1/3}+|w-w'|$  uniformly in time.
As we mentioned above, to show the uniqueness and higher regularity of Eq. \eqref{eq5.2}, we reduce it to the one of the form \eqref{eq10.3} (see \eqref{eq2.44}). It turns out that for such an equation, 
$$
    a^{i j} \in L_{\infty} ((0, T)) C^{\alpha/3, \alpha}_{y, w} (\bR^6), 
$$
and, hence,  the results of \cite{DY_21a} are applicable. 

\subsection{Propagation of smallness} Here, we highlight the main difficulties in the proof of Proposition \ref{proposition 5.2} and describe the key features of the argument. 

\textit{Issue 1: $L_{\infty}$ bound of the electromagnetic field.} 
As usual, to control the `cubic' terms involving $[\bE_f, \bB_f]$ in the energy argument, we need an $L_{\infty}$ bound $[\bE_f, \bB_f]$. However, in contrast to the Vlasov-Poisson-Landau equation, we do not expect an `instant' regularization of the electromagnetic field due to the hyperbolic nature of Maxwell's equations.
 
 To overcome this issue, we rewrite the Maxwell system  into two div-curl systems
\begin{align}
        \label{eq10.4}
&\begin{cases}
	 \nabla_x \times  \bE_f  = - \partial_t \bB_f, \\
	 \nabla_x \cdot  \bE_f =  \int J^{1/2} (p)  f (p)    \cdot \bm{\xi}_1 \, dp,\\
  (\bE_f \times n_x)|_{\partial \Omega} = 0, 
 \end{cases}\\
& \label{eq10.5}
\begin{cases}
 	\nabla_x \times   \bB_f = \partial_t \bE_f  + \int v (p) J^{1/2} (p)    f (p)  \cdot \bm{\xi}_1  \, dp,\\
  \nabla_x  \cdot \bB_f  = 0, \\
 (\bB_f \cdot n_x)|_{\partial \Omega} = 0.
\end{cases}
\end{align}
If we have the bound of the $L_{ \infty } ((0, T)) L_2 (\Omega)$ norm of $\partial_t [\bE_f, \bB_f]$, then, by using the div-curl estimate \eqref{eq3.0.0} with $r = 2$, we can bound the $L_{   \infty } ((0, T)) W^1_2 (\Omega)$ norm of $[\bE_f, \bB_f]$. This   yields an $L_{ \infty } ((0, T)) L_6 (\Omega)$ estimate of $[\bE_f, \bB_f]$ due to the Sobolev embedding 
$W^1_2 (\Omega) \subset L_{6} (\Omega)$. To achieve this, we differentiate Maxwell's equations with respect to $t$ and use the div-curl estimate.  It is clear now that to close the iteration argument, we need to control certain norms for $\partial_t^k [\bE_f, \bB_f]$ and $\partial_t^k f$ for $k \le m$ for some $m$.

\textit{Issue 2: existence of the temporal derivatives of $[f, \bE_f, \bB_f]$.}
One needs to justify that the higher-order temporal derivatives $\partial_t^{k} [f, \bE_f, \bB_f], k \le m$  are sufficiently regular functions. Let us consider the case when $k = 1$. By differentiating formally the  linear VML system \eqref{eq5.2} - \eqref{eq5.5}, we write down the initial boundary value problem for  $\partial_t [f, \bE_f, \bB_f]$  with the initial data   $[f_{0, 1}, \bE_{0, 1}, \bB_{0, 1}]$ defined in \eqref{eq3.3.2} -   \eqref{eq3.3.4}.
 We then use the well-posedness theory for the   KFP equation in the finite energy solution class, developed in Sections  \ref{section 6}  of the present paper, and the well-known results for Maxwell's equations (see \cite{DL_76}). See the details in Proposition \ref{proposition 7.2} and Section \ref{section 7.2}.
 To apply these theories, one needs to impose certain regularity and compatibility conditions on the `initial data'  $[f_{0, 1}, \bE_{0, 1}, \bB_{0, 1}]$  (see \eqref{eq3.3.7}-\eqref{eq3.10.1}).

\textit{The scheme.}  The basic structure of the argument is similar to that in \cite{DGO_22}.
The primary functional that needs to be controlled throughout the iteration argument is the energy norm, while the $L_{\infty} ((0, T)) S_r (\Omega \times \bR^3)$ (see \eqref{eq1.2.7.0} - \eqref{eq1.2.7}) estimates are needed for establishing higher regularity bounds of the lower-order  $t$-derivatives for the closure of the energy argument.
Here, we explain the main steps of the argument and the motivation for designing the functional $y_f$ (see \eqref{eq5.10.4}).
\begin{itemize}
    \item First, we derive the energy bound of $\partial_t^k f, k \le m$    by using the estimates of the terms $A, K$ and $\Gamma (f, g)$ established in \cite{GS_03}. As usual, to close such estimates, one needs to control the $L_{\infty}$-norms of $\partial_t^k [f, \bE_f, \bB_f], k \le m/2$. 
    
    \item  To gain the $L_{\infty}$ regularity of the lower-order derivatives of the electromagnetic field via the $W^1_r$ div-curl estimate, one needs to descend from the top-order temporal derivatives to lower-order ones.  As we descend, the electromagnetic field gains integrability. 
    
    \item To prove the $L_{\infty}$ estimates of $\partial_t^k f, k \le m/2$, we use the weighted  $L_{\infty}^t S_r (\Omega \times \bR^3)$ estimate.  Due to the presence of the term $(v (p) \cdot \bE_f) \sqrt{J} \bm{\xi}_1$ in the linear Landau equation \eqref{eq5.2}  and the loss of $t$-derivatives in the higher regularity estimate of $\partial_t^k [\bE_f, \bB_f]$,   we combined  the $L_{\infty}^t S_r$ estimate with a descend argument. 

    \item The specific gap between the unweighted energy control (up to  $k \le m$)  and the weighted energy one (up to  $k \le m-4$) is motivated by the global estimate for the RVML system established in the subsequent paper \cite{GWP_23}.
\end{itemize}

\subsection{Organization of the paper}
The rest of the paper is organized as follows. In Section \ref{section 6}, we 
establish the results concerning the existence, uniqueness, and higher regularity of a strong solution to the linear Landau equations  \eqref{eq1.0} - \eqref{eq1.0.0} and \eqref{eq27.1.2}. Additionally, in this section, we present the well-posedness result for the KFP equation in the class of finite energy solutions. In Sections \ref{section 7} - \ref{section 8}, we prove Proposition \ref{proposition 5.2} and Theorem \ref{theorem 5.1}, respectively. In Appendix \ref{Appendix A} - \ref{section 7.2}, we collect various auxiliary results.

\section{Regularity theory of the linear relativistic Landau equation with the specular reflection boundary condition}
            \label{section 6} 
The purpose of this section is to present the results on the unique solvability and certain estimates for the linear equation \eqref{eq1.0} - \eqref{eq1.0.0} and its `steady' counterpart \eqref{eq27.1.2}. 
This section is organized as follows. First, in Section \ref{section 6.1},  we present the results on the strong solutions (see Definition \ref{definition 6.1}) to the unsteady linear Landau equations \eqref{eq1.0} - \eqref{eq1.0.0}. Second, in Section \ref{section 6.2}, we establish the `steady' counterparts of the aforementioned results. Finally, in Section \ref{section 6.3}, we prove the well-posedness of the unsteady linear Landau equation in the class of finite energy solutions (see Definition \ref{definition 27.1}).



\subsection{Strong solutions to the unsteady Landau equations}

\label{section 6.1}

\begin{assumption}
			\label{assumption 1.4}
			There exist  $\varkappa \in (0, 1]$ and $K > 0$ such that
\begin{align}
&\label{eq1.4.1}	\|g\|_{   L_{\infty} ((0, T))  C^{\varkappa/3, \varkappa}_{x, p} (\Omega \times \bR^3)   } \le K, \\
& \label{eq1.4.2}
 \|\nabla_p g\|_{   L_{\infty} (\Sigma^T)   } \le K.
\end{align}

\end{assumption}

\begin{assumption}
			\label{assumption 1.5}
For a.e. $(t, x, p) \in \Sigma^T_{-}$,
\begin{equation}
    \label{eq1.5.1}
	g (t, x, p) = g (t, x, R_x p). 
\end{equation}
\end{assumption}



We will use the fact that if   $g$ is sufficiently small (near Maxwellian regime), then the leading coefficients $\sigma_g$ are uniformly nondegenerate.  
\begin{lemma}
        \label{lemma 6.1}
There exists $\varepsilon_{\star} > 0$ and $\delta_0 \in (0, 1)$   such that if  for some $T > 0$, one has 
\begin{equation}
    	\label{eq1'}
\|g\|_{ L_{\infty} (\Sigma^T) }\le \varepsilon_{\star},
\end{equation} 
then, 
\begin{equation}
        \label{eq6.1.2}
    \sigma_g (z) \xi_i \xi_j \ge \delta_0 |\xi|^2, \, \,  z \in \Sigma^T, \xi \in \bR^3, 
    \quad \|[\sigma_g, \nabla_p \sigma_g]\|_{ L_{\infty} (\Sigma^T)  } \le \delta_0^{-1}.
\end{equation}
The constants $\varepsilon_{\star} > 0$ and $\delta_0 \in (0, 1)$ are  independent of $T$.
\end{lemma}

\begin{proposition}
    [unique solvability in  weighted $S_2$ spaces]
			\label{proposition 2.1}
Let 
\begin{itemize}[--]

\item $\lambda \ge 0$, $\kappa \in (0, 1)$, $\varkappa \in (0, 1], K > 0$ be  numbers,

\item $\Omega$ be a $C^{  1, 1 }$ bounded domain,

\item[--]  $b  =  (b^1, b^2, b^3)^T$ and $c$ be bounded measurable functions on $\bR^7_T$ such that for some $K > 0$, 
\begin{equation}
			\label{eq1.2.2}
	\|b\|_{ L_{\infty} (\Sigma^T) } + \|c\|_{ L_{\infty} (\Sigma^T) }  \le K,
\end{equation}

\item $g$ satisfy \eqref{eq1.4.1} -  \eqref{eq1.5.1}, 

\item the condition \eqref{eq1'} hold.
\end{itemize}
Then, there exists  $\theta  = \theta (\kappa, \varkappa) > 0$ 
 such that if
\begin{align}
			\label{eq2'}
&	\eta \in L_{2, \theta} (\Sigma^T), \quad f_0 \in S_{2, \theta} (\Omega \times \bR^3),
\end{align}
then, the following assertions hold.

$(i)$
There exists a unique  strong solution $f$  to  \eqref{eq1.0} -  \eqref{eq1.0.0} (see Definition \ref{definition 6.1}).

$(ii)$ We have $f \in S_{2,  \kappa \theta } (\Sigma^T)$, and, furthermore,  
   \begin{align} 
    \label{eq2.1.20}
   & \|f\|_{ L_2 ((0, T) \times \Omega) W^1_{2, \theta} (\bR^3) } \\
  & \le N_{ 1  }  \|\eta\|_{ L_{2, \theta} (\Sigma^T) } + N_{ 1  } \|f_0\|_{ L_{2, \theta} (\Omega \times \bR^3)}, \notag \\
    			\label{eq2}
	& \|f\|_{ S_{2,  \kappa \theta } (\Sigma^T)  }  + \|f\|_{ L_{14/5, \kappa \theta}  (\Sigma^T)  }  + \| \nabla_p f\|_{ L_{7/3, \kappa \theta}  (\Sigma^T)  }\\
 & \le N_{ 2  } \|\eta\|_{  L_{2, \theta}   (\Sigma^T) } + N_{ 2  } (1+ \lambda) \|f_0\|_{ S_{2,    \theta } (\Omega \times \bR^3)  } + N_2 \|f\|_{ L_{2, \theta} (\Sigma^T) }, \notag
 \end{align}
where $N_1 = N_1 (K, \delta_0, T)$,  $N_{ 2  } = N_{ 2  } (\delta_0,  \varkappa, \kappa, K, \theta,  \Omega) > 0$.
\end{proposition}

\begin{remark}
    \label{remark 2.8}
Invoke the assumptions of Proposition \ref{proposition 2.1} and let $f$ be the  strong solution to \eqref{eq1.0} - \eqref{eq1.0.0}. Then, $f$ satisfies the mirror-extension property, which is defined (imprecisely) below. We will make this statement precise in the proof of the present remark (see p. \pageref{eq2.8.10}).

Let $\xi_n, n \ge 1$ be a dyadic partition of unity in $\bR^3$ and $\chi_k, k  =1, \ldots, m$ be a partition of unity in $\Omega$.
    A  strong solution $f$ satisfies the mirror-extension property if, near the boundary, $f_{k, n}: = f \chi_k \xi_n$ can be `extended' to a function $\mathcal{\dtilde U}$ satisfying the identity
        \begin{align*}                
    & \partial_t  \, \mathcal{\dtilde U} + v \cdot \nabla_y \, \mathcal{\dtilde U} - \nabla_v \cdot (a (t, y, v)  \nabla_v \,  \mathcal{\dtilde U}) \\
    &+ b^i \partial_{v_i} \mathcal{\dtilde U} + c \,  \mathcal{\dtilde U} = \psi\,\,  \text{in} \, \,  \bR^7_T,
\end{align*}     
for certain $a, b, c$, and $\psi$,
      which are `under control.'
\end{remark}

\begin{proposition}[Higher  regularity of a  strong solution]
			\label{proposition 2.3}
Invoke the assumptions of Proposition \ref{proposition 2.1} and let  $r > 2$ be a number. 
Then, there exists a constant
$\theta = \theta (\kappa, \varkappa, r) > 0$
such that if, additionally,
$$
	\eta \in L_{2, \theta} (\Sigma^T)  \cap   L_{r, \theta} (\Sigma^T),
 \quad f_0 \in S_{2, \theta} (\Omega \times \bR^3)  \cap   S_{r, \theta} (\Omega \times \bR^3),
$$
then, for any  strong solution to Eq. \eqref{eq1.0} - \eqref{eq1.0.0}  one has
\begin{align}
\label{eq2.4.3}
&f \in S_{2, \kappa \theta} (\Sigma^T) \cap S_{r, \kappa \theta} (\Sigma^T), \\
			\label{eq2.3.1}
& \|f\|_{ S_{r,  \kappa \theta } (\Sigma^T)  }   \le   N  \sum_{ s \in \{2, r\} } 
\big(\|\eta\|_{  L_{s, \theta}   (\Sigma^T) }  +  (1+\lambda)  \|f_0\|_{ S_{s,   \theta } (\Omega \times \bR^3)  }\big)\\
	& + N  \|f\|_{  L_{2, \theta}   (\Sigma^T) },\notag
\end{align}
where  $N = N (\delta_0,  \kappa, \varkappa, r, K,  \theta, \Omega)$.
Furthermore, 
\begin{itemize}
\item[--] if $r \in (2, 7)$,    we have
    \begin{align}
        \label{eq2.3.1.2}
        \|f\|_{L_{r_1, \kappa \theta} (\Sigma^{T}) } +   \|\nabla_p f\|_{L_{r_2, \kappa \theta} (\Sigma^{T}) } \le \text{r.h.s. of \eqref{eq2.3.1}},
    \end{align}
    where $r_1, r_2 > 1$ are  numbers satisfying the relations
        \begin{align}
            \label{eq2.3.1.3}
        \frac{1}{r_1} = \frac{1}{r} - \frac{1}{7},  \quad   \frac{1}{r_2} =   \frac{1}{r}  - \frac{1}{14},
        \end{align}

\item[--] if $r \in (7, 14)$, 
\begin{align}
            \label{eq2.3.1.4}
    \|f\|_{ L_{\infty, \kappa \theta} (\Sigma^T)  }  + \|\nabla_p f\|_{L_{r_2, \kappa \theta} (\Sigma^{T}) } \le \text{r.h.s. of \eqref{eq2.3.1}},
\end{align}
where $r_2$ is defined  in \eqref{eq2.3.1.3},

\item[--] if $r > 14$, then, for any $\alpha \in (0, 1 - 14/r)$, one has
\begin{align}
                \label{eq2.3.1.5}
  & \sum_{s \in \{2, \infty\} } \|f\|_{ L_{\infty}  ((0, T) \times \Omega) W^1_{ s,  \theta/2} (\bR^3)  } \\
  & + \|[f, \nabla_p f]\|_{  L_{\infty} ((0, T))  C^{\alpha/3, \alpha}_{x, p} (\Omega \times \bR^3) }  \le \text{r.h.s. of \eqref{eq2.3.1}}. \notag
\end{align}
\end{itemize}
In all the estimates \eqref{eq2.3.1.2}, \eqref{eq2.3.1.4}, and \eqref{eq2.3.1.5}, one needs
to take into account the dependence of $N$ on $r_1, r_2$, and $\alpha$.
\end{proposition}

\begin{remark}
    We point out that for the Newtonian KFP equation on the whole space, there is no loss of the weight in the momentum variable in the $S_r^{N}$ (see \eqref{eq1.2.4.0}) estimate (see \cite{DY_21a}). Furthermore,  in \cite{DGY_21},   the present authors have established an $S_r^{N}$ estimate with the loss of weight in the presence of a spatial boundary. In the relativistic case, we,  loosely speaking, lose weight due to the presence of the spatial boundary and the relativistic transport term.
\end{remark}

We will prove the assertions in the order we stated them.

\begin{proof}[Proof of Lemma \ref{lemma 6.1}]
Denote
\begin{equation}
            \label{eq6.1.1}
\index{$\sigma$} \sigma (p) =  2 \int_{\bR^3} \Phi (P, Q)  J (q) \, dq.
\end{equation}
It is well known that $\sigma$ is a bounded uniformly nondegenerate symmetric matrix-valued function (see Lemma \ref{lemma G.2} $(i)$). The desired assertion follows from this and Lemma \ref{lemma A.1} $(i)$.
\end{proof}

We will break down the proof of Proposition \ref{proposition 2.1} into three significant steps. The initial two will be explained in Lemmas \ref{lemma 7.3} - \ref{lemma 7.4}.
Our argument goes as follows. 
\begin{itemize}
    \item First, we construct a variational solution, which we call the ``finite energy weak solution" (see Lemma \ref{lemma 7.3}). It is a quadruple $(f, f_{+}^{\star}, f_{-}^{\star}, f_T^{\star})$, where $f_{\pm}^{\star}$ and $f_T^{\star}$ are the functions that appear in the boundary terms in the integral formulation. We impose additional conditions $f_0 \in L_{\infty} (\Omega \times \bR^3)$ and $\eta \in L_{\infty} (\Sigma^T), \nabla_p b \in L_{\infty} (\Sigma^T)$. 
    The first two are needed to ensure that the boundary terms in the integral formulation are well defined for any test function $\phi \in C^{0, 1}_0 (\overline{\Sigma^T})$.
    
    \item By using a mirror-extension argument as in \cite{GHJO_20}, \cite{DGO_22},   \cite{DGY_21},   we show that if $\theta$ is sufficiently large, then any ``finite energy weak solution" is  a  strong solution (see Definition \ref{definition 6.1}). To implement the mirror-extension argument in the integral formulation,  one needs to work with general test functions that are Lipshitz up to the grazing set. This explains the necessity of the additional boundedness assumptions in the previous paragraph.
    
    \item We use a limiting argument to get rid of the boundedness conditions on $f_0$, $\eta$, and $\nabla_p b$. The resulting solution $f$ satisfies $Y f \in L_2 (\Sigma^T)$, so that the traces are well-defined  (see \eqref{eqC.6.1} in Lemma \ref{lemma C.6}), and they coincide with the functions $f_{\pm}^{\star},  f_T^{\star}, f_0$, and, in addition, the SRBC holds. We will also explain why the limiting procedure preserves the energy identity and the  ``mirror-extension property".
\end{itemize}

We first state a lemma concerning  finite energy weak solutions to the general KFP equation 
   \begin{align}
    \label{eq6.3}
	&	Y f  - \nabla_p \cdot (a \nabla_p f) + b \cdot \nabla_p f + (c + \lambda) f = \eta, \\
     \label{eq6.4}
 &f(t, x, p) = f (t, x, R_x p), \, z \in \Sigma^T_{-}, \, \, f (0, \cdot) = f_0 (\cdot).
  \end{align}

The nonrelativistic counterpart of this result was established in \cite{DGY_21}.

\begin{lemma}
        \label{lemma 7.3}
Let
\begin{itemize}
\item[--] $\Omega$ be a $C^{1, 1}$ domain
    \item[--]  $a = a (z), z \in \bR^7_T$ be a bounded measurable function satisfying
\begin{equation}
			\label{eq1.1}
	\delta |\xi|^2 \le a (z) \xi_i \xi_j \le \delta^{-1}  |\xi|^2,  \, \forall z \in \Sigma^T, \, \xi \in \bR^3
\end{equation}
for some $\delta \in (0, 1)$,

\item[--]  $b$ and $c$ satisfy \eqref{eq1.2.2}
\item[--]  $\nabla_p a, \nabla_p b \in L_{\infty} (\Sigma^T)$, 
 \item[--]  $f_0 \in L_{2, \theta} (\Omega \times \bR^3) \cap L_{\infty} (\Omega \times \bR^3),$

\item[--] $\eta \in L_{2, \theta} (\Sigma^T) \cap L_{\infty} (\Sigma^T)$.
 \end{itemize}
 Then, for any $\lambda \ge 0$, there exists a quadruple $(f$, $f_{+}^{\star}$, $f_{-}^{\star}$,  $f_T^{\star}$) such that
\begin{enumerate}[i)]
     \item \label{i}
 $f, \nabla_p f \in L_{2, \theta} (\Sigma^T)$, 
 $f_{\pm}^\star  \in L_{\infty} ( \Sigma^T_{\pm})$,
 $f_T^{\star} \in L_{2, \theta} (\Omega \times \bR^3) \cap L_{\infty} (\Omega \times \bR^3)$,

 \item \label{ii} $f_{-}^\star (t, x, p) =  f_{+}^\star (t, x, R_x p)$ a.e. on $\Sigma^T_{-}$,

 \item   \label{iii}
 for any $\phi \in C^{0, 1}_0 (\overline{\Sigma^T})$,
 \begin{align}
 			\label{eq1.0.1}
	&-	\int_{\Sigma^T}
    (Y \phi)    f \, dz\\
&
		 + \int_{\Omega \times \bR^3 }     \big( f_T^{\star} (x, p)  \phi (T, x, v)  -  f_0 (x, p) \phi (0, x, p)\big)\, dx dp \notag\\
 	 &
		+  \int_{\Sigma^T_{+}}     f_{+}^\star \phi \, |v (p) \cdot n_x|  dS_x  dp dt
	-   \int_{\Sigma^T_{-}}   f_{-}^\star \phi \, |v (p) \cdot n_x|    dS_x  dp dt \notag\\
	  &
	  +\int_{\Sigma^T}  (a\nabla_{p} f) \cdot \nabla_{p} \phi  \,dz
	   + \int_{\Sigma^T} (c + \lambda) f \phi \, dz
	+ \int_{\Sigma^T} (b \cdot  \nabla_{p} f ) \phi  \,dz
	  = \int_{\Sigma^T}   \eta \phi   \, dz. \notag
	  \end{align}
Furthermore, one has
\begin{align}
    \label{eq7.3.1}
&\|f_T^{\star}\|_{ L_{2, \theta} (\Omega \times \bR^3) } 
+  \|f\|_{   L_2 ((0, T) \times \Omega) W^1_{2, \theta} (\bR^3)   } \\
&\le N   \|\eta\|_{ L_{2, \theta} (\Sigma^T) }
+ N \|f_0\|_{ L_{2, \theta} (\Omega \times \bR^3) },  \notag\\ 
    \label{eq7.3.2}
& \max\{\|f_T^{\star}\|_{ L_{\infty} (\Omega \times \bR^3) },
 \|f^{\star}_{\pm}\|_{ L_{\infty} (\Sigma^T_{\pm}) },
  \|f\|_{ L_{\infty} (\Sigma^T) } \}\\
&\le   \|\eta\|_{ L_{\infty} (\Sigma^T) }
+  \|f_0\|_{ L_{\infty} (\Omega \times \bR^3) }, \notag
\end{align}
where $N = N (\delta, \theta, K, T)$.

We say that $f$ is a finite energy weak solution to \eqref{eq6.3} - \eqref{eq6.4}.
\end{enumerate}
\end{lemma}

\begin{remark}
    \label{remark 5.9}
Here, we elaborate on various notions of weak solutions to the linear Landau equation \eqref{eq1.0} - \eqref{eq1.0.0} that we use in the present paper.

\begin{itemize}
    \item \textit{Finite energy solutions.} 
    These are functions of class $C ([0, T]) L_2 (\Omega \times \bR^3) \cap L_2 ((0, T) \times \Omega) W^1_2 (\bR^3)$ that satisfy the integral formulation of \eqref{eq1.0} - \eqref{eq1.0.0} with the test functions satisfying the SRBC. See Definition \ref{definition 27.1}.

    \item \textit{Finite energy weak solutions.}  In the proof of the existence part in Proposition \ref{proposition 2.1}, we need to construct a `weak solution' in the class $f \in L_2 ((0, T) \times \Omega) W^1_2 (\bR^3)$  such that its integral formulation holds for test functions that are Lipschitz up to the kinetic boundary $\partial \Omega \times \bR^3$. This integral formulation is necessary for the mirror-extension trick, requiring well-defined 'kinetic boundary' terms. Lemma \ref{lemma 7.3} provides a weak solution meeting these requirements.

    \item \textit{Very weak solutions.} We justify the uniqueness in the class of finite energy  solutions via a duality argument, which works for a class of weaker solutions, which we call ``very weak solutions." These are $L_2 (\Sigma^T)$ functions satisfying an integral formulation with all derivatives  `transferred' onto a test function satisfying the SRBC. See Definition \ref{definition 27.3}. 

    \item \textit{Intermediate finite energy solutions.} In the proof of the existence of finite energy solutions (see the argument of Proposition \ref{proposition 27.4}), we first construct a solution in a slightly weaker class, where elements lack the temporal continuity in $L_2 (\Omega \times \bR^3)$ (see definition \ref{definition 27.6}).
\end{itemize}
 We note that
\begin{align*}
     &   \text{finite energy solution} \implies \text{intermediate finite energy solution}
    \implies \text{very weak solution}, \\
    &\text{finite energy weak solution} \implies \text{intermediate finite energy solution}.
\end{align*}

The present authors also used the notion of the \textit{finite energy weak solution}  in the construction of a strong solution to a linear non-relativistic Landau equation (see \cite{DGY_21}).
\end{remark}

\begin{proof}[Proof of Lemma \ref{lemma 7.3}]
We repeat almost word-for-word the argument of Theorem 1.5 in  \cite{DGY_21} (see Section 3 therein). Here, we delineate the argument. The main idea is to discretize the velocity diffusion to obtain a perturbed kinetic transport equation in a bounded domain for which the well-posedness problem is well understood (see, for example,   \cite{U_86} and \cite{BP_87}).  We use the energy argument to derive uniform estimates with respect to the parameters of our approximation scheme.
The key difficulty is to `preserve' the `boundary information' on $(0, T) \times \gamma_{\pm}$ in the $\text{weak}^{*}$ compactness argument.
  By designing a specific discretization of the velocity diffusion that respects the maximum principle, we are able to obtain  $L_{\infty}$ estimates of the solution and its traces that are uniform throughout the approximation scheme.
\end{proof}

We first prove the following lemma, which is  Proposition \ref{proposition 2.1} under more restrictive assumptions mentioned at the beginning of the section.

\begin{lemma}
    \label{lemma 7.4}
Invoke the assumptions of Proposition \ref{lemma 7.3} and assume, additionally,  
\begin{equation}
        \label{eq7.4.1}
        f_0 = 0, \quad \eta, \nabla_p b \in L_{\infty} (\Sigma^T).
\end{equation}
Then, for sufficiently large $\theta  = \theta (\kappa, \varkappa) > 0$, the following assertions hold.

$(i)$ Any finite energy weak solution $f$  to  \eqref{eq6.3} - \eqref{eq6.4} constructed in Lemma \ref{lemma 7.3} must be a strong solution (see Definition \ref{definition 6.1}).

$(ii)$ The  estimate \eqref{eq2} holds.
\end{lemma}

\begin{proof}[Proof of Lemma \ref{lemma 7.4}] 
The proof is split into six steps. First, we localize in space and momentum variables and use a boundary flattening argument. Then, we use a `mirror extension argument (see Step 3 on p. \pageref{eq2.11.4}) to `erase' the boundary conditions and reduce the equation to the one on the whole space. Then, we use a transformation that reduces the equation to a Newtonian KFP equation, and we apply the  $S_2^{N}$ (see \eqref{eq1.2.4.0}) estimate of \cite{DY_21a}.

\textbf{Step 1: localization.}
Let  $\chi_k = \chi_k (x), k  = 1, \ldots, m$ be a standard partition of unity in $\Omega$ 
	 such that
	 supp $\chi_1 \subset \Omega$,
	$0 \le \chi_k \le 1$, $k = 1, \ldots, m$,
	and
\begin{equation}
			\label{eq2.3}
  |\nabla_x \chi_k| \le N/r_0, \quad
  \begin{cases}
    \chi_k = 1\quad \text{in} \, \,  B_{r_0/4} (\mathsf{x_k})\\
    \chi_k = 0 \quad  \text{in} \, \,  B_{ r_0/2 }^c (\mathsf{x_k})
    \end{cases},\quad k  = 2, \ldots, m,
\end{equation}
where $\mathsf{x_k} \in \partial \Omega, k  = 2, \ldots, m.$

Let $\xi_n = \xi_n (p), n \ge 1$  be  sequence of  functions such that 
\begin{align*}
&  \xi_0  \in C^{\infty} (B_1), \quad	\xi_0 = 1 \, \, \text{on} \, \,  \{ |p| \le 3/4\}, \\
&   \xi_n  \in C^{\infty} (\{2^{n-1} < |p| < 2^{n+3/2}\}), \quad \xi_n = 1 \,  \text{on} \, \, \{  2^{n - 1/2} \le |p| \le  2^{n+1}\},  \\
&	|D_p^l \xi_n|   \le N (l) 2^{-n l},  n, l  \in \{1, 2, \ldots\}.
\end{align*}
We will assume that $n \ge 2$ because the case $n = 1$ is handled in the same way.

We denote
\begin{equation}
			\label{eq2.36}
	f_{k, n} (z) := f (z) \chi_k  (x) \xi_n (p) p_0^{ \omega \theta }
\end{equation}
and note that $f_{k, n}$ satisfies the identity 
\begin{equation}
			\label{eq2.60}
Y f_{k, n}  - \nabla_p \cdot (\sigma_g \nabla_p f_{k, n}) +  b \cdot \nabla_p f_{k, n} + (c + \lambda) f_{k, n} = \eta_{k, n}
\end{equation}
in the sense of the integral identity \eqref{eq1.0.1},
where 
\begin{equation}
			\label{eq2.13}
 \begin{aligned}
\eta_{k, n} =&   (v (p) \cdot \nabla_x \chi_k) f p_0^{\omega \theta } \xi_n + \eta \chi_k  p_0^{\omega \theta } \xi_n \\
&
+ 	\chi_k\bigg(- (\partial_{p_i} \sigma^{i j}_g)  (\partial_{p_j} (\xi_n p_0^{\omega \theta })) f  \\
&
- 2 \sigma_g^{i j} (\partial_{p_i} f) (\partial_{p_j} (p_0^{\omega \theta  } \xi_n)) - \sigma_g^{i j}   \partial_{ p_i p_j} (p_0^{\omega \theta  } \xi_n) f \\
&
 + \big(b \cdot \nabla_p (\xi_n p_0^{\omega \theta })\big) f\bigg),
  \end{aligned}
\end{equation}
where $\omega \in (0, 1)$.
We will focus on the near boundary case when $k \ge 2$. At the end of the proof, we discuss the case when $k = 1$.  
Our goal is to show that
\begin{equation}
			\label{eq2.35}
\begin{aligned}
	&\||Y f_{k, n}| + |D^2_p (f_{k, n})| \|_{  L_2  (\Sigma^T) }   \\
 &+\|f_{k, n}\|_{  L_{14/5}  (\Sigma^T) } + \|\nabla_p f_{k, n}\|_{  L_{7/3}  (\Sigma^T) } \\
& \le N  \|(|f| + |\nabla_p f| + |\eta|) 1_{ 2^{n-1} < |p| < 2^{n+3/2}  }\|_{  L_{2, \beta + \omega \theta } (\Sigma^T) }
\end{aligned}
\end{equation}
for some number $\beta =  \beta (\varkappa) > 0$. When $n = 0$, the indicator function in \eqref{eq2.35} should be replaced with $1_{ |p| < 1}$. 

If \eqref{eq2.35} is true, we take $\omega_1 \in (\omega, \frac{1+\omega}{2})$ and $\theta$ large so that, so that $\beta+ \omega \theta  < \omega_1 \theta$,
  raise \eqref{eq2.35} to the power $2$, and sum with respect to $n$ and $k$. We obtain
\begin{align}
    \label{eq2.70}
    	\||Y (f p_0^{ \omega \theta }) \|_{  L_2 (\Sigma^T) } + \|D^2_p (f p_0^{ \omega \theta })| \|_{  L_2  (\Sigma^T) } 
	\le N \||f| + |\nabla_p f| + |\eta|\|_{  L_{2,  \omega_1 \theta} (\Sigma^T) }.
\end{align} 
 Integrating by parts and using the Cauchy-Schwarz inequality, we get for $\varepsilon \in (0, 1)$,
\begin{align*}
  &  \int_{\Sigma^T} |\nabla_p f|^2 p_0^{2 \omega_1 \theta} \, dz    \lesssim_{\theta, \omega_1} \varepsilon \|\nabla_p f\|^2_{ L_{2, \omega_1 \theta} (\Sigma^T)  } + \varepsilon^{-1} \|f\|^2_{ L_{2, \omega_1 \theta} (\Sigma^T)  } \\
  &+ \varepsilon  \|D^2_p f\|^2_{ L_{2, \omega \theta } (\Sigma^T)  } 
  +  \varepsilon^{-1} \|f\|^2_{ L_{2, (2 \omega_1  -  \omega) \theta } (\Sigma^T)  }. 
\end{align*}
Due to our choice of $\omega_1$, we have $2 \omega_1 -  \omega \le 1$. 
Hence, by choosing $\varepsilon$ sufficiently small, 
we can drop the norm involving $\nabla_p f$ on the r.h.s. of \eqref{eq2.70} and replace $\omega_1$ with $1$ therein, and obtain the desired estimate of the weighted $S_2$ norm in \eqref{eq2}.  Similarly, we conclude the validity of the estimates of the second and the third terms of the left-hand side of \eqref{eq2}.

For the sake of convenience, we denote 
\begin{equation}
                \label{eq2.50}
    	U = f_{k, n} \quad H = \mathsf{\eta}_{k, n}.
\end{equation}
Without loss of generality, we may assume that $\omega = \frac{1}{2}$.

\textbf{Step 2: boundary flattening.}
We fix a point $\mathsf{x}_k \in \partial \Omega, k = 2, \ldots, m$ and relabel it as $x_0$.
There exists a function $\rho \in C^{1, 1}_b (\bR^2)$ such that
\begin{align*}
     &   \partial \Omega \cap B_{r_0} (x_0) \subset \{x: x_3 = \rho (x_1, x_2) \},\\
    &\Omega_{r_0} (x_0) : =  \Omega \cap B_{r_0} (x_0) \subset \{x: x_3 < \rho (x_1, x_2)\}.
\end{align*}
Let 
\begin{equation}
            \label{eq2.55}
    \Psi: \Omega_{r_0} (x_0)  \times \bR^3 \to \bH_{-} = \bR^3_{-}
\times \bR^3, \quad (x, p) \to (y, w)
\end{equation}
 be the  transformation given by
 \begin{equation}
            \label{eq2.56}
   y = \psi (x), \quad w = (D \psi (x)) p,
\end{equation}
where $\psi$ is the inverse of 
\begin{equation*}
    \psi^{-1} (y)
    = \begin{pmatrix}
	y_1 \\ y_2 \\ \rho (y_1, y_2)
       \end{pmatrix}
	+ y_3 \begin{pmatrix}
			-\rho_1^{(y_3)} \\ -\rho_2^{(y_3)} \\ 1
	          \end{pmatrix},
\end{equation*}
where $\rho_i = \partial_{x_i} \rho, i = 1, 2$, and  $\rho^{\varepsilon}$ is a standard mollification of $\rho$.
It follows from the expression of the Jacobi matrix  $(\frac{\partial x}{ \partial y})$ (see \eqref{eqA.6.0}) that $\psi$ is local a $C^{1, 1}$ diffeomorphism.

 A similar diffeomorphism was used to study the Newtonian KFP and the Landau equations in a bounded domain with the SRBC (see \cite{GHJO_20}, \cite{DGO_22}, \cite{DGY_21}).
$\Psi$ has two  special features:
\begin{itemize}
\item it preserves the form of the Newtonian kinetic Fokker-Planck equation in the sense explained in Section 2.1 of \cite{DGY_21};

\item 
 it preserves the SRBC, i.e.,
\begin{equation}
			\label{eq2.8}
	\widehat U_{-}^{\star} (t, y_1, y_2,  w) = \widehat U^{\star}_{+} (t, y_1, y_2, \bm{R} w), \, 
	\text{whenever} \, w_3 < 0, 
\end{equation}
where
\begin{equation}
			\label{eq2.8.1}
 \widehat U_{\pm}^{\star} (t, y_1, y_2,  w) = U_{\pm}^{\star} (t, x (y_1, y_2, 0), p (y_1, y_2, 0, w)), 
\end{equation}
and $U_{\pm}^{\star}$ were introduced in Lemma \ref{lemma 7.3}.
The identity \eqref{eq2.8}  follows from the fact that whenever $y_3 = 0$, one has
\begin{align*}
     &(R_x p) (y, w) = (p - 2 (p\cdot n_x) n_x) (y, w) \\
	& = \begin{pmatrix}w_1 + \rho_1 w_3\\
     w_2 + \rho_2 w_3 \\ \rho_1 w_1 + \rho_2 w_2 - w_3 \end{pmatrix}
     = \bigg(\frac{\partial x}{\partial y}\bigg)_{| y_3 = 0} \begin{pmatrix} w_1\\ w_2\\ -w_3 \end{pmatrix}.
\end{align*}
where  the Jacobi matrix is computed in  \eqref{eqA.6.0}.
\end{itemize}
The first property does not hold for the relativistic Fokker-Planck (see \eqref{eq2.1}). Nevertheless, 
 this equation  can still be reduced to a Newtonian KFP  type equation (see Step 4 below).

 Next,  denote
\begin{align}
&\label{eq2.0.0}
	\widehat u (t, y, w) = u (t, x (y), p (y, w)), \quad	\mathsf{J}_{\psi} = \bigg|\text{det} \bigg(\frac{\partial x}{\partial y}\bigg)\bigg|^2,   \\
	& \label{eq2.0}
 W  = \frac{w}{ \big(1 +  \big|\frac{\partial x}{\partial y} w\big|^2\big)^{1/2}},
 \quad	A  = \bigg(\frac{\partial y}{\partial x}\bigg) \widehat \sigma_g \bigg(\frac{\partial y}{\partial x}\bigg)^T, \quad B  =  \bigg(\frac{\partial y}{\partial x}\bigg) \widehat b,\\
	&\label{eq2.0.1}
	        X =  (X_1, X_2, X_3)^T =   \bigg(\frac{\partial y}{\partial x}\bigg) \bigg(\frac{\partial p}{\partial y}\bigg) W
        =  \bigg(\frac{\partial y}{\partial x}\bigg) \frac{\partial \big(\frac{\partial x}{\partial y} w\big)} {\partial y} W.
\end{align}
For a function $\Xi$ on $\Omega_{r_0} (x_0) \times \bR^3$, we denote 
\begin{equation}
            \label{eq2.49}
    	\widetilde \Xi (y, w) = \widehat \Xi (y, w) \mathsf{J}_{\psi} (y).
\end{equation}

Changing variables in \eqref{eq1.0.1} (see Section \ref{section E.1}), we conclude for any $\phi \in C^{0, 1}_0 ([0, T] \times \overline{\Omega_{r_0} (x_0) } \times \bR^3)$ (see  p. \pageref{eq1.11}), we have 
	\begin{align}
							\label{eq2.1}
 	& - \int_{  \mathbb{H}^T_{-} }
      (\partial_t \widehat \phi +  W \cdot \nabla_y \widehat  \phi) \,     \widetilde U  dy dw dt
      + \int_{\bH_{-}}   \widetilde U^{*}  (T, y, w) \widehat \phi (T, y, w) \, dy dw \notag\\
	   &
       + \int_{  \mathbb{H}_{-}^T }
       \big((\nabla_w \widetilde U)^T A \nabla_{w} \widehat \phi
       + \widetilde U X \cdot \nabla_{w} \widehat \phi
	+ \big(B \cdot \nabla_{w} \widetilde U \, \big) \, \widehat \phi
       + (\widehat c  + \lambda)  \widetilde U  \, \widehat \phi \,\big) \,  dy dw dt \notag\\
         &
		+ \int_0^T \int_{\bR^2 \times \bR^3_{+}}  |w_3|    \widetilde U^{  \star }_{+} \widehat \phi \, dy_1 dy_2  dw dt 
	-   \int_0^T \int_{\bR^2 \times \bR^3_{-}}  |w_3|    \widetilde U^{  \star }_{-} \widehat \phi \, dy_1 dy_2 dw  dt\notag\\
	&= \int_{ \mathbb{H}^T_{-}} \widehat \phi \,  \widetilde H \, dy dw dt,
\end{align}	
where $\widetilde U^{  \star }_{+} = \widehat U^{  \star }_{+} \mathsf{J}_{\psi}|_{y_3= 0}$.

\textbf{Step 3: mirror extension.} For a function $\Xi  = \Xi (x, p)$ on $\Omega_{r_0} (x_0) \times \bR^3$, we denote 
\begin{equation}
        \label{eq2.11.4}
    \overline{\Xi} (y, w) : = \begin{cases} 	\widetilde \Xi (y,  w), \, \, (y, w) \in \bH_{-},\\
								\widetilde \Xi (\bm{R} y, \bm{R} w),  \, \, (y, w)  \in  \bH_{+}
					\end{cases}
\end{equation}
(see \eqref{eq2.49}. We call $ \overline{\Xi}$ the mirror extension of $\Xi$.

Next, let $G \subset \bR^3$ be the even extension of $\psi (\Omega_{r_0} (x_0)) \subset \overline{\bR^3_{-}}$ across the plane $y_3 = 0$.
We  set
\begin{align}
 &        \label{eq2.11}
   \mathcal{A} (t, y, w)
         = \begin{cases} A (t, y, w), \quad (t, y, w) \in (0, T) \times \overline{\psi (B_{r_0} (x_0))} \times \bR^3,\\
            \bm{R} \, A (t, \bm{R} y, \bm{R} w)\,   \bm{R}, \, \, (t, y, w) \in  (0, T) \times (G \cap \bR^3_{+}) \times \bR^3,
            \end{cases}  \\
&
\label{eq2.11.2}
 \mathcal{B} (t,  y,  w) = \begin{cases} B (t, y, w), \, \, (t, y, w) \in \overline{\bH^T_{-}}, \\
 \bm{R} \,  B (t, \bm{R} y, \bm{R} w), \, \, (t, y, w) \in \bH^T_{+}, \end{cases} \\
 & \label{eq2.11.1}
   \mathcal{X} (y, w) = \begin{cases} X (y, w), \, \, (y, w) \in \overline{\bH_{-}}, \\ \bm{R} \,  X (\bm{R} y, \bm{R} w), \,\,  (y, w) \in \bH_{+}, \end{cases} \\
&\label{eq2.11.3}
\mathcal{W} (y, w) = 
\begin{cases} 
\frac{w}{ \big(1 + |V_1|^2\big)^{1/2}}, \, \, (y, w) \in \overline{\psi (B_{r_0} (x_0))} \times \bR^3, \\
  \frac{w}{ \big(1 +  |V_2|^2\big)^{1/2}}, \,\,  (y, w) \in  (G \cap \bR^3_{+}) \times \bR^3, 
\end{cases}
\end{align}
where
\begin{align*}
& V_1 = \bigg(\frac{\partial x}{\partial y}\bigg) (y) w, \quad V_2 =    \bigg(\big(\frac{\partial x}{\partial y}\big) (\bm{R} y)\bigg) (\bm{R} w).
\end{align*}
 We also set $\mathcal{C}$ to be the even extension in $y_3$ and $w_3$ of $\widehat c$.

We now find an equation satisfied by $\overline{U}$.
  We fix a test function $\phi \in C^{0, 1} ([0, T) \times \overline{G} \times \bR^3)$ vanishing for large $z$.
Replacing $\widehat \phi$ with $\phi (t, \bm{R} y, \bm{R} w)$ in the identity   \eqref{eq2.1} and changing variables $x \to \bm{R} x, w \to \bm{R} w$ give
	\begin{align}
	            	\label{eq2.2}
 	& - \int_{  \mathbb{H}^T_{+} }
      (\partial_t  \phi +  \mathcal{W} \cdot \nabla_y   \phi) \,     \overline{U}  dy dw dt
      + \int_{ \bH_{+} } \phi (T, y, w) \overline{U} (y, w) \, dy dw \\
	   &
       + \int_{  \mathbb{H}_{+}^T }
       \big(   (\nabla_w \overline{U})^T \mathcal{A} \nabla_{w}  \phi
       + \overline{U} \mathcal{X} \cdot \nabla_{w}  \phi
	+ \big(\mathcal{B} \cdot \nabla_{w} \overline{U} \, \big) \,  \phi
       + (\mathcal{C}  + \lambda)  \overline{U}  \,  \phi \,\big) \,  dy dw dt \notag\\
         &
		+ \int_0^T \int_{\bR^2 \times \bR^3_{-}}  |w_3|    \overline{U}^{  \star }_{+} (t, y_1, y_2, \bm{R} w)  \phi \, dy_1 dy_2  dw dt \notag\\
	&-   \int_0^T \int_{\bR^2 \times \bR^3_{+}}  |w_3|    \overline{U}^{  \star }_{-}  (t, y_1, y_2, \bm{R} w) \phi \, dy_1 dy_2 dw  dt \notag\\
	&= \int_{ \mathbb{H}^T_{+}}  \phi \,  \overline{H} \, dy dw dt. \notag
\end{align}	
Adding Eq. \eqref{eq2.2} to  Eq. \eqref{eq2.1} with $\widehat \phi$ replaced with $\phi$  and using  \eqref{eq2.8}, we cancel  the integrals over the incoming/outgoing boundaries  and   conclude that
the mirror extension $\overline{U}$ satisfies the  identity 
\begin{equation} 	
			\label{eq2.5.0}
		\partial_t \overline{U} +  \mathcal{W}   \cdot \nabla_y \overline{U}  - \nabla_w \cdot (\mathcal{A} \nabla_w \overline{U}) + \mathcal{B} \cdot \nabla_w \overline{U} - \nabla_w \cdot (\mathcal{X} \overline{U})   + (\mathcal{C} + \lambda) \overline{U} = \overline{H}
\end{equation}
in the weak sense on $[0, T) \times \overline{G} \times \bR^3$, i.e., for any $\phi \in C^{0, 1}_0 ([0, T) \times \overline{G} \times \bR^3)$,
\begin{equation} 	
			\label{eq2.5}
\begin{aligned}
 	& - \int_{\bR^7_T}
      (\partial_t  \phi +  \mathcal{W} \cdot \nabla_y   \phi) \,     \overline{U}  dy dw dt
	  \\
	   &
       + \int_{  \bR^7_T }
       \big(   (\nabla_w \overline{U})^T \mathcal{A} \nabla_{w}  \phi
       + \overline{ U } \mathcal{X} \cdot \nabla_{w}  \phi
	+ \big(\mathcal{B} \cdot \nabla_{w} \overline{ U } \, \big) \,  \phi
       + (\mathcal{C}  + \lambda)  \overline{ U }  \,  \phi \,\big) \,  dy dw dt\\
	&= \int_{ \bR^7_T }  \phi \,  \overline{H}\, dy dw dt.
\end{aligned}	
\end{equation}

\textbf{Step 4: reducing Eq. \eqref{eq2.5} to  a Newtonian  KFP  equation}.
Recall  that
 $U (t, \cdot) = f_{k, n} (t, \cdot)$  is supported on $\Omega_{r_0/2} (x_0)  \times \{ 2^{n-1} < |w| < 2^{n+3/2}\}$, and hence, 
\begin{equation}
            \label{eq2.53.0}
	\widehat U (t, \cdot, \cdot) \, \, \text{vanishes outside} \, \, B_{3r_0/4} (x_0) \times \{  2^{n-3/2} <  |w| < 2^{n+2}\}
\end{equation}
for sufficiently small $r_0$.
For any $y \in G$, we denote
\begin{equation}
            \label{eq2.53}
    	  \mathcal{W}_y (w) = \mathcal{W} (y, w)
\end{equation}
(see \eqref{eq2.11.3}). By the assertion $(ii)$ in  Lemma \ref{lemma A.3}, for any $y \in G$, the mapping
$$
	\mathcal{W}_y: \{ |w| < 2^{n+2} \} \to \bR^3
$$
is a diffeomorphism onto its image, and by \eqref{A.3.0}-\eqref{A.3.8}, one has
\begin{equation}
			\label{eq2.15}
\begin{aligned}
	&  \sup_{ |w| < 2^{n+2} }  |D^j \mathcal{W}_y| \le N 2^{  - j n }, j = 0, 1, 2, \, \, \sup_{ \mathcal{W}_y (\{|w| < 2^{n+2}\}) }  |D (\mathcal{W}_y)^{-1}| \le N 2^{3 n}, \\
&  \sup_{ \mathcal{W}_y (\{|w| < 2^{n+2}\}) } |D^2 (\mathcal{W}_y)^{-1}| \le N 2^{5 n},
\end{aligned}
\end{equation}
where $N=  N (\Omega) > 0$.
We also introduce the mapping 
\begin{equation}
            \label{eq2.52}
	\Upsilon_n (y, w) = (y, \mathcal{W} (y, w)) : G \times \{|w| < 2^{n+2} \}  \to \bR^6.
\end{equation}
Due to Lemma \ref{lemma A.5} $(i)$,  $\Upsilon_n$ is a globally bi-Lipschitz map onto its image,
so that, if we change variables 
\begin{equation}
            \label{eq2.54}
	v = \mathcal{W}_y w
\end{equation}
 in \eqref{eq2.5}, then the new integral identity \eqref{eq2.14} will hold on a set of Lipschitz test functions.  

Next, for  a function 
$\Xi = \Xi (y, w)$ on $G \times \{|w| < 2^{n+2}\}$, we set
\begin{equation}
			\label{eq2.22}
	  \hathat \Xi (y, v) = \Xi (y, (\mathcal{W}_y)^{-1} (v)), \,  (y, v) \in \Upsilon_n (G \times  \{|w| < 2^{n+2}\}).
\end{equation}
For the sake of convenience, we change the notation as follows:
\begin{equation}
			\label{eq2.22.1}
	\mathcal{U} := \overline{U}, \quad  \mathcal{H} := \overline{H}.
\end{equation}
We fix a test function 
$$
	\phi \in C^{0, 1}_0 ([0, T] \times \overline{G} \times \{|w| \le 2^{n+2}\})
$$
and   change variables  
 $$
    w = (\mathcal{W}_y)^{-1} (v)
 $$
in the identity  \eqref{eq2.5}. Due to the identity \eqref{eqC.3} in Section \ref{section E.2}, we obtain
\begin{align}
\label{eq2.14}
 	&- \int_{  \bR^7_T }
      (\partial_t  \hathat\phi +  v \cdot \nabla_y  \hathat \phi) \,  \mathcal{\hathat U}  \mathsf{J}_{\mathcal{W}} dy dv dt
      + \int_{ \bR^6 } \hathat\phi (T, y, v) \,   \mathcal{\hathat U}  (T, y, v) \mathsf{J}_{\mathcal{W}} dy dv \\
	   &
       + \int_{  \bR^7_T }
       \big(   (\nabla_v\,  \mathcal{\hathat U})^T \mathbb{A} \nabla_{v} \hathat  \phi
       + \big((\mathbb{X} + \mathbb{G}) \cdot \nabla_v \hathat \phi\big) \,  \mathcal{\hathat U} \notag \\
&	+ \big(\mathbb{B} \cdot \nabla_{v} \, \mathcal{\hathat U}\, \big) \,  \hathat \phi
       + (\mathbb{C}  + \lambda) \,  \mathcal{\hathat U}  \, \hathat \phi \,\big) \, \mathsf{J}_{\mathcal{W}} dy dv dt\notag\\
	& 
	= \int_{ \bR^7_T  } \hathat \phi \,  \mathcal{ \hathat H } \, \mathsf{J}_{\mathcal{W}} dy dv dt, \notag
\end{align}	
where 
\begin{align}
&\label{eq2.9.2.1}
	\mathsf{J}_{\mathcal{W}} = \bigg|\text{det} \frac{\partial w}{\partial v}\bigg|,\\
&	\label{eq2.9}
\mathbb{A} (t, y, v)
=  \big(\frac{\partial v}{\partial w}\big)    \mathcal{\hathat A}  (t, y, v) \big(\frac{\partial v}{\partial w}\big)^T,\\
&
\label{eq2.9.2}
	  \mathbb{B} (t, y, v) = 
	  \big(\frac{\partial v}{\partial w}\big)  \mathcal{\hathat B} (t, y, v),
	  \quad 	\mathbb{C} (t, y, v) = (\widehat c) (t, y, w (y, v)), \\
	&\label{eq2.9.1}
\mathbb{X} (t, y, v) =   \big(\frac{\partial v}{\partial w}\big)\mathcal{\hathat X}  (t, y, w (y, v)) 1_{ y \in G,  |w (y, v)| < 2^{n+2}  }, \\
\label{eq2.9.1.1}
&\mathbb{G} (t, y, v) = (\frac{\partial v}{\partial w}\big)\big(\frac{\partial w}{\partial y}\big) v \,  1_{ y \in G,  |w (y, v)| < 2^{n+2}  }. 
\end{align}
As we mentioned in the previous paragraph, thanks to Lemma \ref{lemma A.5} $(i)$, we may replace $\hathat \phi$ with any  test function
$$
	\phi \in C^{0, 1}_0 \big([0, T] \times \Upsilon_n (\overline{G} \times \{|w| \le 2^{n+2}\})\big)
$$
in the identity \eqref{eq2.14}.

We  now  replace $\mathbb{A}$ with $\mathfrak{A}$ as follows so that $\mathfrak{A} = \mathbb{A}$ on the support of $\dtilde U$ 
 contained in 
$
	 \Upsilon_n (B_{3r_0/4} (x_0) \times \{ |w| < 2^{n+2}\}).
$
Let $\zeta_n = \zeta_n (y, v)$ be a smooth cutoff function  such that $0 \le \zeta_n \le 1$ and
\begin{align}
	&	\zeta_n  = 1 \, \,  \text{on} \,\,   \Upsilon_n (B_{3r_0/4} (x_0) \times \{|w| < 2^{n+2}\}), \notag\\
& \label{eq2.17}
	|\nabla_x \zeta_n| + |\nabla_v \zeta_n| \le N (\Omega).
\end{align}
Introduce
\begin{align}
	& 	 \label{eq2.10}
	\mathfrak{A}  = \mathbb{A}  \zeta_n + (1 -  \zeta_n) \bm{1}_3, \\
	 & \label{eq2.10.1}
\dtilde U = \hathat U \mathsf{J}_{\mathcal{W}}.
\end{align}
We also extend $\mathbb{B}, \mathbb{X}$, and $\mathbb{C}$ by $0$ outside  $[0, T] \times \Upsilon (\overline{G} \times \{|w| \le 2^{n+2}\})$.
It follows that for any $\phi \in C^{0, 1}_0 ([0, T] \times \bR^6)$ such that $\phi (T, \cdot) \equiv 0$, we have 
	\begin{align}
 	& 
 	 \int_{ \bR^7_T }
     \bigg(-(\partial_t \phi +  v \cdot \nabla_y \phi) \,     \mathcal{\dtilde U}  
	+     (\nabla_v \, \mathcal{\dtilde U})^T \mathfrak{A} \nabla_{v} \phi 
 + \lambda \, \mathcal{\dtilde U} \phi\bigg) \,  dy dv dt  \notag \\
	   &
	   	\label{eq2.7} 
   = \int_{  \bR^7_T } \big(-\mathcal{\hathat U} \,  (\mathbb{X}+\mathbb{G}) \cdot \nabla_{v} \hathat \phi
	- \big(\mathbb{B} \cdot  \nabla_{v} \, \mathcal{\hathat U}\, \big) \, \hathat \phi
       - \mathbb{C} \, \mathcal{\hathat U}  \, \hathat \phi \,\big) \, \mathsf{J}_{\mathcal{W}} dy dv dt \\
	&+ \int_{ \bR^7_T } \hathat \phi \,  \mathcal{\hathat H} \mathsf{J}_{\mathcal{W}}  + (\nabla_v \mathsf{J}_{\mathcal{W}})^T\mathbb{A}   (\nabla_v \hathat \phi) \,  \mathcal{\hathat U} \,   dy dv dt.  \notag
\end{align}
In other words, the identity 
\begin{equation}
			\label{eq2.44}
\begin{aligned}
&\partial_t \, \mathcal{\dtilde U} +  v \cdot \nabla_y \, \mathcal{\dtilde U} - \nabla_v \cdot (\mathfrak{A} \nabla_v \, \mathcal{\dtilde U}) + \lambda \, \mathcal{\dtilde U} \\
&   = \bigg(\nabla_v \cdot \big((\mathbb{X}+\mathbb{G})  \, \mathcal{\hathat U}) 
	- \mathbb{B} \cdot \nabla_{v}\,  \mathcal{\hathat U}
       - \mathbb{C} \,    \mathcal{\hathat U}    
	+   \mathcal{\hathat H}\bigg) \mathsf{J}_{\mathcal{W}} \\
	& -\nabla_v \cdot \big(\mathbb{A}  (\nabla_v \mathsf{J}_{\mathcal{W}}) \, \mathcal{\hathat U}\big)
	=: \text{RHS}
\end{aligned}
\end{equation}
holds in the weak sense.
For the reader's convenience, we briefly review the notation introduced above.
\begin{itemize}
\item[--] $U = f_{k, n}$, $H = \eta_{k, n}$ (see  \eqref{eq2.50}, \eqref{eq2.13}),
\item[--] $\widehat U$ is $U$ in coordinates $(t, y, w)$,
\item[--] $\widetilde U$ is $\widehat U$ multiplied by the Jacobian  determinant of the change of variables $(x, v) \to (y, w)$ (see \eqref{eq2.49}),
\item[--] $\mathcal{U} := \overline{U}$ is the mirror-extension of $\widetilde U$ (see \eqref{eq2.11.4}),
\item[--] $\mathcal{\hathat U}$ is $\mathcal{U}$ in coordinates $(t, y, v)$,
\item[--] $\mathcal{\dtilde U}$ is $\mathcal{\hathat U}$ multiplied by the Jacobian  determinant of the change of variables $w \to v$,
\item[--]    $\sigma_g (z)$  (see \eqref{eq6.0}) is the matrix of the leading coefficients in the original equation, 
\item[--] $G$ is the even extension of $\psi (\Omega_{r_0} (x_0))$ across the plane $\{y_3 = 0\}$,
\item[--] $A$ and $B$ (see \eqref{eq2.0}) are the diffusion and drift  coefficients on $(0, T) \times \bR_{-}^3 \times \bR^3$ obtained after the change of variables $(x, v) \to (y, w)$,
\item[--] $\mathcal{A}, \mathcal{B}, \mathcal{C}, \mathcal{X}$ (see \eqref{eq2.11}) are the drift, diffusion, and discount, and `geometric' coefficients  `extended' across the boundary $\{y_3 = 0\} \times \bR^3$,
\item[--] $\mathcal{\hathat A}$, $\mathcal{\hathat B}$, $\mathcal{\hathat C}$, $\mathcal{\hathat X}$ (see \eqref{eq2.22})  are the  coefficients $\mathcal{A}$, $\mathcal{B}$, $\mathcal{C}$, $\mathcal{X}$  in the new coordinates $t, y, v$,
\item[--] $\mathbb{A}, \mathbb{B}, \mathbb{C}$, $\mathbb{X}$ are the drift, diffusion,  discount, and `geometric' coefficients on $(0, T) \times \Upsilon_n (G \times \bR^3)$ obtained after the change of variable $w \to v$, 
\item[--]  $\mathbb{G}$ (see \eqref{eq2.9}) is a second `geometric' coefficient due to the change of variables $w \to v$,
\item[--] $\mathfrak{A}$ (see \eqref{eq2.10}) is an extension of $\mathbb{A}$  to $\bR^7_T$.
\end{itemize}

\textbf{Step 5: $S_2^{N}$ estimate in the $t, y, v$ coordinates.}
We now apply Lemma \ref{lemma 4.1}. We first check its conditions.

\textit{Estimates of the coefficients $\mathbb{A}$, $\mathfrak{A}$, $\mathbb{B}$, $\mathbb{X}$.}
In Lemma \ref{lemma A.7}, we show that the following bounds are valid:
\begin{align}
	\label{eq2.58}
 &   	N_0 (\Omega) 2^{- 6 n} |\xi|^2    \le  \mathfrak{A} (z) \xi_i \xi_j,\\
 & \label{eq2.62}
 \quad |\mathbb{A}| I_{ (y, v) \in \Upsilon_n (G \times \{|p| < 2^{n+2}\})} + |\mathfrak{A}| \le N (\Omega, K),\\
			\label{eq2.18}
& \|\nabla_v \mathbb{A}\|_{   L_{\infty}  ((0, T) \times \Upsilon_n (G \times \{|w| < 2^{n+2}\}))  } + 	\|\nabla_v \mathfrak{A}\|_{   L_{\infty}  (\bR^7_T)  } \le N,\\
						\label{eq2.23}
&	\|\mathfrak{A}\|_{   L_{\infty} ((0, T)) C^{\varkappa/3, \varkappa}_{y, v} (\bR^6)} \le  N (K, \Omega, \varkappa)  2^{  n}, \\
		\label{eq2.45}
&	\|\mathbb{B}\|_{   L_{\infty}  ((0, T) \times \Upsilon_n (G \times \{|w| < 2^{n+2}\}))  } \le N (\Omega, K) 2^{-n}, \\
		\label{eq2.46}
&	\|\mathbb{X}\|_{   L_{\infty}  (\Upsilon_n (G \times \{|w| < 2^{n+2}\}))  } \le N (\Omega),\\
		\label{eq2.47}
&	\|\nabla_v \mathbb{X}\|_{   L_{\infty}  (\Upsilon_n (G \times \{|w| < 2^{n+2}\}))  } \le N (\Omega)  2^{ 2 n}, \\
	\label{eq2.47.1}
& 	\|\mathbb{G}\|_{   L_{\infty}  (\Upsilon_n (G \times \{|w| < 2^{n+2}\}))  }  \le N (\Omega), \\
	\label{eq2.47.2}
& 	\|\nabla_v \mathbb{G}\|_{   L_{\infty}  (\Upsilon_n (G \times \{|w| < 2^{n+2}\}))  }  \le N (\Omega) 2^{ 4 n}.
\end{align}

\textit{$L_2$-integrability of the right-hand side (RHS) of Eq. \eqref{eq2.44}.} 
To show this, we need to first estimate $\mathsf{J}_{\mathcal{W}}$, $\mathcal{\hathat H}$, $\mathcal{\hathat U}$, $\nabla_v \,  \mathcal{\hathat U}$.

First, we estimate  $\mathsf{J}_{\mathcal{W}}$ (see \eqref{eq2.9.2.1}). 
By \eqref{eq2.15} and Lemmas \ref{lemma A.3}  $(ii)$ (see \eqref{A.3.8})
\begin{equation}
			\label{eq2.31}
 \begin{aligned}
	& N_1  \le |\mathsf{J}_{\mathcal{W}} 1_{ y \in G,   w (y, v) < 2^{n+2} }| \le N 2^{9 n}, \\
	&|\nabla_v \mathsf{J}_{\mathcal{W}} 1_{  y \in G,  w (y, v) < 2^{n+2} }| \le  N 2^{11 n},\\
		& |\nabla_y \mathsf{J}_{\mathcal{W}} 1_{  y \in G,  w (y, v) < 2^{ n+2} }| \le N 2^{ 11 n},\\
	& |D^2_v \mathsf{J}_{\mathcal{W}} 1_{  y \in G,  w (y, v) < 2^{ n+2} }| \le N 2^{ 13 n},\\
 \end{aligned}
\end{equation}
where $N_1  = N_1 (\Omega)$, $N = N (\Omega)$.

Second, we bound $\mathcal{\hathat H}$, $\mathcal{\hathat U}$, $\nabla_v \,  \mathcal{\hathat U}$.
By \eqref{eq2.31},  
\begin{equation}
		\label{eq2.27.1} 
\begin{aligned}
	&\|\mathcal{\hathat H}  \mathsf{J}_{\mathcal{W}} 1_{  y \in G,  w (y, v) < 2^{n+2} } \|_{  L_{2} (\bR^7_T) }  \le N 2^{(9 n)/2} \||\mathcal{\hathat H}|^2  \mathsf{J}_{\mathcal{W}} 1_{ y \in G,  w (y, v) < 2^{n+2} }\|^{1/2}_{  L_1 (\bR^7_T) }\\
& = N 2^{(9 n)/2} \|\overline{H}\|_{  L_2    ((0, T) \times G \times \{|w| < 2^{n+2}\})    } \le N 2^{(9 n)/2} \|H\|_{  L_{2} (\Sigma^T) },
\end{aligned}
\end{equation}
where $N  = N (\Omega)$.
Similarly,  by \eqref{eq2.31},
\begin{equation}
			\label{eq2.27}
\begin{aligned}			
&	\|(|\mathcal{\hathat U}| + |\nabla_v \, \mathcal{\hathat U}|) \mathsf{J}_{\mathcal{W}} 1_{ y \in G,  w (y, v) < 2^{n+2} }\|_{  L_{2} (\bR^7_T) }\\
	&\le N (\Omega) 2^{(9 n)/2} \||U| + |\nabla_p U|\|_{  L_{2} (\Sigma^T) }.
	\end{aligned}
\end{equation}

Next, combining  \eqref{eq2.62} - \eqref{eq2.18}, \eqref{eq2.45} - \eqref{eq2.47.2}, and \eqref{eq2.31}, we get 
\begin{equation}
		\label{eq2.29.1}
\begin{aligned}
&	\|\text{RHS}\|_{  L_2 (\Sigma^T) } \\
&\le N 2^{\beta n}  \|(|\mathcal{\hathat U}| + |\nabla_v  \, \mathcal{\hathat U}| + | \mathcal{\hathat H}|) \mathsf{J}_{\mathcal{W}} 1_{ y \in G,  w (y, v) < 2^{n+2} } \|_{  L_2 (\Sigma^T) },\\
\end{aligned}
\end{equation}
where $\beta > 0$ is some constant independent of $n, \varepsilon, c$, and $\theta$, which might change from  line to line. 
Furthermore, by \eqref{eq2.27.1} - \eqref{eq2.27}, we get
\begin{equation}
		\label{eq2.29}
		\|\text{RHS}\|_{  L_2 (\Sigma^T) }  \le N 2^{\beta n}  \||U| + |\nabla_p U| + |H\|_{  L_2 (\Sigma^T) }.
\end{equation}

\textit{$S_2^{N}$-estimate in the $t, y, v$ variables.}
An application of  Lemma \ref{lemma 4.1} with 
\begin{equation}			\label{eq2.41}
	\delta = N_0 2^{-6n} \, \text{(see} \,  \eqref{eq2.58}),  \quad K = N 
 \, \,  \text{(see} \,  \eqref{eq2.23}),
\end{equation}
 gives $\mathcal{\dtilde U} \in S_2^{N} (\bR^7_T)$. Furthermore, by the same lemma (see \eqref{eq4.1.3}),  one has
\begin{equation}
            \label{eq2.41.1}
\begin{aligned}
	\|\mathcal{\dtilde U}\|_{ S_2^{N} (\bR^7_T) } 
	 &\le N 2^{ \beta n} \big(\||\text{RHS}| + |\mathcal{\dtilde U}| + |\nabla_v \,  \mathcal{\dtilde U}|\|_{ L_2 (\bR^7_T) }\big).
\end{aligned}
\end{equation}
 By the fact that $\mathcal{\dtilde U} = \mathcal{\hathat U} \mathsf{J}_{\mathcal{W}}$ (see \eqref{eq2.10.1}) and the bounds \eqref{eq2.31} and \eqref{eq2.27},  we have
 \begin{equation}
                        \label{eq2.41.2}
 \begin{aligned}
  &  \||\mathcal{\dtilde U}| + |\nabla_v \,  \mathcal{\dtilde U}|\|_{ L_2 (\bR^7_T) }\\
  &  \le N 2^{\beta n} \|\mathsf{J}_{\mathcal{W}} (|\mathcal{\hathat U}|^2 + |\nabla_v \, \mathcal{\hathat U}|^2) 1_{ y \in G,  w (y, v) < 2^{n+2} }\|^{1/2}_{ L_1 (\bR^7_T) }\\
  & \le  N 2^{\beta n} \||U| + |\nabla_p U|\|_{  L_2 (\Sigma^T) }.
 \end{aligned}
 \end{equation}
 Combining \eqref{eq2.41.1} - \eqref{eq2.41.2} with \eqref{eq2.29}, we obtain
\begin{equation}
			\label{eq2.30}
	\||(\partial_t + v \cdot \nabla_y) \,  \mathcal{\dtilde U}| +  |D^2_v \, \mathcal{\dtilde U}|\|_{ L_2 (\bR^7_T) } \le N 2^{\beta n} \||U| + |\nabla_p U| + |H|\|_{  L_2 (\Sigma^T) }.
\end{equation}

\textbf{Step 6: Going back to the original variables $t, x, p$.}
\textit{Estimate of $D^2_p U$.}
 First, by the chain rule and change of variables,
\begin{equation}
			\label{eq2.32}
\begin{aligned}
      &   \|D^2_p U\|_{  L_2 ((0, T) \times \Omega_{r_0} (x_0) \times \bR^3)   } \\
&  \le N (\Omega)  \||\nabla_w \widehat U| + |D^2_w \widehat U|\|_{ L_2 ((0, T) \times \psi (\Omega_{r_0} (x_0)) \times \bR^3) }.	 
\end{aligned}
\end{equation}
Furthermore, recall  
\begin{itemize}
\item $\mathcal{U} = \overline{U}$, where the latter is the even extension in the $y_3, w_3$ variables of the function $\widehat U$ (see \eqref{eq2.11.4});
\item the definition of $\mathcal{\dtilde U}$ in \eqref{eq2.10.1}.
\end{itemize}
Then, by \eqref{eq2.15} and the estimates of $\mathsf{J}_{\mathcal{W}}$ and its derivatives (see \eqref{eq2.31}), 
\begin{equation}
			\label{eq2.33}
\begin{aligned}
 	&\||\nabla_w \widehat U| + |D^2_w \widehat U|\|_{ L_2 ((0, T) \times \psi (\Omega_{r_0} (x_0)) \times \bR^3) } \\
		&\le  N 2^{\beta n} \|(|\nabla_v\,  \mathcal{\hathat U}|^{ 2 } + |D^2_v \, \mathcal{\hathat U}|^{ 2 }) \mathsf{J}_{\mathcal{W}} 1_{ y \in G, |w (y, v)| < 2^{n+2}} \|^{1/2}_{ L_1 (\bR^7_T) } \\
	&\le  N 2^{\beta n}  \||\mathcal{\dtilde U}| + |\nabla_v \,  \mathcal{\dtilde U}| + |D^2_v \,  \mathcal{\dtilde U}|\|_{ L_2 (\bR^7_T) }.
\end{aligned}
\end{equation}

Next, combining \eqref{eq2.30} -  \eqref{eq2.33},  we obtain
\begin{equation}
			\label{eq2.34}
\begin{aligned}
	&	\|D^2_p U\|_{ L_2 ((0, T) \times \Omega_{r_0} (x_0) \times \bR^3) } \\
&
\le  N 2^{\beta n} \||U| + |\nabla_p U| + |H|\|_{  L_2 (\Sigma^T) }.
\end{aligned}
\end{equation}
Since $U (t, x, \cdot), H (t, x, \cdot)$ vanish outside $\{ 2^{n-1} < |p| < 2^{n+3/2}\}$, we may replace the right-hand side of \eqref{eq2.34} with
\begin{equation}
			\label{eq2.38}
	  \||U| + |\nabla_p U| + |H|\|_{  L_{2, \beta} ((0, T) \times \Omega \times \bR^3)}.
\end{equation}
Recall that 
\begin{equation}
\label{eq2.37}
U = f \eta_k p_0^{\theta/2}, \quad H = \eta_{k, n},
\end{equation}
  where $\eta_{k, n}$ is defined in \eqref{eq2.13}.   
We conclude that the expression in \eqref{eq2.38} is less than
$$
	N  \||f| + |\nabla_p f| + |\eta|\|_{  L_{2, \theta/2 + \beta} ((0, T) \times \Omega \times \{ 2^{n-1} < |p| < 2^{n+3/2}\})  }.
$$
Thus, the estimate \eqref{eq2.35} for $D^2_p U$ is proved.

\textit{Estimate of the transport term.}
First, by the estimate \eqref{eqD.1.1} in Lemma \ref{lemma D.1}, we have 
\begin{equation}
            \label{eq2.37.0}
\begin{aligned}
\|Y U\|_{ L_2 (\Sigma^T ) } 
&\le N \|(\partial_t + W \cdot \nabla_y)\widehat U\|_{ L_2 ( \bH^T_{-}   ) } \\
&\quad + N \||U|+|\nabla_p U|\|_{ L_{2, 1} (\Sigma^T) }.
\end{aligned}
\end{equation}
Similarly, by the identity \eqref{eqC.3} in Section \ref{section E.2}, 
\begin{equation}
            \label{eq2.37.1}
\begin{aligned}
&	\|(\partial_t + W \cdot \nabla_y)\widehat U\|_{ L_2 ( \bH^T_{-}   ) } \\
&\le N    \||\big(|\partial_t + v \cdot \nabla_y) \,  \mathcal{\hathat U}|^2 + |\mathbb{G}|^2 \, |\nabla_v \,  \mathcal{\hathat U}|^2\big) \mathsf{J}_{\mathcal{W}} \|^{1/2}_{ L_1 ( \bR^7_T) },
\end{aligned}
\end{equation}
where $\mathbb{G}$ is defined in \eqref{eq2.9.1.1}.
Note that
\begin{equation}
            \label{eq2.37.2}
    (\partial_t + v \cdot \nabla_y) \,  \mathcal{\dtilde U} = (\partial_t + v \cdot \nabla_y) \,  \mathcal{\hathat U}
    + (v \cdot \nabla_y \mathsf{J}_{\mathcal{W}}) \,  \mathcal{\hathat U} \mathsf{J}_{\mathcal{W}}.
\end{equation}
Then, by \eqref{eq2.37.1} - \eqref{eq2.37.2}, the Jacobian estimate \eqref{eq2.31}, and the estimate of $\mathbb{G}$ \eqref{eq2.47.1}, and \eqref{eq2.27}, we get
\begin{align*}
 &   \|(\partial_t + W \cdot \nabla_y)\widehat U\|_{ L_2 ( \bH^T_{-}   ) }\\
  &  \le N  \|(\partial_t + v \cdot \nabla_y) \,  \mathcal{\hathat U}\|_{ L_2 (\bR^7_T) }
  + N 2^{\beta n} \||\widehat U| + |\nabla_p  \widehat U|\|_{ L_2 (\bR^7_T) }.
\end{align*}
Combining \eqref{eq2.37.0} with \eqref{eq2.37.2} and \eqref{eq2.30} gives
\begin{align*}
&	\|Y  U\|_{ L_2 (\Sigma^T ) } \\
&\le  N 2^{\beta n}   \||U| + |\nabla_p U|+ |H|\|_{ L_2 ( \bR^7_T) }
	  + N \|\nabla_p U\|_{ L_{2, 1} (\Sigma^T) }\\
	&\le	N  \|(|f| + |\nabla_p f| + |\eta|) \|_{  L_{2, \theta/2 + \beta} ((0, T) \times \Omega \times \{ 2^{n-1/2} < |p| < 2^{n+3/2}\})  }
\end{align*}
provided that $\beta > 1$, which we may certainly assume.
Thus, the estimate \eqref{eq2.35} holds for $Y U$. 
Finally, note that by the embedding theorem for the $S_2^{N} (\bR^7_T)$ space (see Theorem 2.1 \cite{PR_98}), the norms
$$
    \|\mathcal{\dtilde U}\|_{ L_{14/5} (\bR^7_T) }, \|\nabla_v \, \mathcal{\dtilde U}\|_{ L_{7/3} (\bR^7_T) }
$$
are bounded by the right-hand side of \eqref{eq2.41.1}. Then, repeating the above argument, we prove the bound of the second and the third terms on the right-hand side of \eqref{eq2.35}.
\end{proof}

\begin{proof}[Proof of Proposition \ref{proposition 2.1}]
We first impose the additional assumptions  \eqref{eq7.4.1}, which will be removed at the end of the proof.

    \textit{Existence.} Let $b_n \in W^1_{\infty} (\Sigma^T), n \ge 1$ be sequence of functions such that
    $b_n \to b$  a.e., and $\|b_n\|_{L_{\infty}  (\Sigma^T)} \le N_1$ with $N_1$ independent of $n$. We set $f_n$ to be a finite energy weak solution (see Lemma  \ref{lemma 7.3}) to the equation
    \begin{align}
            \label{eq2.1.10}
            Y f_n - \nabla_p \cdot (\sigma_g \nabla_p f_n) + b_n \cdot \nabla_p f_n + (c+\lambda) f_n =  \chi_n (\eta), f_n (0, \cdot) \equiv 0,
    \end{align}
    where $\chi_n (t) = -n \vee t \wedge n$. By \eqref{eq2} in Proposition \ref{proposition 2.1}, we have
    \begin{align}
        \label{eq2.1.11}
        \|f_n\|_{ S_{2,  \kappa \theta } (\Sigma^T)  }  + \|f_n\|_{ L_{14/5, \kappa \theta}  (\Sigma^T)  }  + \| \nabla_p f_n\|_{ L_{7/3, \kappa \theta}  (\Sigma^T)  }  \le N \|\eta\|_{  L_{2, \theta}   (\Sigma^T) }.
    \end{align}
    By  this estimate, there exists a function $f$ such that  $f_n \to f$ in the weak* topology of $S_{2, \kappa \theta} (\Sigma^T)$, so that the bound \eqref{eq2} is true for the limiting function $f$.  We now show that $f$ satisfies the initial condition and the SRBC.  By Ukai's trace lemma \eqref{eqC.6.1} and \eqref{eq2.1.11}, 
    \begin{align*}
         \|f_n\|_{ L_{2} (\Sigma^T_{\pm}, w |v \cdot n_x|)  } \le N \|\eta\|_{ L_{2, \theta} (\Sigma^T) },
    \end{align*}
     Then, $(f_n)_{\pm} \to f_{\pm}^{\star}$ in the weak* topology of $L_{2} (\Sigma^T_{\pm}, w |v \cdot n_x|)$, and $f_{\pm}^{\star}$ satisfy the SRBC.
    Furthermore, since $Y f_n \in L_2 (\Sigma^T)$, we have   Green's identity \eqref{eqC.1.1} with $u$ replaced with $f_n$.
    Then, by using a limiting argument, we see that the integrals over $\Sigma^T_{\pm}$ converge to those with the integrand  $f_{\pm}^{\star}$. Hence, the latter are the traces  of $f$ on $\Sigma^T_{\pm}$, and $0$ is the trace on $\{t = 0\} \times \Omega \times \bR^3$. Testing Eq. \eqref{eq2.1.11} and passing to the limit in \eqref{eq2.1.10}, we conclude that the identity \eqref{eq1.0} holds in the $L_2 (\Sigma^T)$ sense, and, thus, $f$  is a  strong solution to \eqref{eq1.0} - \eqref{eq1.0.0} (see Definition \ref{definition 6.1}). Finally, the `energy' estimate \eqref{eq2.1.20} is obtained via the same limiting argument.
    
    To show the existence with the initial condition $f_0 \in S_{2, \theta} (\Omega \times \bR^3)$ satisfying the SRBC, we reduce the problem to the case when $f_0 \equiv 0$ by replacing $f (z)$ with $\sff (z) = f (z) - \phi (t) f_0 (x, p)$, where $\phi \in C^{\infty}_0 (\bR)$ such that $\phi (0) = 1$. We note that $\sff$ satisfies the identities
    \begin{align*}
   & \partial_t \sff  + v (p) \cdot \nabla_x \sff -  \nabla_p \cdot (\sigma_g \nabla_p \sff) + b \cdot \nabla_p \sff + (c + \lambda) \sff = \widetilde \eta,\\
   &\sff^{*}_{-} (t, x, p) = \sff^{*}_{+} (t, x, R_x p), z \in \Sigma^T_{-}, \quad \sff (0, \cdot) \equiv 0,
\end{align*}
where
\begin{align*}
    \widetilde \eta = \eta  + \phi' f_0 +  \phi \big(v (p) \cdot \nabla_x f_0) - \nabla_p \cdot (\sigma_g \nabla_p f_0) + b \cdot \nabla_p f_0 + (c + \lambda) f_0\big). 
\end{align*}
Since the $L_{\infty}$-norms of  $\sigma_g, \nabla_p \sigma_g,  b, c$ are bounded by $N$ (see 
\eqref{eq6.1.2} and \eqref{eq1.2.2}), we have 
$$
    \|\widetilde \eta\|_{  L_{2, \theta} (\Sigma^T) } \le N \|\eta\|_{  L_{2, \theta} (\Sigma^T) }  + N (1+ \lambda) \|f_0\|_{S_{2,  \theta} (\Sigma^T)}.
$$ 
This concludes the proof of the existence part.

\textit{Uniqueness.}
Let $f$ be a  strong solution to Eq \eqref{eq1.0} with $\eta \equiv 0, f_0 \equiv 0$. Then, we may use a variant of the energy identity for functions satisfying the SRBC (see \eqref{eqC.5.1} in Lemma \ref{lemma C.5}) with $u = f$ and $\phi =  f e^{-2 \lambda' t}$. Integrating by parts in $p$ and using the Cauchy-Schwarz inequality, we get 
\begin{equation}
                \label{eq2.4.14}
    \int_{\Sigma^T} \big(\delta_0 |\nabla_p f|^2 +  (\lambda + \lambda' - N) |f|^2\big) \, dz \le 0,
\end{equation}
where $N = N (K) > 0$ and $\delta_0$ is the ellipticity constant of $\sigma_g$ (see Lemma \ref{lemma 6.1}). Hence, taking $\lambda' > N$ gives $f \equiv 0$. The uniqueness is proved. 
\end{proof}

\begin{proof}[Proof of Remark \ref{remark 2.8}]
Invoke all the notation in the proof of Lemma \ref{lemma 7.4}. We say that $f$ satisfies the mirror extension property if 
\begin{equation}
    \label{eq2.8.10}
    \text{the identity   \eqref{eq2.44}  holds for} \,\, \mathcal{\dtilde U} \,\, \text{on} \,  \bR^7. 
\end{equation}

To show this,  we regularize $f$ by using an approximation scheme $f_n$ defined as in the proof of Proposition \ref{proposition 2.1} (see \eqref{eq2.1.10}). Then we construct $\mathcal{\dtilde U}$ for such $f_n$ and $\eta_n$. Since $f_n \to f$ in the weak* topology of $S_{2, \kappa \theta} (\Sigma^T)$, by passing to the limit in the integral formulation of \eqref{eq2.44} as $n \to \infty$, we conclude that \eqref{eq2.8.10} is true.
\end{proof}

\begin{proof}[Proof of Proposition \ref{proposition 2.3}]
We inspect the proof of Proposition \ref{proposition 2.1}. 
We use a bootstrap method to show that $\mathcal{\dtilde U}$ (see \eqref{eq2.10.1}) is of class $S_r^{N} (\bR^7_T)$ and to estimate $\|\mathcal{\dtilde U}\|_{S_r^{N} (\bR^7_T) }$.
In particular, one needs to use an induction argument with `the base' at $S_2^{N} (\bR^7_T)$. In the induction step, one uses the embedding theorem for $S_r^{N}$ spaces (see Theorem 2.1 in \cite{PR_98}) combined with the $S_r^{N}$-estimate \eqref{eq4.1.4} in Lemma \ref{lemma 4.1} $(ii)$. Cf. the proof of Theorem 1.7 in \cite{DGY_21} on p. 493 - 494.  We point out that the embedding theorem  in \cite{PR_98} is stated for $S_2^{N} (\bR^7_T)$ with $T  = \infty$. Nevertheless, it is easily seen that the case $T < \infty$ is treated by the same method, which involves using the explicit fundamental solution to $\partial_t + v \cdot \nabla_x - \Delta_v$ (cf. Lemma \ref{lemma 14.C.7}).
\end{proof}

\subsection{Strong solutions to steady linear Landau equations}

\label{section 6.2}
In this section, we establish the results analogous to those in Propositions \ref{proposition 2.1} and  \ref{proposition 2.3} for the steady KFP equation \eqref{eq27.1.2}.

\begin{proposition}[steady $S_r$ estimate in the presence of SRBC]
\label{proposition 14.C.6}
Invoke the assumptions of Proposition \ref{proposition 2.1}. In addition, assume that $g, b$,  $c$, and $\eta$ are independent of $t$.
Let $r \in [2, \infty)$. 
Then, there exists a constant
$\theta = \theta (r, \varkappa, \kappa) > 0$
such that if, additionally,
\begin{equation}
    \label{eq14.C.6.1}
    	\eta \in L_{2, \theta} (\Omega \times \bR^3)  \cap   L_{r, \theta} (\Omega \times \bR^3),
     \end{equation}
then, the following assertions hold.

$(i)$ There exists a unique  strong solution $f$ to Eq. \eqref{eq27.1.2}. In addition, $f \in L_2 (\Omega) W^1_{2, \theta} (\bR^3)$.

$(ii)$ For any  strong solution $f$  to Eq. \eqref{eq27.1.2} such that  $f\in L_2 (\Omega) W^1_{2, \theta} (\bR^3)$,  one has
\begin{align}
    \label{eq14.C.6.3}
f \in S_{2, \kappa \theta} (\Omega \times \bR^3) \cap S_{r, \kappa \theta}  (\Omega \times \bR^3),
 \end{align}
 and,
 \begin{align}
    \label{eq14.C.6.6}
	&\|f\|_{ S_{2, \kappa \theta} (\Omega \times \bR^3) } + \|f\|_{ S_{r, \kappa \theta} (\Omega \times \bR^3) } \\
& \le  N   \bigg(\|\eta\|_{  L_{2, \theta}   (\Omega \times \bR^3) } +   \|\eta\|_{  L_{r, \theta}   (\Omega \times \bR^3) }   + \|f\|_{  L_{2, \theta}   (\Omega \times \bR^3) }\bigg), \notag
 \end{align}
 where  $N = N (\varkappa, \kappa,  r,  \delta_0,  \theta, K,  \Omega)$.

Furthermore, in the case when $r < 6$,  we have
    \begin{align}
        \label{eq14.C.6.7}
        \|f\|_{L_{r_1, \kappa \theta} (\Omega \times \bR^3) } +   \|\nabla_p f\|_{L_{r_2, \kappa \theta} (\Omega \times \bR^3) } \le \text{r.h.s. of \eqref{eq14.C.6.6}},
    \end{align}
    where $r_1, r_2 > 1$ are the numbers satisfying the relations
        \begin{align}
        \label{eq14.C.6.10}
        \frac{1}{r_1} > \frac{1}{r} - \frac{1}{6},  \quad   \frac{1}{r_2} >  \frac{1}{r}  - \frac{1}{12}.
        \end{align}
In the case when $r \in (6,  12)$,
\begin{align}
    \label{eq14.C.6.8}
    \|f\|_{ L_{\infty, \kappa \theta} (\Omega \times \bR^3)  }  + \|\nabla_p f\|_{L_{r_2, \kappa \theta} (\Omega \times \bR^3) } \le \text{r.h.s. of \eqref{eq14.C.6.6}},
\end{align}
where $r_2$ satisfies \eqref{eq14.C.6.10}.
Finally, in the case when $r > 12$,
 \begin{align}
    \label{eq14.C.6.9}
& \|[f, \nabla_p f]\|_{ L_{\infty, \kappa \theta} (\Omega\times \bR^3)  } +	\|[f, \nabla_p f]\|_{  C^{\alpha/3,  \alpha}_{x, p} (\Omega \times \bR^3) } \\
& \le \text{r.h.s. of \eqref{eq14.C.6.6}},  \notag
\end{align}
where $\alpha \in  (0, 1 - \frac{12}{r})$. In  \eqref{eq14.C.6.7}, \eqref{eq14.C.6.8}, and \eqref{eq14.C.6.9}, one needs to take into account the dependence of $N$  on the additional parameters such as $r_1, r_2$ and $\alpha$.
\end{proposition}

\begin{proof}[Proof of Proposition \ref{proposition 14.C.6}]
We repeat the argument  of Propositions \ref{proposition 2.1} - \ref{proposition 2.3}
with the following modifications:
\begin{itemize}
    \item[--] One needs to use the steady counterparts of Theorem 2.6 and Corollary 2.8 in \cite{DY_21a} (see Remark 2.11 therein).
    \item[--] The estimates \eqref{eq14.C.6.7} and \eqref{eq14.C.6.8}-\eqref{eq14.C.6.9} are proved by using the embedding results in Lemma \ref{lemma 14.C.7}.
\end{itemize}
\end{proof}

\begin{corollary} 
   \label{corollary 2.5}
For any $\kappa \in (0, 1)$, there exists $\theta  = \theta (\kappa, r) > 0$ such that for any $f \in S_{r, \theta} (\Omega \times \bR^3)$ satisfying the SRBC, the following assertions hold.

$(i)$ If $r \in [2, 7)$, we have
\begin{align}
    \label{eq2.5.1}
    \|f\|_{L_{r_1, \kappa \theta} (\Omega \times \bR^3) } +   \|\nabla_p f\|_{L_{r_2, \kappa \theta} (\Omega \times \bR^3) } \lesssim_{\theta,  \kappa,    r, r_1, r_2, \Omega} \|f\|_{S_{r, \theta} (\Omega \times \bR^3)},
\end{align}
where $r_1$ and $r_2$ are numbers satisfying \eqref{eq14.C.6.10}.

$(ii)$ If $r \in (6, 12)$, 
\begin{align}
        \label{eq2.5.2}
 \|f\|_{ L_{\infty, \kappa \theta} (\Omega \times \bR^3)  }  + \|\nabla_p f\|_{L_{r_2, \kappa \theta} (\Omega \times \bR^3) } \lesssim_{\theta, \kappa,   r,  r_2, \Omega} \|f\|_{S_{r, \theta} (\Omega \times \bR^3)}, 
\end{align}
where $r_2$ satisfies \eqref{eq14.C.6.10}.

 $(iii)$ If $r > 12$, then, for any $\alpha \in (0, 1-12/r)$, we have
 \begin{align}
            \label{eq2.5.3}
\|[f, \nabla_p f]\|_{   C^{\alpha/3, \alpha}_{x, p} (\Omega \times \bR^3) } \lesssim_{\theta, \kappa,  r, \alpha, \Omega} \|f\|_{S_{r} (\Omega \times \bR^3)}.
 \end{align}
 \end{corollary}

\begin{proof}[Proof of Corollary \ref{corollary 2.5}]
Let
$$
  \eta: =   \frac{p}{p_0} \cdot \nabla_x f - \nabla_p \cdot (\sigma_g \nabla_p f)
$$
and note that $\eta \in L_{r, \theta} (\Omega \times \bR^3)$.
Since $f$ has the mirror extension property (see Remark \ref{remark 2.8} and its proof on p. \pageref{eq2.8.10}), the function  $\mathcal{\dtilde U}$ satisfies (cf. \eqref{eq2.44})
\begin{align*}
& v \cdot \nabla_y \, \mathcal{\dtilde U} - \nabla_v \cdot (\mathfrak{A} \nabla_v \, \mathcal{\dtilde U})\\
&   = \big(\nabla_v \cdot (\mathbb{X} \, \mathcal{\hathat U}) 
	+   \mathcal{\hathat H}\big) \mathsf{J}_{\mathcal{W}} \\
	& -\nabla_v \cdot \big(\mathbb{A}  (\nabla_v \mathsf{J}_{\mathcal{W}}) \, \mathcal{\hathat U}\big)
	- \big(v \cdot \nabla_y \mathsf{J}_{\mathcal{W}}\big) \,\,  \mathcal{\hathat U}.
\end{align*}
Then, applying the steady  $S_r^{N}$ estimate in Proposition \ref{proposition 14.C.6} to the above equation and using the embedding theorem for the steady $S_r^{N}$ spaces (see Lemma \ref{lemma 14.C.7}), and going back to the original variables as in the proof of Lemma \ref{lemma 7.4}, we obtain the desired estimates \eqref{eq2.5.1} - \eqref{eq2.5.3}.
\end{proof}

\subsection{Finite energy solutions to unsteady KFP equations}
\label{section 6.3}

The goal of this section is to establish the existence and uniqueness result for the unsteady linear Landau equation \eqref{eq1.0} - \eqref{eq1.0.0} in the class of finite energy solutions (Definition \ref{definition 27.1}). In particular, we employ a duality argument to prove the uniqueness and utilize an approximation argument to establish existence. The well-posedness result is used to prove  Lemma  \ref{lemma 6.8} about differentiating finite energy solutions in $t$. This lemma plays a  crucial role in demonstrating the temporal differentiability of the nonlinear RVML system (see the assertion $(ii)$ in Theorem \ref{theorem 5.1} and $a)-b)$ in Proposition  \ref{proposition 7.2}).

\begin{proposition}
\label{proposition 27.4}
We invoke the assumptions of Proposition \ref{proposition 2.1}
and assume, additionally, that
\begin{align}
    \label{eq27.4.1}
    \|\nabla_p \cdot b\|_{ L_{\infty} (\Sigma^T) } \le K.
\end{align}
Then, for any $\theta \ge 0$ and  
$$
    f_0 \in L_{2, \theta} (\Sigma^T), \quad \eta \in L_2 ((0, T) \times \Omega) W^{-1}_{2, \theta} (\bR^3),
$$
there exists a unique finite energy  solution to Eq. \eqref{eq1.0} - \eqref{eq1.0.0} (see Definition \ref{definition 27.1}).  In addition, for any $t \in (0, T)$, $f$ satisfies the energy identity
\begin{align}
        \label{eq27.4.12}
   & \int_{\Omega \times \bR^3} \big(f^2 (t, x, p)  -  f^2_0 (x, p)\big) p_0^{2 \theta}\, dx dp \\
   & + \int_{\Sigma^t} (\nabla_p f)^T \sigma_g \nabla_p (f p_0^{2 \theta}) + \big((b \cdot \nabla_p f) f + (c+\lambda) f^2\big)  p_0^{2 \theta}\, dz = \int_{(0, t)  \times \Omega }  \langle \eta, f p_0^{2 \theta} \rangle \notag \, dx d\tau, \notag
\end{align}
where  $\langle  \cdot, \cdot \rangle$ is defined in \eqref{eq1.2.16}.
A similar result holds for the steady equation \eqref{eq27.1.2}.
\end{proposition}

Before we prove Proposition \ref{proposition 27.4}, we first establish the uniqueness in the class of very weak solutions defined below.

\begin{definition}[very weak solution]
\label{definition 27.3}
        We say that $f$  is a very weak  solution to \eqref{eq1.0} - \eqref{eq1.0.0} if
    $$
        f \in  L_2 (\Sigma^T),
    $$
    and for any test function $\phi \in S_2 (\Sigma^T)$ satisfying SRBC and $\phi (T, \cdot) \equiv 0$, we have
    \begin{align}
        \label{eq27.3}
    &  - \int_{\Omega \times \bR^3}  f_0 (x, p) \phi (0, x, p) \, dx dp \\
    &+ \int_{\Sigma^T}  f \bigg(- Y \phi- \nabla_p \cdot (\sigma_g \nabla_p \phi) 
    -   f \nabla_p  \cdot  (b  \phi)  + c \phi\bigg) \, dz =    \int_{(0, t)  \times \Omega }  \langle \eta, \phi \rangle \notag \, dx d\tau. \notag
\end{align}
\end{definition}

\begin{remark}
    \label{remark 27.3.1}
We note that due to Lemma \ref{lemma C.5}, any test function $\phi$ in Definition \ref{definition 27.3} belong to $C ([0, T]) L_2 (\Omega \times \bR^3)$. Hence, any finite energy solution (see Definition \ref{definition 27.1}) is a very weak solution provided that $b$ is sufficiently regular. See also Remark \ref{remark 5.9} for the comparison with other notions of weak solutions used in this paper.
\end{remark}

\begin{lemma}[Uniqueness of very weak solutions]
    \label{lemma 27.2}
We invoke the assumptions of Proposition \ref{proposition 2.1} and assume, additionally, that $\nabla_p \cdot b \in L_{\infty} (\Sigma^T)$.
Then, the uniqueness holds for the problem \eqref{eq1.0} -\eqref{eq1.0.0}  in the class of very weak solutions.
\end{lemma}

\begin{proof}
Assume that $u^{(j)}, j = 1, 2,$ are  very weak solutions  to \eqref{eq1.0} -\eqref{eq1.0.0} and denote $u = u^{(1)} -u^{(2)}$.
Then, for any function $\phi \in S_2 (\Sigma^T)$ satisfying SRBC and the condition $\phi (T, \cdot) \equiv 0$,  we have
\begin{align*}
    \int_{\Sigma^T}  u \bigg(-Y\phi - \nabla_p \cdot  (\sigma_g \nabla_p \phi)  -(\nabla_p \cdot b) \phi - b \cdot \nabla_p \phi + (c+\lambda) \phi\bigg)  \, dz = 0.
\end{align*}
Let $\zeta \in C^{\infty}_0 (\bR^3)$ be a nonnegative function such that $\zeta = 1$ on $B_1 (0)$ and denote $\zeta_n (\cdot) = \zeta (\cdot/n), n > 0$.
We consider the equation 
\begin{align*}
       &    - Y \phi_n  -  \nabla_p \cdot (\sigma_g \nabla_p \phi_n)  - b \cdot \nabla_p \phi_n + (c + \lambda - \nabla_p \cdot b)  \phi_n = u \zeta_n, \\
       & \phi_n (T, \cdot) \equiv 0, \quad  \phi_n (t, x, p) = \phi_n (t, x, R_x p), \, z \in \Sigma^T_{-}.
\end{align*}
Since $u \zeta_n \in L_{2, \theta} (\Sigma^T)$  for any $\theta > 0$,
 by Proposition \ref{proposition 2.1} the above equation has a unique strong solution  $\phi_n \in S_{2} (\Sigma^T)$.
Then, we have 
$$
    \int_{\Sigma^T} u^2 \zeta_n \, dz  = 0,
$$
and by nonnegativity of $\zeta$, we have $u^2  \zeta_n = 0$ a.e. Since $n$ is arbitrary, we conclude $u \equiv 0$, as desired.
\end{proof}

\begin{proof}[Proof of Proposition \ref{proposition 27.4}]
The uniqueness follows from Remark \ref{remark 27.3.1} and Lemma \ref{lemma 27.2}. 

\textit{Existence.}
For the sake of clarity, we will only consider the case when $\theta =  0$, as the case when $\theta > 0$ is handled by the same argument.
The proof is split into two steps.

We will need an auxiliary notion of finite energy solutions, which we call intermediate finite energy solutions.

\textbf{Step 1: construction of an intermediate finite energy solution.} 
\begin{definition}
    \label{definition 27.6}
    We say that $f$ is an intermediate finite energy  solution if 
$$
    f \in L_{\infty} ((0, T)) L_2 (\Omega \times \bR^3) \cap L_2 ((0, T) \times \Omega) W^1_2 (\bR^3),
$$
and for any test function $\phi$ satisfying the conditions \eqref{eq27.2.1} - \eqref{eq27.2.3},  and
$\phi (T, x, p)  \equiv 0$
(see Remark \ref{remark 27.7}), one has
\begin{align}
    \label{eq27.5}
    &- \int_{ \Sigma^T }  f (Y \phi) \, dz   - \int_{\Omega \times \bR^3}  f_0 (x, p) \phi (0, x, p) \, dx dp \\
    &+ \int_{ \Sigma^T } \bigg((\nabla_p \phi)^T \sigma_g \nabla_p f 
    + (b \cdot \nabla_p f) \phi + (c+\lambda) \phi \bigg) \, dz = 
    \int_{ \Sigma^T } \langle \eta, f \rangle \, dx d\tau \notag
\end{align}
(see \eqref{eq1.2.16}). 
\end{definition}
We refer to Remark \ref{remark 5.9} for a review of various definitions of weak solutions employed throughout this paper.

Proof by approximations and weak* compactness. 
Let $f_{0, n}, n \ge 1$ be a sequence of functions such that
\begin{align}
           \label{eq27.4.3}
     &   f_{0, n} \in S_{2, 2 \theta} (\Omega \times \bR^3), \quad f_{0, n} \, \text{satisfies SRBC}, \\
     & f_{0, n}  \,  \to f_0 \, \, \text{in} \, \, L_{2} (\Omega \times \bR^3), \notag
\end{align}
where $\theta$ is large. For example, one can choose $f_{0, n} \in C^{\infty}_0 (\Omega \times \bR^3)$ such that $f_{0, n} \to f_0$ in $L_{2} (\Omega \times \bR^3)$, so that both   conditions in \eqref{eq27.4.3} are   satisfied.

Furthermore, let  $\zeta, \xi \in C^{\infty}_0 (\bR^3)$ be functions such that $\int_{\bR^3} \zeta \, dp = 1$ and $\xi  = 1$ on $B_1 (0)$.
We set 
$$
       \zeta_n (p) =  n^{-3} \zeta (p/n), \quad \xi_n (p) = \xi (p/n).
$$ 
For any function $h \in L_{1, loc} (\bR^7_T)$, we denote
$$
    h_{(n)} (t, x, p) =  (h \ast_p \zeta_n) (t, x, p).
$$

Next,  let $\eta_0, \bm{\eta}_1$ be any $L_2 (\Sigma^T)$ functions such that
\begin{align}
    \label{eq27.4.10}
    \eta = \eta_0 + \nabla_p \cdot \bm{\eta}_1.
\end{align}
Since \eqref{eq27.4.3} is valid and one has
 \begin{align*}
     &  \xi_n \eta_0, \,  \nabla_p \cdot (\xi_n \bm{\eta}_1)_{(n)} \in L_{2, \theta} (\Sigma^T), \quad \forall \theta > 0,
 \end{align*}
by Proposition \ref{proposition 2.1} $(i)$  there exists a unique  strong solution $f_n \in S_2 (\Sigma^T)$ to the  equation
   \begin{align}
           \label{eq27.4.4}
 &	Y f_{n}  - \nabla_p \cdot (\sigma_g \nabla_p f_{n}) + b \cdot \nabla_p f_{n} + (c + \lambda)  f_{n} = \xi_n \eta_0 + \nabla_p \cdot (\xi_n  \bm{\eta}_1)_{(n)}, \\
 &f_{n} (t, x, p) = f_{n} (t, x, R_x p), \, z \in \Sigma^T_{-}, \quad f_n (0, \cdot) = f_{0, n} (\cdot). \notag
  \end{align}
By using the energy identity \eqref{eqC.5.1}, integration by parts, and the Cauchy-Schwarz inequality, we get
\begin{align}
       \label{eq27.4.6}
     &  \|f_n\|_{ L_{\infty} ((0, T)) L_{2} (\Omega \times \bR^3)  } + \|f_n\|_{ L_2 ((0, T) \times \Omega) W^1_2 (\bR^3)}  \\
     & \le N   \|f_n (0, \cdot)\|_{ L_{2} (\Omega \times \bR^3) }  + N \||\eta_0| + |\bm{\eta}_1|\|_{L_2 (\Sigma^T)}, \notag 
\end{align}
where $N = N (\delta_0, K, T) > 0$.

By the weak* compactness argument, there exists a function  $u$  and a subsequence  $n'$ such that
\begin{align*}
&
u \in L_{\infty} ((0, T)) L_{2} (\Omega \times \bR^3) \cap L_2 ((0, T) \times \Omega) W^1_2 (\bR^3),\\
    & f_{n'} \to u \quad \text{in the weak* topology of} \, L_{\infty} ((0, T)) L_{2} (\Omega \times \bR^3),\\
&  f_{n'} \to u \quad \text{in the weak topology of} \, L_2 ((0, T) \times \Omega) W^1_2 (\bR^3).
\end{align*}
Hence, by passing to the limit in the integral formulation \eqref{eq27.5} of  Eq. \eqref{eq27.4.4},    we conclude that $u$ satisfies the integral formulation \eqref{eq27.5} with $u$ in place of $f$. Thus, $u$ is an intermediate finite energy  solution to \eqref{eq1.0} - \eqref{eq1.0.0}, and, by the uniqueness in Lemma \ref{lemma 27.2}, we conclude $f = u$. 
Taking $\text{liminf}$ in  \eqref{eq27.4.6} and then infimum over all $\eta_0, \bm{\eta}_1 \in L_2 (\Sigma^T)$ satisfying \eqref{eq27.4.10},  we obtain the estimate 
\begin{align}
    \label{eq27.4.9}
    &  \|f\|_{ L_{\infty} ((0, T)) L_{2} (\Omega \times \bR^3)  } + \|f\|_{ L_2 ((0, T) \times \Omega) W^1_2 (\bR^3)}  \\
     & \le N   \|f (0, \cdot)\|_{ L_{2} (\Omega \times \bR^3) }  + N \|\eta\|_{L_2 ((0, T) \times \Omega)  W^{-1}_2 (\bR^3) }, \notag 
\end{align}
where $N = N (\delta_0, K, T)$.

\textbf{Step 2: existence of a finite energy solution.} We first show that 
   \begin{align}
          \label{eq27.4.8}
    f_{n} \to f\, \text{strongly in} \,  L_{\infty} ((0, T))  L_{2} (\Omega \times \bR^3) \, \text{and in} \,   L_2 ((0, T) \times \Omega) W^1_2 (\bR^3).
  \end{align}
We note that $w_n  = f_n  - f$ is an intermediate finite energy  solution to
   \begin{align}
          \label{eq27.4.7}
 &	Y w_{n}  - \nabla_p \cdot (\sigma_g \nabla_p w_{n}) + b \cdot \nabla_p w_{n} + (c+\lambda)  w_{n}\\
 & = (\xi_n - 1) \eta_0 +  \nabla_p \cdot \big((\xi_n \bm{\eta}_1)_{(n)} -  \bm{\eta}_1\big), \notag \\
 &w_{n} (t, x, p) = w_{n} (t, x, R_x p), \, z \in \Sigma^T_{-}, \quad  w_n (0, \cdot) = f_{0, n} (\cdot) - f_0 (\cdot). \notag
  \end{align}
By the estimate \eqref{eq27.4.9} obtained in Step 1, we have
\begin{align*}
     &  \|w_n\|_{ L_{\infty} ((0, T)) L_{2} (\Omega \times \bR^3)  } + \|w_n\|_{ L_2 ((0, T) \times \Omega) W^1_2 (\bR^3)}  \\
     & \le N   \|f_{0, n}  - f_0\|_{ L_{2} (\Omega \times \bR^3) }\\
     & + N \||\eta_0 - \xi_n \eta_0| + |\bm{\eta}_1 - (\xi_n \bm{\eta}_1)_{(n)}|\|_{L_2 (\Sigma^T)}.
\end{align*}
Passing to the limit, we prove \eqref{eq27.4.8}.
Since $f_n \in S_2  (\Sigma^T)$, $f_{0, n} \in L_2 (\Omega \times \bR^3)$, and $f_n$ satisfies SRBC,  by Lemma \ref{lemma C.5}, we have  $f_n \in C ([0, T])  L_{2} (\Omega \times \bR^3)$. Then, due to the convergence \eqref{eq27.4.8}, we conclude $f \in C ([0, T])  L_{2} (\Omega \times \bR^3)$.

Finally, we prove the validity of the weak formulation \eqref{eq27.2}.
We fix arbitrary $t \in [0, T]$. By the energy identity \eqref{eqC.5.1} with $t$ in place of $T$ applied to  Eq. \eqref{eq27.4.4}, we get
\begin{align*}
    &- \int_{ \Sigma^t }  f_n Y  \phi \, dz  + \int_{\Omega \times \bR^3} (f_n \phi) (t, x, p) - f_{0, n} (x, p) \phi (0, x, p) \, dx dp \\
    &+ \int_{ \Sigma^t } \bigg((\nabla_p \phi)^T \sigma_g \nabla_p f_n 
    + b \cdot (\nabla_p f_n) \phi + (c+\lambda) \phi \, dz = \int_{(0, t) \times \Omega}  \eta_0 \xi_n \phi - (\nabla_p \phi) \cdot (\xi_n  \bm{\eta}_1)_{(n)} \, dx d\tau.
\end{align*}
Passing to the limit as $n \to \infty$, we obtain the desired identity \eqref{eq27.2}. Thus, $f$ is the finite energy solution to  \eqref{eq1.0} - \eqref{eq1.0.0}, as desired.
\end{proof}

\section{Proof of Proposition \ref{proposition 5.2}}
                \label{section 7}
                
The section is organized as follows. 
First, in Section \ref{section 7.1} we prove the desired estimate \eqref{eq5.2.4.1} given that the linear RVML system \eqref{eq5.2} - \eqref{eq5.5} is well-posed  and the triple $[f, \bE_f, \bB_f]$ is sufficiently regular.  See the details in Proposition \ref{proposition 7.2}.
We justify the existence, uniqueness, and higher regularity in the proof of Proposition \ref{proposition 7.2} in Appendix  \ref{section 7.2}.

Denote
\begin{align}
 			\label{eq1.3}
   & 
   \sigma_{g^{+}+g^{-}}  = \underbrace{   2  \int_{\bR^3} \Phi (P, Q)  J (q) \, dq}_{= \sigma (p)} +  \int_{\bR^3} \Phi (P, Q)  J^{1/2} (q)  g (t, x, q) \cdot \bm{\xi}_0 \, dq, \\
& \label{eq6.1}
\index{$a_g$} a_g^i (z) =  -  \int \Phi^{ i j } (P, Q)   J^{1/2} (q)    \big(\frac{p_i}{2 p_0}  g (t, x, q)
 	+ \partial_{q_j}  g (t, x, q)\big) \cdot  \bm{\xi}_0\, dq,\\
 \label{eq6.2}
 & \index{$C_g$} C_g (z) = 
	  -  \frac{1}{2} \sigma^{ i j} \frac{ p_i}{  p_{ 0 } } \frac{ p_j}{ p_0} + \partial_{p_i} \big(\sigma^{ i j}  \frac{ p_j}{ p_0}\big) \\
&	-  \int  \big(\partial_{p_i} 	- \frac{ p_i}{2 p_0}\big)   \Phi^{ i j } (P, Q)  J^{1/2} (q)  \partial_{q_j}  g (t, x, q) \cdot \bm{\xi}_0 \, dq\notag,\\
\label{eq7.6}
& K g = - J^{-1/2} (p) \partial_{p_i}  \bigg(J (p) \int \Phi^{ i j } (P, Q)  J^{1/2} (q) \big(\partial_{q_j} g (t, x, q)  \\
&+ \frac{ q_j}{2 q_0} g (t, x, q)\big) \cdot  \bm{\xi}_0 \, dq\bigg) \bm{\xi}_0.\notag
\end{align}

 The following lemma will be used many times in the sequel.
\begin{lemma}
        \label{lemma 7.1}
Under the assumptions of Proposition \ref{proposition 5.2}, we have
\begin{align}
  \label{eq7.1.3}
  & 
    \|\partial_t^k g\|_{ L_{\infty} ((0, T) \times \Omega) W^{1}_{r, \theta/2^{k+  9  } } (\bR^3) }  
 \le N_0  \sup_{\tau \le T} \sqrt{\cI_{g} (\tau)} \le N_0  \sqrt{\varepsilon_0}, \, \, r \in \{2, \infty\},  k = 0, 1, \ldots, m - 8,\\
 \label{eq7.1.3.1}
 &  \|\partial_t^k g\|_{ L_{\infty} ((0, T)) C^{\alpha/3, \alpha}_{x, p} (\Omega \times \bR^3) }  
 \le N_0  \sup_{\tau \le T} \sqrt{\cI_{g} (\tau)}   \le N_0   \sqrt{\varepsilon_0}, \, \,  k = 0, 1, \ldots, m - 8,\\
  \label{eq7.1.4}
  & \|\partial_t^k [\bE_g, \bB_g]\|_{  L_{\infty} ((0, T) \times \Omega) } \le  N_0 \sqrt{\varepsilon_0}, \, \,  k = 0, 1, \ldots,  m- 7, \\
 \label{eq7.1.5}
  &\|\partial_t^k a_g\|_{  L_{\infty} (\Sigma^T) }  \le N_0 \sqrt{\varepsilon_0}, \, k = 0, 1, \ldots, m - 8, \\
 \label{eq7.1.6}
  &  \|| \sigma_{g^{+}  + g^{-} }| + |\nabla_p  \sigma_{g^{+}  + g^{-} }| + |C_g|\|_{  L_{\infty} (\Sigma^T) } \le N_0, \\
     \label{eq7.1.7}
    &\|\partial_t^k  \sigma_{g^{+}  + g^{-} }| + |\partial_t^k \nabla_p \sigma_{g^{+}  + g^{-} }| + |\partial_t^k C_g|\|_{   L_{\infty} (\Sigma^T) }  \le N_0 \sqrt{\varepsilon_0}, \,  k = 1, \ldots, m-8, \, 
\end{align}
where $\alpha \in (0, 1-12/r_4)$. Furthermore, for $h = [\sigma_{g^{+}  + g^{-} }, \nabla_p \sigma_{g^{+}  + g^{-} }, C_g, a_g]$,  and $i \in \{1, \ldots, 4\}$,  
\begin{align} 
    \label{eq7.1.8}
 \| \partial_t^k   h \|_{  L_{\infty} ((0, T))    L_{r_i} (\Omega) L_{\infty} (\bR^3)   }
     \le N_0 \sqrt{\varepsilon_0}, \, \, 
     k \le m - 4 -  i,
\end{align}
where $r_i, i = 1, \ldots, 4,$ are given by \eqref{eq5.10.10}, 
and $N_0  = N_0 (r_1, \ldots, r_4, \theta, \Omega, \alpha, m).$
\end{lemma}

\begin{proof}
In this proof $N_0 = N_0 (r_1, \ldots, r_4, \theta, \Omega, \alpha, m)$ might change from line to line.
We note that by the definition of $\cH_g (T)$ in \eqref{eq5.10.1}   and the assumption  $y_g (T) < \varepsilon_{ 0 }$ (see \eqref{eq5.2.4.2}), 
the fact that $r_4 > 12$, and the  embedding theorem for  steady $S_p$ spaces (see \eqref{eq2.5.3}), we have for $k \le \, m-8$ and $r \in \{2, \infty\}$, 
\begin{align}
    \label{eq7.1.9}
   &\|\partial_t^k g\|_{ L_{\infty} ((0, T) \times \Omega) W^{1}_{ r,  \theta/2^{k+ 9  } } (\bR^3) }  
     \lesssim_{r_4, \theta, k, \Omega}  \|\partial_t^k g\|_{  L_{\infty} ((0, T)) S_{r_4, \theta/2^{k+8} } (\Omega \times \bR^3) } \\
   & \le N_0 \sqrt{\cH_g (T)} \le   N_0 \sqrt{\varepsilon_0}. \notag
\end{align}
By the same embedding result, we obtain \eqref{eq7.1.3.1}. The estimate of the $L_{\infty}$ norm  $\partial_t^k [\bE_g, \bB_g]$ follows  from the fact that $r_3 > 3$ (see \eqref{eq5.10.2}), the Sobolev embedding theorem, the definition of $\cH_g (T)$ (see \eqref{eq5.10.1}), and the smallness assumption \eqref{eq5.2.4.2}. 
Furthermore, using the identities \eqref{eq6.1} -  \eqref{eq7.6}, the estimate  \eqref{eqB.2.1} with $r = \infty$ in  Lemma \ref{lemma B.2}, and the bound \eqref{eq7.1.3} with $r = \infty$, we obtain \eqref{eq7.1.5} -  \eqref{eq7.1.7}.

Finally, using the bound   \eqref{eqB.2.1}  again, we get
for fixed $t, x$,
\begin{align*}
     \| \partial_t^k   h (t, x, \cdot) \|_{     L_{\infty}   (\bR^3) }
     \le N_0 |\partial_t^k [\bE_g, \bB_g] (t, x)|  + N_0 \|\partial_t^k [g, \nabla_p g] (t, x, \cdot)\|_{   L_{r_i} (\bR^3)   }.
\end{align*}
Taking the $L_{\infty} ((0, T))  L_{r_i} (\Omega)$ norm and invoking the definition of $\cH_g$ in \eqref{eq5.10.1}, we get for  $k \le  m-4-i$ (cf. \eqref{eq7.1.9}),
\begin{align*}
  &  \text{l.h.s of \eqref{eq7.1.8}} \\
  &   \le N_0 \|\partial_t^k [\bE_g, \bB_g]\|_{ L_{\infty} ((0, T))   L_{r_i} (\Omega)   } + N_0 \|\partial_t^k [g, \nabla_p g]\|_{  L_{\infty} ((0, T))   L_{r_i} (\Omega \times \bR^3)   }
     \le N_0 \sqrt{\varepsilon_0}.
\end{align*}
\end{proof}

 The next lemma asserts the well-posedness of the linear RVML system. We will prove it in Appendix \ref{section 7.2} (see p. \pageref{eqF.1}).
\begin{proposition}
    \label{proposition 7.2}
    Under the assumptions of Proposition \ref{proposition 5.2}, there exist a triple $[f, \bE_f, \bB_f]$
    such that
\begin{enumerate}
    \item[a)] $\partial_t^k f, k \le m-5,$ is a  strong solution (see Definition \ref{definition 6.1}) to the linear Landau equation \eqref{eq5.2} differentiated $k$ times with respect to $t$ with the initial condition $\partial_t^k f = f_{0, k}$ (see \eqref{eq3.3.2}),
    \item[b)] $\partial_t^k f, m-4 \le k \le m,$ is a finite energy  solution to \eqref{eq5.2}  (see Definition \ref{definition 27.1}) differentiated $k$ times with respect to $t$ with the initial condition $f_{0, k}$,
    \item[c)] $\partial_t^k [\bE_f, \bB_f], k \le m-1,$ is a  strong solution to  Maxwell's equations  \eqref{eq5.2.0}-\eqref{eq5.3}  differentiated $k$ times with respect to $t$ with the perfect conductor BC and the initial condition $[\bE_{0, k}, \bB_{0, k}]$ (see \eqref{eq3.3.3} - \eqref{eq3.3.4}), whereas $\partial_t^m f$ is a  weak solution to differentiated Maxwell's equations,
    
    \item[d)] for any $k \le m$, we have $\partial_t^k [\nabla_x \cdot  \bE_f, \nabla_x \cdot  \bB_f] = \partial_t^k [\rho_f, 0]$ (cf. \eqref{eq5.4}),

   \item[e)]
       \begin{align}
        \label{eq7.90}
       \partial_t^k f (t, \cdot) \in  W^1_{\infty, \theta/2^{k+9}} (\Omega \times \bR^3), k \le m-8,  \, \text{for any} \,  t \in (0, T], 
    \end{align}
\end{enumerate}
    Furthermore, any two triples $[f^{(j)}, \bE_f^{(j)}, \bB_f^{(j)}], j = 1, 2,$ satisfying $a)-e)$ must coincide. 
\end{proposition}

The next result shows that given the energy-dissipation control (see \eqref{eq5.10.15} - \eqref{eq5.10.16}), one can establish the higher-regularity control by estimating $\cH_f (T)$ (see \eqref{eq7.96}).

\begin{proposition}
    \label{proposition 7.3}
Assuming that Proposition \ref{proposition 7.2} is valid, we have
\begin{align}
    \label{eq7.96}
 &  \cH_f (T)
 =   \sum_{i=1}^4 \sum_{k=0}^{  m-4-i } \|\partial_t^k f (\tau, \cdot)\|^2_{  S_{r_i, \theta/2^{ k + 2i   } } (\Omega \times \bR^3) }  \\
& + \sum_{k=0}^{m-1} \|\partial_t^k  [\bE_f, \bB_f] (\tau, \cdot)\|^2_{ W^1_{ 2 } (\Omega) } 
+ \sum_{i=2}^3 \sum_{k=0}^{ m  - 4 -  i} \|\partial_t^k  [\bE_f, \bB_f] (\tau, \cdot)\|^2_{ W^1_{r_i} (\Omega) } 
  \le N \varepsilon_0 \sup_{\tau \le T} \cI_f (\tau)  \notag \\
&  + N  \sum_{k=0}^{m-4} \bigg(\|\partial_t^k f\|^2_{  L_{\infty} ((0, T)) L_{2, \theta/2^k} (\Omega \times \bR^3) } 
 +  \|\partial_t^k [\bE_f, \bB_f] \|^2_{ L_{\infty} ((0, T))  L_{ 2 } (\Omega) }\bigg), \notag
\end{align}
where $N = N (r_1, \ldots, r_4, \alpha, \Omega, \theta, m)$.
\end{proposition}

\begin{proof}
   Here we estimate the  functional $\cH_f (T)$ (see \eqref{eq5.10.1}).
For the sake of clarity, we assume, additionally,   
\begin{align}
   \label{eq7.91}
   \cH_f (T) < \infty,
\end{align}
which is used to perform the descent argument (see Section \ref{section 4}).
This assumption will be removed at the end of this step.

First, we differentiate Maxwell's equations formally $k$ times in the $t$ variable and rewrite them as two systems of div-curl type as in \eqref{eq10.4} - \eqref{eq10.5}.
By the $W^1_{r_i}$ div-curl estimate (see \eqref{eq3.0.0}),   we have 
\begin{align}
    \label{eq7.92}
  & \|\partial_t^k [\bE_f, \bB_f]\|_{ L_{\infty} ((0, T)) W^1_{ r_{i} } (\Omega) } \\
  & \lesssim_{\Omega} 
  \sum_{l=k}^{k+1}  \|\partial_t^l [\bE_f, \bB_f]\|_{ L_{\infty} ((0, T)) L_{ r_{i} } (\Omega) }
  + \|\partial_t^{k} f\|_{ L_{\infty} ((0, T)) L_{ r_{i} } (\Omega \times \bR^3) }, \notag
\end{align}
where $k \le m-1$ if $i = 1$ $(r_1 = 2)$, and $k \le m-4-i$ if $i \in  \{2, 3\}$.

 Next,  differentiating the equation \eqref{eq5.2} formally and using the expressions of $A$ and $\Gamma (f, g)$ in \eqref{eqG.1.1} and \eqref{eqG.1.3}, and those of $\sigma_{g^{+}  + g^{-} }$, $a_g$, and  $C_g$ (see \eqref{eq1.3} - \eqref{eq6.2}), we conclude that  for each $t$, the function $u (t, \cdot) =  \partial_t^k f (t, \cdot)$, $k \le m-5,$ is a  strong solution to the  `steady' equation
\begin{align} 
    \label{eq7.48}
    & \frac{p}{p_0}  u  -  \nabla_p \cdot (\sigma_{g^{+}  + g^{-} }  \nabla_p u) 
    +  \bm{\xi} (\bE_g + v (p) \times \bB_g - a_g) \cdot \nabla_p u  \\
    & +(C_g - \frac{\bm{\xi}_1}{2}  v (p) \cdot \bE_g) u \notag \\
    &
        = -\partial_t^{k+1} f  + \bm{\xi} (v (p) \cdot \partial_t^k \bE_f) J^{1/2} 
          +   K (\partial_t^k f) 
        + 1_{k > 0} \sum_{j=1}^3 \sum_{k_1+k_2 = k, k_1 \ge 1 }\eta_{k_1, k_2}^j,  \notag\\
      \label{eq7.48.1}
    & u (t, x, p) = u (t, x, R_x p), (x, p) \in \gamma_{-},  \\
     & \eta^1_{k_1, k_2} = -\bm{\xi} \partial_t^{k_1}(\bE_g + v (p) \times  \bB_g ) \cdot \nabla_p (\partial_t^{k_2} f) + \frac{\bm{\xi}}{2}  (v (p) \cdot \partial_t^{k_1}  \bE_g) \partial_t^{k_2} f, \notag\\
     & \eta^2_{k_1, k_2} = \big(\partial_{p_i } \partial_t^{k_1} \sigma_{g^{+}  + g^{-} }^{i j}
     - \partial_t^{k_1} a_g^i\big) (\partial_{ p_i} \partial_t^{k_2}  f) + (\partial_t^{k_1} C_g) \partial_t^{k_2} f,  \notag\\
          & \eta^3_{k_1, k_2} =  (\partial_t^{k_1}  \sigma_{g^{+}  + g^{-} }^{i j} ) (\partial_{p_i p_j} \partial_t^{k_2}  f). \notag
\end{align}

We   apply     Proposition \ref{proposition 14.C.6} (see \eqref{eq14.C.6.6}) to $\partial_t^k f^{\pm} (t, \cdot)$ for each $t$ with 
\begin{align*}
    b = \pm (\bE_g + v (p) \times \bB_g) - a_g, \quad c = C_g - \frac{\bm{\xi}}{2}  v (p) \cdot \bE_g.
\end{align*}
We first check its assumptions  \eqref{eq1.4.1} - \eqref{eq1.4.2}, \eqref{eq1.5.1}, \eqref{eq1.2.2}.

We note that \eqref{eq1.4.1} - \eqref{eq1.4.2} in Assumption \ref{assumption 1.4} hold with $K = 1$ due to \eqref{eq7.1.3} - \eqref{eq7.1.3.1} in Lemma \ref{lemma 7.1} provided that $\varepsilon_0$ is sufficiently small.
 Similarly,  \eqref{eq1.2.2} with $K = 1$ follow directly from  
\eqref{eq7.1.6} in Lemma \ref{lemma 7.1}. 
Finally, \eqref{eq1.5.1} is valid due to \eqref{eq5.2.16}.

We  fix $i =  1, \ldots, 4$ and $0 \le k \le  m-4-i$.
Then, by the estimates  \eqref{eq14.C.6.6} - \eqref{eq14.C.6.7} and \eqref{eq14.C.6.8} - \eqref{eq14.C.6.9}   with $\theta/2^{k+2i - 1 }$  in place of $\theta$ and $\kappa = 1/2$ applied for each $t \in [0, T]$, we get  
\begin{align}
 & \label{eq7.37}
  \|  \partial_t^k f\|^2_{ L_{\infty} ((0, T)) S_{r_i, \theta/2^{k + 2 i    }} (\Omega \times \bR^3)   } 
      + 1_{ i < 4}   \|  \partial_t^k f\|^2_{ L_{\infty} ((0, T)) L_{r_{i+1}, \theta/2^{k+ 2 i}} (\Omega \times \bR^3)   }   
    \\
    & \le N  \sum_{s \in \{2, r_i\} }  \|\text{r.h.s. of \eqref{eq7.48}} \|^2_{ L_{\infty} ((0, T))   L_{s, \theta/2^{k +2 i-1 } } (\Omega \times \bR^3) } \notag \\
    & + N \|\partial_t^k f\|^2_{ L_{\infty} ((0, T))   L_{2, \theta/2^{k + 2 i - 1} } (\Omega \times \bR^3) }. \notag
    \end{align}

    Furthermore, 
  \begin{align}  
    \label{eq7.37.1}
  &\sum_{s \in \{2, r_i\} } \|\text{r.h.s. of \eqref{eq7.48}} \|^2_{ L_{\infty} ((0, T)) L_{s, \theta/2^{k    + 2 i - 1  } } (\Omega \times \bR^3) }  \le   \sum_{j = 1}^4 \cI_{j, k},  \\
&
\cI_{j, k} =  1_{k > 0} \sum_{s \in \{2, r_i\} } \sum_{k_1+k_2 = k, k_1 \ge 1 } \| \eta_{k_1, k_2}^j\|^2_{  L_{\infty} ((0, T)) L_{s, \theta/2^{k  + 2 i -1 } } (\Omega \times \bR^3) }, j  = 1, 2, 3, \notag\\
&
 \cI_{4, k} =  \sum_{s \in \{2, r_i\} } \|\partial_t^{k+1} f\|^2_{  L_{\infty} ((0, T)) L_{s, \theta/2^{k + 2 i-1} } (\Omega \times \bR^3) },    \notag \\
& 
\cI_{5, k} =  \| \partial_t^k \bE_f\|^2_{  L_{\infty} ((0, T)) L_{  r_i } (\Omega)  }, \notag \\
& 
 \cI_{6, k} =  \sum_{s \in \{2, r_i\} }  \| K (\partial_t^k f)\|^2_{  L_{\infty} ((0, T)) L_{s, \theta/2^{k + 2 i-1} }  (\Omega \times \bR^3) }.  \notag
\end{align}
We chose  weights $\theta/2^{k+2i}$ because for each $i = 1, \ldots, 4,$ and $k \le m-4-i$, we need to compensate for 
\begin{enumerate}
    \item[a)] the  `natural' weight loss in the steady $S_p$ estimate with $\kappa = 1/2$ (see \eqref{eq14.C.6.6} - \eqref{eq14.C.6.7} in Proposition \ref{proposition 14.C.6}),

   \label{b)}
    \item[b)] the presence of the term $\partial_t^{k+1} f \in L_{2, \theta/2^{k+1}} (\Sigma^T)$ on the r.h.s. of \eqref{eq7.48}, which has a worse decay than $\partial_t^k f$.
\end{enumerate}
Loosely speaking,  due $a)-b)$, for each $i$, the loss factor in the weight parameter is $\frac{1}{4}$, which leads to the factor $2^{-2i}$ in the `hierarchy of weights.'


\textit{Estimates of $\cI_{1, k}$ and $\cI_{2, k}$.} We will show that
\begin{align}
        \label{eq7.40}
& \cI_{1, k} + \cI_{2, k} \le N \varepsilon_0 \sum_{s \in \{2, r_i\} } \sum_{k_2 =  0}^{k-1} \|\partial_t^{k_2} [f, \nabla_p f]\|^2_{ L_{\infty} ((0, T)) L_{s, \theta/2^{k_2 + 2 i  } }  (\Omega \times \bR^3) } \\ 
& +  N \varepsilon_0 1_{ k \ge m -  7   }  \sum_{k_2 =  0}^{   2 } \|\partial_t^{k_2} [f, \nabla_p f]\|^2_{  L_{\infty} ((0, T)) L_{\infty, \theta/2^{k_2 + 9 } } (\Omega \times \bR^3) }. \notag
\end{align}
Recall that $k \le m-4-i$. We will consider the case when $k \ge m - 7$ since the remaining case is easier to handle due to \eqref{eq7.1.3} - \eqref{eq7.1.7}.
Furthermore, splitting the sum into $k_1 \le m- 8$ and $k_1 \ge m - 7$  gives
\begin{align}
    \label{eq7.42}
    &    \cI_{1, k} + \cI_{2, k}   \le  \sum_{k_1+k_2 = k: 1 \le k_1 \le m-   8  } \|\partial_t^{k_1}  h\|^2_{ L_{\infty} (\Sigma^T) } 
    \sum_{ s \in \{2, r_i\} } \| \partial_t^{k_2} [f, \nabla_p f]\|^2_{   L_{\infty} ((0, T)) L_{s, \theta/2^{k +2i-1 }  } (\Omega \times \bR^3) }\\
    & +  1_{   m - 7  \le k \le  m-4-i  } \sum_{ k_1+k_2 = k: k_1 \ge m - 7   } \sum_{ s \in \{2, r_i\} } \|(\partial_t^{k_1} 
 h) p_0^{-2} \|^2_{ L_{\infty} ((0, T))  L_s (\Omega \times \bR^3) } \notag \\
 & \times \| \partial_t^{k_2} [f, \nabla_p f] \|^2_{  L_{\infty, 2+ \theta/2^{k  +2 i-1 }} (\Sigma^T) }, \notag
    \end{align}
    where $h = [\bE_g, \bB_g, a_g, C_g, \sigma_{g^{+}  + g^{-} }, \nabla_p \sigma_{g^{+}  + g^{-} }].$
Due to \eqref{eq7.1.3}  - \eqref{eq7.1.5} and \eqref{eq7.1.7} in Lemma \ref{lemma 7.1},
\begin{align}
    \label{eq7.45}
   1_{1 \le k_1 \le m - 8  } \|\partial_t^{k_1}   h\|^2_{ L_{\infty} (\Sigma^T)} \le N \varepsilon_0, 
   \end{align}
   and by  \eqref{eq7.1.8} in the same lemma, we get
 \begin{align*}
  1_{ k_1 \le m-4-i } \sum_{ s \in \{2, r_i\} } \|(\partial_t^{k_1}  h) p_0^{-2}\|^2_{ L_{\infty} ((0, T)) L_s (\Omega \times \bR^3) } \le N \varepsilon_0.
\end{align*}

Next, since $k_1 \ge 1$, we have $k_2 \le k - 1$, so that the second factor in the first term on the r.h.s. of \eqref{eq7.42}  is bounded by
$$
     \sum_{ s \in \{2, r_i\} } \| \partial_t^{k_2} [f, \nabla_p f]\|^2_{ L_{\infty} ((0, T)) L_{s, \theta/2^{k_2  +2i  }} (\Omega \times \bR^3) },
$$
as desired.
Furthermore, if $k_1 \ge m-7$, one has $k_2 \le 2$ (recall that $k \le m-5$), and hence, $k_2+9 \le 11 < m-6 \le k+2i-1$, which gives
\begin{align}
    \label{eq7.77}
    2+ \theta/2^{k+2i-1} \le \theta/2^{k_2 + 9}
\end{align}
for large $\theta$.
  Then,   for sufficiently large $\theta$, the second factor in the second term on the r.h.s of \eqref{eq7.42} is bounded by 
$$
     \|\partial_t^{k_2} [f, \nabla_p f]\|^2_{ L_{\infty, \theta/2^{k_2 + 9  }} (\Sigma^T) }.
$$
Thus, the inequality in \eqref{eq7.40} is true.

\textit{Estimate of $\cI_{3, k}$.} We will show that
\begin{align}
    \label{eq7.41}
 &  \cI_{3, k} \le N \varepsilon_0 \sum_{s \in \{2, r_i\} } \sum_{ k_2 = 0 }^{k-1} \|\partial_t^{k_2} D^2_p f\|^2_{ L_{\infty} ((0, T)) L_{s, \theta/2^{k_2 + 2 i }} (\Omega \times \bR^3) }  \\
 & + N \varepsilon_0  1_{k \ge m- 7   }  \sum_{ s \in \{2, r_4\} } \sum_{ k_2 = 0 }^{  2  } \| \partial_t^{k_2} D^2_p f\|^2_{ L_{\infty} ((0, T)) L_{ r_4, \theta/2^{k_2 + 8   }} (\Omega \times \bR^3) }.  \notag
\end{align}
Inspecting the  proof of \eqref{eq7.40} and using \eqref{eq7.45}, we conclude
\begin{align}
    \label{eq7.93}
    \sum_{s \in \{2, r_i\} } \sum_{k_1+k_2 = k, 1 \le  k_1 \le m - 8  }  \|\eta_{k_1, k_2}^3\|^2_{ L_{\infty} ((0, T)) L_{s, \theta/2^{k_2 +2i-1  }} (\Omega \times \bR^3) } \le \text{r.h.s. of \eqref{eq7.41} }.
\end{align}
Hence, we may assume that $k_1 \ge m - 7$, $k_2 \le 2$.

We denote
$$
  p_0^{ \theta/2^{k+2i-1} } = p_0^{-2} p_0^{ 2 + \theta/2^{k+2i-1}}
    =:  w_1 (p) w_2 (p).
$$
By using  H\"older's inequality in the $x, p$ variables with the exponents $r_{i+1}/r_i$ and $\eta_i/r_i$, where $\eta_i: = (r_{i}^{-1} - r_{i+1}^{-1})^{-1}$, we get
\begin{align}
\label{eq14.B.1.20}
 &    \cI_{3, k}  = \sup_{\tau \le T} \bigg(\int_{\Omega \times \bR^3} |\partial_t^{k_1}  \sigma_{g^{+}  + g^{-} }|^{ r_{i} } (\tau, x, p) \, |D^2_p \partial_t^{k_2} f (\tau, x, p)|^{r_i} \,  p_0^{ r_i (\theta/2^{k+2i-1}) } \, dx dp \bigg)^{2/r_i} \\
 & \le 
 \sup_{\tau \le T} \cI_{3, 1, k} (\tau) \, \cI_{3, 2, k} (\tau), \notag  \\
&
    \cI_{3, 1, k} (\tau) = \bigg(\int_{\Omega \times \bR^3} |\partial_t^{k_1}  \sigma_{g^{+}  + g^{-} } (\tau, x, p)|^{r_{i+1}}  w_1^{r_{i+1}  } (p)\, dx dp\bigg)^{2/r_{i+1}}, \notag \\
&
    \cI_{3, 2, k} (\tau) =  \bigg(\int_{ \Omega \times \bR^3 } |D^2_p \partial_t^{k_2} f  (\tau, x, p)|^{\eta_i} \,  w_2^{ \eta_i } (p)  \, dx dp\bigg)^{2/\eta_i}. \notag
\end{align}

We estimate $\cI_{3, 1, k}$ first.
Recalling the definition of $r_{i+1}$ in \eqref{eq5.10.10} and using the embedding result in   \eqref{eq2.5.1} in  Corollary \ref{corollary 2.5}) with $r_i$ in place of $r$, and invoking the definition of $\cH_g (T)$ in \eqref{eq5.10.1},   we find
\begin{align*}
    &     \|\partial_t^{k_1} g (\tau, \cdot)\|^2_{  L_{ r_{i+1} } (\Omega \times \bR^3)}  \\
    &\le N \|\partial_t^{k_1} g (\tau, \cdot)\|^2_{  S_{r_i, \theta/2^{k_1+2i} } (\Omega \times \bR^3)}
     \le N \cH_g (T) \le N \varepsilon_0.
\end{align*}
Furthermore, differentiating the identity  \eqref{eq6.0} and using the pointwise bound \eqref{eqB.2.1} in Lemma \ref{lemma B.2}, and the fact that $k_1 \le m-4-i$, we find
\begin{align}
    \label{eq14.B.1.11}
       & \cI_{3, 1, k} (\tau) = \| p_0^{-2} \partial_t^{k_1}  \sigma_{g^{+}  + g^{-} } (\tau, \cdot) \|^2_{  L_{ r_{i+1} } (\Omega \times \bR^3) }  \\
        &
    \le N   \|\partial_t^{k_1} g (\tau, \cdot)\|^2_{  L_{ r_{i+1} } (\Omega \times \bR^3) } \le N \varepsilon_0. \notag
\end{align}
We move to  $\cI_{3, 2, k}$. We first  note that since $k_2 \le 2$, \eqref{eq7.77} is valid, and hence, we may replace $w_2 (p)$ with $p_0^{   \theta/2^{k_2+8} }$.
Furthermore, by the definition of $r_1, \ldots r_4$ in \eqref{eq5.10.10} and the fact that  $\Delta r < \frac{1}{42}$, we have
$$
    \frac{1}{\eta_i} = \frac{1}{r_i}  - \frac{1}{r_{i+1}} = \frac{1}{6} -  \Delta r \ge \frac{1}{6}  -\frac{1}{42} \ge \frac{1}{7}.
$$
Hence, by interpolating between $L_2$ and $L_{r_4}$ ($r_4 > 14$), we obtain
\begin{align}
    \label{eq7.94}
     &  \cI_{3, 2, k} (\tau)
    \le N  \sum_{ s \in \{2, r_4\} }\|\partial_t^{k_2} f (\tau, \cdot)\|^2_{     S_{s, \theta/2^{k_2+8} }  (\Omega \times \bR^3) }.
\end{align}
Combining \eqref{eq7.93} - \eqref{eq7.94}, we conclude that \eqref{eq7.41} holds.

\textit{Estimate of $\cI_{5, k}$.}
In the case when $i = 1$ and $r_i = 2$, we keep $\cI_{5, k}$ as is.
In the remaining case $i > 1$, we first note that by  Sobolev embedding and the fact that 
$$
  1 - \frac{6}{r_{i-1}} >  - \frac{6}{r_{i}}, 
$$
which follows from \eqref{eq5.10.10}, 
 we have
\begin{align*}
     \|\partial_t^k [\bE_f, \bB_f] (\tau, \cdot)\|_{  L_{ r_{i} } (\Omega) }   \lesssim_{\Omega}    \|\partial_t^k [\bE_f, \bB_f] (\tau, \cdot)\|_{ W^1_{ r_{i-1} } (\Omega) }.
\end{align*}
Hence, by \eqref{eq7.92} with $i$ replaced with $i-1$, we obtain
\begin{align}
    \label{eq7.95}
    & \cI_{4, k} =  \|\partial_t^k [\bE_f, \bB_f] \|^2_{ L_{\infty} ((0, T))  L_{ r_{i} } (\Omega) } \\
    & \le N \|\partial_t^{k} f\|^2_{ L_{\infty} ((0, T)) L_{ r_{i-1} } (\Omega \times \bR^3) } + N \sum_{l=k}^{k+1} \|\partial_t^l [\bE_f, \bB_f]\|^2_{ L_{\infty} ((0, T)) L_{ r_{i-1} } (\Omega) }.  \notag 
\end{align}

\textit{Estimate of $\cI_{6, k}$.}
By \eqref{eqB.1.1} in Lemma \ref{lemma A.1} and interpolation and H\"older's inequalities,  for any $\varepsilon_1 \in (0, 1)$ and sufficiently large $\theta$, we have
\begin{align}
    \label{eq7.97}
       & \cI_{6, k} = \sum_{s \in \{2, r_i\} }  \| K (\partial_t^k f)\|^2_{  L_{\infty} ((0, T)) L_{s, \theta/2^{k + 2 i-1} } (\Omega \times \bR^3) } \\
       & \le N   \sum_{s \in \{2, r_i\} }  \|\partial_t^k f\|^2_{  L_{\infty} ((0, T))  L_s (\Omega) W^1_s (\bR^3)    }  \notag \\
       & \le   \varepsilon_1 \|D^2_p \partial_t^k f\|^2_{ L_{\infty} ((0, T))  L_{r_i, \theta/2^{ k+ 2i}  } (\Omega \times \bR^3) } 
       + N \varepsilon_1^{-1} \|\partial_t^k f\|^2_{ L_{\infty} ((0, T))  L_{r_i, \theta/2^{k + 2 (i- 1)}  } (\Omega \times \bR^3) }.\notag
\end{align}

Finally, gathering all the estimates \eqref{eq7.37.1} - \eqref{eq7.40},  \eqref{eq7.41},  \eqref{eq7.95}-\eqref{eq7.97} gives
\begin{align}
    \label{eq7.43}
     & \sum_{s \in \{2, r_i\} } \|\text{r.h.s. of \eqref{eq7.48}} \|^2_{ L_{s, \theta/2^{k  +2i-1 } } (\Sigma^T) }  \\
     & \le     
     N \varepsilon_1^{-1}  \sum_{l=k}^{k+1}  \|\partial_t^{  l   } f\|^2_{  L_{\infty} ((0, T))  L_{    r_i, \theta/2^{ 
      l + 2 (i-1)  }  }  (\Omega \times \bR^3)  } 
    + \varepsilon_1 \|\partial_t^k D^2_p f\|^2_{  L_{\infty} ((0, T))  L_{r_i, \theta/2^{ k+ 2i}  } (\Omega \times \bR^3) }   \notag \\
  &   + \|\partial_t^{k+1} f\|^2_{  L_{\infty} ((0, T))  L_{    2, \theta/2^{  k + 2 i - 1}  }  (\Omega \times \bR^3)  }   + N 1_{i=1}  \|\partial_t^{k} \bE_f\|^2_{ L_{\infty} ((0, T)) L_{ 2 } (\Omega \times \bR^3) } \notag  \\
  &    + N 1_{i> 1}  \bigg(\|\partial_t^{k} f\|^2_{ L_{\infty} ((0, T)) L_{ r_{i-1} } (\Omega \times \bR^3) } + \sum_{l=k}^{k+1} \|\partial_t^l [\bE_f, \bB_f]\|^2_{ L_{\infty} ((0, T)) L_{ r_{i-1} } (\Omega) } \bigg)  \notag\\
    & + 1_{k > 0}  N  \varepsilon_0 \sum_{s \in \{2, r_i\} } \sum_{ k_2 = 0 }^{k-1} \|\partial_t^{k_2} f\|^2_{ L_{\infty} ((0, T)) S_{s, \theta/2^{k_2 + 2 i   }} (\Sigma^T) }  \notag  \\
    & + N  \varepsilon_0 1_{k \ge m- 7  } \sum_{ k_2 = 0 }^{  2 } \bigg(\| \partial_t^{k_2} [f, \nabla_p f]\|^2_{ L_{\infty, \theta/2^{k_2 + 9  }} (\Sigma^T) }
    + \| \partial_t^{k_2} f \|^2_{  L_{\infty} ((0, T))  S_{ r_4, \theta/2^{k_2 + 8  }} (\Omega \times \bR^3) } \bigg).  \notag 
\end{align}
We note that the first term involving weighted $L_{\infty}^{t, x, p}$ norm on the r.h.s. is bounded by 
$$
       \varepsilon_0   1_{k \ge m- 7  } \sum_{ k_2 = 0 }^{  2 } \| \partial_t^{k_2} [f, \nabla_p f]\|^2_{ L_{\infty, \theta/2^{k_2 + 9  }} (\Sigma^T) }  \le   N \varepsilon_0 \sup_{\tau \le T}  \cI_f (\tau)
$$ 
due to the first inequality in \eqref{eq7.1.3}.

 Combining  \eqref{eq7.92}, \eqref{eq7.37}, and \eqref{eq7.43},  and summing up over $k \le m-4-i$, we get
\begin{align*}
   &  \sum_{k=0}^{m-4-i}  \| \partial_t^k f\|^2_{ L_{\infty} ((0, T)) S_{r_i, \theta/2^{k +2i  }} (\Omega \times \bR^3)   }  +  1_{ i < 4}  \sum_{k=0}^{m-4-i}  \|  \partial_t^k f\|^2_{ L_{\infty} ((0, T)) L_{r_{i+1}, \theta/2^{k + 2 i}} (\Omega \times \bR^3)   }  \\
  &  + 1_{i=1}  \sum_{k=0}^{m-i} \|\partial_t^k [\bE_f, \bB_f] \|^2_{ L_{\infty} ((0, T))  W^1_{ r_{i} } (\Omega) }        \\
   & + 1_{i > 1} \sum_{k=0}^{m-4-i}   \|\partial_t^k [\bE_f, \bB_f] \|^2_{ L_{\infty} ((0, T))  W^1_{ r_{i} } (\Omega) }   
   + 1_{i > 1}  \sum_{k=0}^{m-4-i} \|\partial_t^k [\bE_f, \bB_f] \|^2_{ L_{\infty} ((0, T))  L_{ r_{i+1} } (\Omega) } \\
   &  \le    \varepsilon_1   \sum_{k=0}^{m-4-i}  \| \partial_t^k f\|^2_{  L_{\infty} ((0, T))  S_{r_i, \theta/2^{ k+ 2i}  } (\Omega \times \bR^3)   }     
   + N \varepsilon_0   \sup_{\tau \le T} \cI_f (\tau)  \notag  \\
& +   N  \varepsilon_1^{-1}  \sum_{k=  0  }^{m-3-i}  \| \partial_t^{k} f\|^2_{ L_{\infty} ((0, T))   L_{r_i, \theta/2^{k + 2 (i - 1)} }  (\Omega \times \bR^3) }  \notag \\
   &   + N 1_{i > 1} \sum_{k=0}^{m-4-i}  \|\partial_t^k f\|^2_{ L_{\infty} ((0, T)) L_{r_{i-1}}  (\Omega \times \bR^3)   }   + N 1_{i > 1} \sum_{k=0}^{m-3-i} \|\partial_t^k [\bE_f, \bB_f]\|^2_{ L_{\infty} ((0, T)) L_{r_{i-1}  }  (\Omega) }    \\
  & + N  \varepsilon_1^{-1} \sum_{k=0}^{m- 3   -i } \|\partial_t^k f\|^2_{  L_{\infty} ((0, T)) L_{2, \theta/2^k} (\Omega \times \bR^3) }.
  \end{align*} 
  We point out that 
  \begin{itemize}
  \item  by choosing  $\varepsilon_1$ sufficiently small, we may absorb the first term on the r.h.s.  into  the l.h.s.,
  \item the fourth term on the r.h.s. is bounded by the third   one due to H\"older's inequality provided that $\theta$ is sufficiently large, 
  
  \item  if we replace $i$ with $i-1$ in the second term on the l.h.s., we obtain the  third one  on the r.h.s.,
  \item  if we replace $i$ with $i-1$ in the fifth term on the l.h.s., the resulting term dominates the fifth term on the r.h.s..
  \end{itemize}

Then, by using induction on $i$,  we obtain   the desired estimate \eqref{eq7.96}. 

We note that the assumption \eqref{eq7.91} is actually not necessary, as one can use an induction argument by ascending from $k = 0$ to $k  = m-4-i$ and using the bounds \eqref{eq7.37} and \eqref{eq7.43}.
At each step of the induction argument, one needs to use the existence and uniqueness results 
\begin{itemize}
    \item[a)] for  finite energy  and strong solutions to the steady KFP equation \eqref{eq27.1.2} (see Propositions \ref{proposition 14.C.6} and \ref{proposition 27.4}), 
    \item[b)] for strong solutions to  Maxwell's equation (see Chapter $(vii)$ in \cite{DL_76}).
\end{itemize}
\end{proof}


\subsection{Proof of the bound    (\ref{eq5.2.4.1}).} 
\label{section 7.1}
In this proof, $N = N (r_1, \ldots, r_4, \alpha, \Omega, \theta, m)$.
We will do some formal calculations below 
assuming that Proposition \ref{proposition 7.2} is valid.
We first prove the energy-dissipation estimate. Combining this inequality with  $L_{\infty}^t S_p$ and $L_{\infty}^t W^1_2 (\Omega)$ bounds  in \eqref{eq7.96}, we are able to close the estimate of $y_f (T)$.

\textbf{Step 1: Energy-dissipation bound.} 
 Here, we estimate the total energy and dissipation, that is, $\cE_f (\tau) + \int_0^{\tau} \cD \, dt$ (see \eqref{eq5.10.15} and \eqref{eq5.10.16}).
First, applying the standard energy identity for the weak solution to Maxwell's equations differentiated $k$ times with respect to $t$  and using the Cauchy-Schwarz inequality, we have
\begin{align}
    \label{eq7.80}
  &  \frac{1}{2} \|\partial_t^k  [\bE_f, \bB_f]\|^2_{ L_{\infty} ((0, T)) L_2 (\Omega) } \\
  & \le \frac{1}{2} \| \bE_{0, k}, \bB_{0, k}\|^2_{ L_2 (\Omega) } + N \|\partial_t^k f\|^2_{L_2 (\Sigma^T) }
   +  \|\partial_t^k \bE_f\|^2_{  L_2 ((0, T) \times \Omega)  }. \notag
\end{align}

Next, for the sake of convenience, we introduce 
\begin{align}
        \label{eq7.98}
    \theta_k 
    = \begin{cases}
    &\theta/2^k, \, \,  k = 0, 1, \ldots, m-4, \\
    & 0, \, \,  k = m-3, \ldots, m.
    \end{cases}
\end{align}

Differentiating the linear Landau equation \eqref{eq5.2} formally $k$ times  in the  $t$ variable  and  using a variant of the energy identity \eqref{eq27.4.12}, we get for each $\tau > 0$,   
\begin{align}
    \label{eq7.11}
& \frac 1 2 \big(\|\partial_t^k f (\tau, \cdot)\|^2_{ L_{2,  \theta_k } (\Omega \times \bR^3) } - \| f_{0, k}\|^2_{ L_{2, \theta_k } (\Omega \times \bR^3) }\big)\\
& 
+ \underbrace{ \int_{ \Sigma^{\tau} }  \langle L \partial_t^k f, \partial_t^k f\rangle  \, p_0^{2  \theta_k} dz }_{ = I_1 }\notag \\
& -  \bm{\xi}_1 \underbrace{  \int_{ \Sigma^{\tau} } (v (p) \cdot \partial_t^k \bE_f)    (\partial_t^k f)  \,   \sqrt{J}    p_0^{2 \theta_k } dz}_{ =  I_2 } \notag \\
& = \underbrace{  \int_{ \Sigma^{\tau} } \langle(\partial_t^k \Gamma (f, g)),    (\partial_t^k f) \rangle  \,  p_0^{2 \theta_k } dz \notag}_{ I_4  }\\
& + \frac{ \bm{\xi} }{2} \underbrace{   \sum_{k_1+k_2=k} \binom{k}{k_1} \int_{ \Sigma^{\tau} } v (p)  \cdot (\partial_t^{k_1} \bE_g) (\partial_t^{k_2} f) (\partial_t^k f)  p_0^{2 \theta_k }\, dz }_{ I_5 }\notag \\
&- \bm{\xi} \underbrace{   \sum_{k_1+k_2=k} \binom{k}{k_1} \int_{ \Sigma^{\tau} }  \big(\partial_t^{k_1} \bE_g + v (p) \times  (\partial_t^{k_1}  \bB_g)\big) \cdot (\nabla_p \partial_t^{k_2}  f) (\partial_t^k f) \,  p_0^{2 \theta_k} dz }_{ I_6  },\notag
\end{align}
where $f_{0, k}$ is defined in \eqref{eq3.3.2}.

\textit{Estimate of `quadratic terms.'}
Using the fact that $L = -A-K$ (see \eqref{eq5}) and combining the coercivity estimate of $A$ in \eqref{eq2.4.1} in Lemma \ref{lemma 2.4} with the estimate of $K$ in \eqref{eq2.4.2}, we have
\begin{equation}
            \label{eq7.12}
    I_1 \ge \kappa \| \nabla_p \partial_t^k f\|^2_{ L_{2, \theta_k} (\Sigma^{\tau}) }  - N_1 (\theta, k) \|\partial_t^k f\|^2_{ L_{2} (\Sigma^{\tau}) }.
\end{equation}

Next, by the Cauchy-Schwarz inequality,  we get 
\begin{align}
                \label{eq7.14}
   I_2 \le   \|\partial_t^k \bE_f  \|^2_{ L_2 ((0, \tau) \times \Omega) } +  N \| \partial_t^k f\|^2_{L_{2} (\Sigma^{\tau})}.
\end{align}

\textit{Estimate of `cubic terms.'} 
To estimate $I_4 - I_6$, we need to prove the following \textbf{Claim}: for  any nonnegative integers $k_1, k_2$ such that $k_1+k_2 = k$, one has
\begin{align}
    \label{eq7.15}
  &  (i) \,   \bigg|\int_{ \Sigma^{T} } (\Gamma (\partial_t^{k_2} f, \partial_t^{k_1} g)) \cdot (\partial_t^k f) p_0^{2 \theta_k}    \, dz\bigg|   \le N \sqrt{\varepsilon_0} y_f (T), \\
           \label{eq7.16}
 &  (ii) \,  \bigg|\int_{ \Sigma^{T} } (|\partial_t^{k_1} f| + |\nabla_p \partial_t^{k_1} f|) |\partial_t^{k_2} [\bE_g, \bB_g]|  |\partial_t^k f|\,  p_0^{2 \theta_k} dz \bigg| \\
 & \le N \sqrt{\varepsilon_0} \, y_f (T). \notag
\end{align}
Then, applying $(i) - (ii)$, we get
\begin{align}
    \label{eq7.17}
    I_4 + I_5 + I_6 \le N \sqrt{\varepsilon_0} y_f (T).
\end{align}

First, we consider the case when $k \le m-4$, so that $\theta_k = \theta/2^k$. 
We start with $(i)$.
It suffices to consider the case  when
$m-7  \le k \le m -4$ 
since the remaining case is simpler thanks to \eqref{eq7.1.3} - \eqref{eq7.1.7}.
 By the estimate \eqref{eqG.4.1} in Lemma \ref{lemma G.4} with $r = 2$ in  and  $L_{\infty}^{t, x} - L_2^{t, x} - L_2^{t, x}$ H\"older's inequality,  the integral on the left-hand side of \eqref{eq7.15} is bounded by 
\begin{align*}
  &  N (\theta)  \mathcal{J}_{ k }^1   \big(1_{k_1 \le m - 8, k_2 \ge 1 } \mathcal{J}_{ k_2 }^1 \mathcal{J}_{ k_1 }^2  
 + 1_{ m- 7 \,   \le k_1,  k_2 \le 3   } \mathcal{J}_{k_2}^3 \mathcal{J}_{k_1}^4\big),  \\
  & \mathcal{J}_{l}^1 = \|\partial_t^{l} f\|_{ L_2 ((0, \tau) \times \Omega) W^1_{2, \theta/2^k} (\bR^3) },  \\
  &  \mathcal{J}_{l}^2 = \|\partial_t^{l} g\|_{  L_{\infty} ((0, \tau) \times \Omega) W^1_{2} (\bR^3)},   \\
  & \mathcal{J}_{l}^3 =  \| \partial_t^{l} f\|_{  L_{\infty} ((0, \tau) \times \Omega) W^1_{2, \theta/2^k} (\bR^3)  }, \\
  &  \mathcal{J}_{l}^4 =  \|\partial_t^{l} g\|_{ L_{2} ((0, \tau) \times \Omega) W^1_{2}  (\bR^3) }. 
 \end{align*}
 By the definition of  $\cD_f$ in \eqref{eq5.10.16},  for $l \le m-4$,
 $$
    \mathcal{J}^1_{l} \le  \sqrt{\cD_f (T)} ,
$$ 
 and  similarly, by the smallness assumption on $y_g (T)$ (see \eqref{eq5.2.4.2}), we get for $l \le m$
\begin{align*}
 \mathcal{J}_{l}^4  \le    \sqrt{\cD_{ g} (T)}   \le \sqrt{\varepsilon_0}.
\end{align*}
 Next,   due to the bound \eqref{eq7.1.3} in   Lemma \ref{lemma 7.1}, we have
 \begin{align*}
     &   1_{ k_1 \le m - 8   } \mathcal{J}_{k_1}^2  \le  N \sqrt{\varepsilon_0}. 
\end{align*}
Furthermore, observe that for $k_2 \le 3$ and $k \ge m-7$, one has  $k_2 + 9 \le 12  < m-7   \le k$ (recall that $m \ge  20$), so that 
$1_{k_2 \le  3, k \ge m -  7 \,  } \,  \theta/2^k  <   \theta/2^{k_2 + 9  }.$
By  this and   the the first inequality in \eqref{eq7.1.3}, we conclude  
\begin{align*}
    1_{k_2 \le  3, k \ge m-7  } \mathcal{J}_{k_2}^3 \le   \| \partial_t^{k_2} f\|_{ L_{\infty} ((0, T) \times \Omega) W^{1}_{2, \theta/2^{  k_2 + 9   } } (\bR^3) } \le  N \sqrt{\cH_f (T)}.
  \end{align*}
Combining the above estimates,  we obtain \eqref{eq7.15}.

The assertion $(ii)$ is proved in a similar way.
We note that in the case when $k \ge m-3$, we have $\theta_k = 0$, and the same argument gives the desired bounds \eqref{eq7.15}-\eqref{eq7.16}.

Finally, gathering the estimates  \eqref{eq7.11} - \eqref{eq7.14} and \eqref{eq7.17} and summing up over $k$, 
and invoking definitions of $\cE_f$ and $\cD_f$ in \eqref{eq5.10.15} and \eqref{eq5.10.16}, respectively, we obtain
\begin{align}
 \label{eq7.22}   
 & \sup_{\tau \le T} \cE_f (\tau) + \int_0^{T} \cD_f (\tau) \, d\tau    \\
 &
  = \sum_{k=0}^m \bigg(\|\partial_t^k f\|^2_{ L_{\infty} ((0, T)) L_{2 } (\Omega \times \bR^3) }  +  \|\partial_t^k f\|^2_{ L_2 ((0, T) \times \Omega) W^1_2 (\bR^3)}  + \|\partial_t^k [\bE_f, \bB_f]\|^2_{ L_{\infty} ((0, T)) L_2 (\Omega)   }\bigg) \notag \\
&   +  \sum_{k=0}^{m-4}  \bigg(\|\partial_t^k f\|^2_{ L_{\infty} ((0, T)) L_{2, \theta/2^k} (\Omega \times \bR^3)  }  
  +  \|\partial_t^k f\|^2_{ L_2 ((0, T) \times \Omega) W^1_{2, \theta/2^k}  (\bR^3)  }\bigg)  \notag \\
 &   \le  N  \sum_{k=0}^m \bigg(\|f_{0, k}\|^2_{ L_{2} (\Omega \times \bR^3) }  
 +    \|[\bE_{0, k}, \bB_{0, k}]\|^2_{ L_{2} (\Omega) } +  \|\partial_t^k f\|^2_{ L_2 (\Sigma^T) } 
 +  \|\partial_t^k \bE_f\|^2_{ L_2 ((0, T) \times \Omega) }\bigg)\notag  \\
 & +   N \sum_{k=0}^{m-4}    \|f_{0, k}\|^2_{ L_{2, \theta/2^k} (\Omega \times \bR^3) } 
    + N \sqrt{\varepsilon_0} \, y_f (T).\notag
\end{align}

\textbf{Step  2: closing the estimate of $y_f$.}
 Finally, combining  \eqref{eq7.22} with \eqref{eq7.96}, we obtain
 \begin{align}
 \label{eq7.60}
& y_f (T) \le N \sqrt{\varepsilon_0} y_f (T)  +  N \sum_{k=0}^m \bigg(\|f_{0, k}\|^2_{ L_{2 } (\Omega \times \bR^3) }  +   \|[\bE_{0, k},  \bB_{0, k}]\|^2_{ L_2 (\Omega) }\bigg) \\
&   + N \sum_{k=0}^{m-4} \|f_{0, k}\|^2_{  L_{2, \theta/2^k} (\Omega \times \bR^3) } 
+ N \sum_{k=0}^m \bigg(\|\partial_t^k f\|^2_{ L_2 (\Sigma^T) }  
 + \|\partial_t^k \bE_f\|^2_{ L_2 ((0, T) \times \Omega) } \bigg)\notag \\
  & 
    \le   N (\sqrt{\varepsilon_0} + T) y_f (T)  + N \varepsilon_0/M. \notag 
 \end{align}
By  choosing $\varepsilon_0 < (4N)^{-2}$, $T < (4N)^{-1}$, $M > 4N$, and  using the smallness assumption   on $[f_{0, k}, \bE_{0, k}, \bB_{0, k}]$ in \eqref{eq5.1.1}, we obtain the desired estimate $y_f (T) < \varepsilon_0$.

\section{Proof of Theorem \ref{theorem 5.1}}
                    \label{section 8}
We first state an auxiliary result that is useful for establishing both the existence and the uniqueness of the solution to the RVML system. 
See also the proof of Lemma 8.2 in \cite{KGH_20}.

\begin{lemma}
            \label{lemma 8.1}
Invoke the assumptions of Theorem \ref{theorem 5.1}  and let  $\varepsilon_0$, $\theta$,  $M$,  and $T$ be the constants introduced in the statements of that theorem (see  \eqref{5.1.0}). Furthermore, 
let $[g^{(j)}, \bE_{  g^{(j)} }, \bB_{  g^{(j)} }], j = 1, 2,$ be functions satisfying the  \eqref{eq5.2.12.1} - \eqref{eq5.2.4.2} and 
     $[f^{(j)}, \bE_{ f^{(j)} }, \bB_{ f^{(j)} }], j = 1, 2,$  be two strong solutions to the linear RVML system \eqref{eq5.2} - \eqref{eq5.5} with $[g, \bE_g, \bB_g]$ replaced with $[g^{(j)}, \bE_{ g^{(j)} }, \bB_{ g^{(j)} }], j = 1, 2,$ such that $[f^{(j)}, \bE_{ f^{(j)} }, \bB_{ f^{(j)} }], j = 1, 2,$ satisfy the conditions analogous to $(i) - (iv)$ in Theorem \ref{theorem 5.1}.  
We also denote $f^{1, 2} = f^{(1)} - f^{(2)}$, $\bE_{f}^{ 1, 2}= \bE_{f^{(1)}} - \bE_{f^{(2)}}$ and define $\bB_f^{1, 2}$,  $g^{1, 2}, \bE_{g}^{ 1, 2}, \bB_{g}^{1, 2}$ in the same way.
Then, we have 

\begin{align}     
            \label{eq8.1.0}    
   & \|\partial_t^k  f^{1, 2} \|^2_{ L_{\infty} ((0, T)) L_2 (\Omega \times \bR^3) } 
   + \|\partial_t^k f^{1, 2}\|^2_{ L_{   2    } ((0, T) \times \Omega) W^1_2 (\bR^3) }\\
   &+ \|\partial_t^k [\bE_{f }^{1, 2}, \bB_{f }^{1, 2}]\|^2_{   L_{\infty} ((0, T))  L_2 (\Omega)  } 
   \le \frac{1}{2} \big(\|\partial_t^k  g^{1, 2} \|^2_{ L_{\infty} ((0, T)) L_2 (\Omega \times \bR^3) } \notag \\
   &+     \| \partial_t^k g^{1, 2}\|^2_{   L_2 ((0, T) \times \Omega) W^1_2 (\bR^3)   } 
   + \|\partial_t^k [\bE_{g }^{1, 2}, \bB_{g }^{1, 2}]\|^2_{  
     L_{\infty} ((0, T))  L_2 (\Omega)   }\big). \notag
  \end{align}
\end{lemma}

\begin{proof}[Proof of Lemma \ref{lemma 8.1} ]
    We inspect the argument we used to establish the energy-dissipation bound \eqref{eq7.22}. In particular, we write down the equation satisfied by $\partial_t^k f^{1, 2}$ and use a variant of the energy identity in \eqref{eq27.4.12}.
    The `quadratic' terms in the energy identity are estimated in the same way as in \eqref{eq7.12}-\eqref{eq7.14}.
    On the other hand, we need to slightly modify the estimates of the `cubic' terms. For the sake of clarity, we focus on the integral 
    \begin{align*}
    & I_{4} = \int_{ \Sigma^{\tau} } \langle\partial_t^k \big(\Gamma (f^{(1)}, g^{(1)}) - \Gamma (f^{(2)}, g^{(2)})\big),    \partial_t^k f^{1, 2} \rangle \,  dz   \\
    & = \underbrace{\int_{ \Sigma^{\tau} } \langle\partial_t^k \big(\Gamma (f^{1, 2}, g^{(1)})\big),    \partial_t^k f^{1, 2} \rangle}_{ = I_{4, 1} } \,  dz 
    + \underbrace{\int_{ \Sigma^{\tau} } \langle\partial_t^k \big(\Gamma (f^{(2)}, g^{1, 2})\big),    \partial_t^k f^{1, 2} \rangle \,  dz}_{ = I_{4, 2} }.
    \end{align*}
Inspecting the proof of \eqref{eq7.15}, using the bounds  $y_{g^{(1)}} < \varepsilon_0$, $y_{f^{(2)}} < \varepsilon_0$ combined with Lemma \ref{lemma 7.1}, and employing the Cauchy-Schwarz inequality, we conclude
\begin{align*}
   & I_{4, 1} \le N \sqrt{ \varepsilon_0} \sum_{k=0}^m \bigg( \|\partial_t^k f^{1, 2}\|^2_{ L_{\infty} ((0, T)) L_{2 } (\Omega \times \bR^3) }  +  \|\partial_t^k f^{1, 2}\|^2_{ L_2 ((0, T) \times \Omega) W^1_2 (\bR^3)}\bigg), \\
   & I_{4, 2} \le N  \varepsilon_0 \sum_{k=0}^m  \bigg( \|\partial_t^k g^{1, 2}\|^2_{ L_{\infty} ((0, T)) L_{2 } (\Omega \times \bR^3) }  +  \|\partial_t^k g^{1, 2}\|^2_{ L_2 ((0, T) \times \Omega) W^1_2 (\bR^3)}\bigg)  \\
   &
   + \frac{1}{2}  \|\partial_t^k f^{1, 2}\|^2_{ L_2 ((0, T) \times \Omega) W^1_2 (\bR^3)},
\end{align*}
where $N  = N (r_1, \ldots, r_4, \theta, \Omega, \alpha, m)$. For the closure, one needs to take $\varepsilon_0$ and $T$ sufficiently small. 
\end{proof}

\begin{proof}[Proof of Theorem \ref{theorem 5.1}] We note that the uniqueness follows directly from the above lemma. 

The existence is proved by passing to the limit in the iteration scheme \eqref{eq5.1}-\eqref{eq5.1.0.1}. Since the argument is standard (see, for example, \cite{KGH_20}), we will not present it here but point out major steps. 
\begin{enumerate}
\item By  Proposition \ref{proposition 7.2}, the sequence $[f^n, \bE^n, \bB^n], n \ge 1,$ is well-defined, and 
by \eqref{eq5.2.4.1} in  Proposition \ref{proposition 5.2}, one has $y_{f_n} (T) < \varepsilon_{ 0 }$ for each $n$. By \eqref{eq8.1.0} in Lemma \ref{lemma 8.1}, the sequence $\partial_t^k f^n, k \le m,$ is Cauchy in 
$L_{\infty}^t L_2^{x, p} \cap  L_2^{t, x} W^1_2 (\bR^3)$, and $[\bE^n, \bB^n], n \ge 1,$ is a Cauchy sequence in the space $L_{\infty}^t L_2^x$, and hence, $[f^n, \bE^n, \bB^n]$ converge to some  $[f, \bE, \bB]$. In addition, using \eqref{eq8.1.0} again, we conclude that all the temporal derivatives up to order $m$ also converge.
\item By using Green's identity \eqref{eqC.5.1}, we write down the weak formulation of the system with a test function $\phi$ satisfying \eqref{eq27.2.1} - \eqref{eq27.2.3}.
 Due to the uniform in $n$ estimates in Proposition \ref{proposition 5.2} and the fact that $f_n$ converges to $f$, we may pass to the limit in the weak formulation. In particular, one needs to use Lemma \ref{lemma 8.1} to pass to the limit in the integrals involving the Lorentz and collisional terms.

    \item  Due to the convergence in  $(1)$, $\partial_t^k f  \in  L_{\infty} ((0, T)) L_2 (\Omega \times \bR^3) \cap L_2 ((0, T) \times \Omega) W^1_2 (\bR^3), k \le m,$ is an intermediate finite energy solution (see Definition \ref{definition 27.6}  in the proof of Proposition \ref{proposition 27.4})  to Eq. \eqref{eq5.0} formally  differentiated $k$ times in $t$ with the SRBC and $\partial_t^k f (0, \cdot) \equiv f_{0, k} (\cdot)$.
  We point out that in the proof of the aforementioned proposition, we showed that any intermediate finite energy solution is a finite energy solution in the sense of Definition \ref{definition 27.1}, and hence, $\partial_t^k f \in C ([0, T]) L_2 (\Omega \times \bR^3), k \le m$, as desired.
Furthermore, since $\partial_t^k f \in S_2 (\Sigma^T), k \le m-5$, by Remark \ref{remark 27.8}, we conclude that $\partial_t^k f,  k \le m$
is a strong solution to the $k$ times differentiated Landau equation.
    
    \item By using a limiting argument, we conclude that $\partial_t^k [\bE_f, \bB_f] \in C ([0, T]) L_2 (\Omega) \cap L_{\infty} ((0, T)) W^1_2 (\Omega), k \le m-1,$  is a  strong solution to Maxwell's equations \eqref{eq5.0.1}-\eqref{eq5.0.2} with the perfect conductor BC, initial data $[\bE_{0, k}, \bB_{0, k}]$, whereas $\partial_t^m [\bE_f, \bB_f] \in  L_{\infty} ((0, T)) L_2 (\Omega)$ is a weak solution.  In addition, the identities in \eqref{eq5.0.3} formally differentiated $k$ times in $t$ are valid.  The fact that  $\partial_t^m [\bE_f, \bB_f] \in C ([0, T]) L_2 (\Omega)$   can be proved by a mollification argument as in the proof of Theorem 4.1 in Chapter $(vii)$ in \cite{DL_76}. 
\end{enumerate}
\end{proof}

\textbf{Acknowledgement.} The authors thank the anonymous referees for their interest in the present work and express sincere gratitude to the first referee for identifying several typos.

\appendix

\section{}

\label{Appendix A}

\begin{lemma}
		\label{lemma A.6}
		Let $\Psi: \Omega_{r_0} (x_0) \times \bR^3 \to \mathbb{H}_{-}$ be a local diffeomorphism given by \eqref{eq2.55} -  \eqref{eq2.56}.   
		Then, the following assertions hold.

$(i)$		  For 
$$
C (y) =  \bigg(\frac{\partial x}{\partial y}\bigg)^T \bigg(\frac{\partial x}{\partial y}\bigg),
$$
one has
\begin{equation}
            \label{eqA.6.2}
    C^{i 3} (y) = 0, \, \, i \in \{1, 2\}, \quad \text{if} \, \, y_3 = 0.
\end{equation}

$(ii)$ For any $y \in \psi (\Omega_{r_0}) \cap \{y_3 = 0\}$ and any $w$,
\begin{equation}
			\label{eqA.6.3}
	\bigg|\bigg(\frac{\partial x}{\partial y}\bigg) w\bigg| =  \bigg|\bigg(\frac{\partial x}{\partial y}\bigg) \bm{R} w\bigg|.
\end{equation}

$(iii)$
Let $u$ be a function on $\overline{\Omega_{r_0} (x_0)} \times \bR^3$ satisfying 
\begin{equation}
        \label{eqA.6.1}
	u (x, p) = u (x,  R_x p), (x, p) \in \gamma_{-}
\end{equation}
and denote
\begin{align*}
&	\mathsf{U} (x, p) = \int_{\bR^3} \Phi (P, Q) u (x, q) \, dq, \quad \mathsf{\widehat U} (y, w) = \mathsf{U} (x (y), p (y, w)),\\
&
	\mathfrak{U} (y, w)  = \bigg(\frac{\partial y}{\partial x}\bigg) \mathsf{ \widehat  U} (y, w) \bigg(\frac{\partial y}{\partial x}\bigg)^T.
\end{align*}
Then, one has
\begin{equation}
        \label{eqA.6.6}
	\mathfrak{U}^{i 3} (y, w)  = - \mathfrak{U}^{i 3} (y,  \bm{R} w), \,  i \in  \{1, 2\}, \quad \text{if} \, \, y_3 = 0.
\end{equation}
\end{lemma}

\begin{proof}
$(i)$ We assume that $\rho^{(\varepsilon)}$ is a mollification of $\rho$ with a standard mollifier $\iota$ and let $\rho^{(\varepsilon), j}$ be the mollification of $\rho$ with the mollifier $(-y_j \partial_{y_j} \iota), j = 1, 2$. The assertion follows from the identities
\begin{align}
\label{eqA.6.0}
&
\bigg(\frac{\partial x}{\partial y}\bigg) = 	 \begin{pmatrix}
						 1 - y_3 \rho_{1 1}^{(y_3) }&  -y_3 \rho_{1 2}^{ (y_3) }& - \rho_1^{ (y_3)  }
       + y_3  (\rho_{11}^{ (y_3), 1 } + \rho_{12}^{ (y_3), 2 }) \\
						-y_3 \rho_{1 2}^{(y_3) }  &  1 - y_3 \rho_{2 2}^{ (y_3) } & - \rho_2^{(y_3)} +  y_3  (\rho_{12}^{ (y_3), 1 }  + \rho_{22}^{ (y_3), 2 })   \\
						\rho_1	&		\rho_2		& 1\\
						\end{pmatrix},  \\
&	C (y_1, y_2, 0)  						= \begin{pmatrix}
						 1  &  0&  \rho_1\\
						0  &  1 &  \rho_2\\
						-\rho_1	&		-\rho_2		& 1\\
						\end{pmatrix}
						 \begin{pmatrix}
						 1  &  0& - \rho_1\\
						0  &  1 & - \rho_2\\
						\rho_1	&		\rho_2		& 1\\
						\end{pmatrix} \notag\\
						&
						= \begin{pmatrix}
						 1   + \rho_1^2&  \rho_1 \rho_2 &  0\\
						\rho_1 \rho_2  & 1+ \rho^2_{2} &  0\\
						0	&		0		&1+\rho_1^2+\rho_2^2
						\end{pmatrix},\notag
\end{align}
where $\rho_{i j} = \partial_{y_i y_j} \rho$.

$(ii)$ The desired identity  follows from the equality
$$
	\bigg|\bigg(\frac{\partial x}{\partial y}\bigg) \bm{R} w\bigg|^2 = C^{i j} (y) (\bm{R} w)_i  (\bm{R} w)_j
$$
and \eqref{eqA.6.2}.

$(iii)$ We  denote $q_0 = (1+|q|^2)^{1/2}$, 
\begin{align*}
&	w' = \bigg(\frac{\partial y}{\partial x}\bigg) q, \, \, 	\widehat p_0  = p_0 (x (y),  p (y, w)), \, \, \widehat q_0  = q_0 (x (y),  q (y, w')),\\
&
    \widehat P (y, w) = (\widehat p_0, p (y, w)), \, \, \widehat Q (y, w') = (\widehat q_0, q (y, w)).
\end{align*}

Furthermore, changing variables $q = \big(\frac{\partial x}{\partial y}\big) w'$ gives
\begin{align}
&		\mathsf{\widehat U} (y, w)  = \int_{\bR^3} \Phi (\widehat P,  Q) u (x (y), q) \, dq \notag\\
& \label{eqA.6.11}
= \underbrace{\bigg|\text{det}  \bigg(\frac{\partial x}{\partial y}\bigg)\bigg|}_{ = \mathcal{J} } \int_{\bR^3} \Phi (\widehat P, \widehat Q) \widehat u (y, w') \, dw', \\
\label{eqA.6.11.1}
&		\mathfrak{ U} (y, w)  = \mathcal{J}    \int_{\bR^3} \Xi (y, w, w')    \widehat u (y, w') \, dw',
 \end{align}
where $\widehat u (y, w') = u (x (y), q (y, w'))$ (cf. \eqref{eq2.0.0}), and
\begin{equation}
                \label{eqA.6.4}
    	\Xi (y, w, w') =  \bigg(\frac{\partial y}{\partial x}\bigg) 
	\Phi (\widehat P, \widehat Q)   \bigg(\frac{\partial y}{\partial x}\bigg)^T.
\end{equation}
Furthermore, by the change of variables $w' \to \bm{R} w'$, 
\begin{equation}
            \label{eqA.6.12}
\mathfrak{ U} (y, \bm{R} w)  	= \mathcal{J}  \int_{\bR^3} \Xi (y, \bm{R} w, \bm{R} w')    \widehat u (y, \bm{R} w') \, dw'.
\end{equation}
Since $u$ satisfies the SRBC (see \eqref{eqA.6.1}), we have 
\begin{equation}
            \label{eqA.6.13}
	\widehat u (y, \bm{R} w) = \widehat u (y, w),  \, \,   \quad \text{if} \, \, y_3 = 0,
\end{equation}
Thus, due to \eqref{eqA.6.11} - \eqref{eqA.6.13}, to prove \eqref{eqA.6.6},  
it suffices to demonstrate that
\begin{equation}
                \label{eqA.6.5}
    \Xi^{i 3} (y, w, w') = - \Xi^{i 3} (y, \bm{R} w, \bm{R} w'), \,  i \in \{1, 2\}, \, \,   \text{whenever} \, \, y_3 = 0.
\end{equation}

\textbf{Verification of \eqref{eqA.6.5}.}
First, by the definition of $\Phi$ in \eqref{eq1.8} - \eqref{eq1.11}, 
\begin{align}
    \label{eqA.6.14}
    \Xi = \frac{\Lambda (\widehat P, \widehat Q)}{\widehat p_0 \widehat q_0}   \bigg(\frac{\partial y}{\partial x}\bigg)  S (\widehat P, \widehat Q) \bigg(\frac{\partial y}{\partial x}\bigg)^T.
\end{align}
We will need the following identities:
\begin{align}
\label{eqA.6.7}
     & 	\widehat p_0   =  \bigg(1 + \bigg|\bigg(\frac{\partial x}{\partial y}\bigg) w\bigg|^2\bigg)^{1/2}, \quad \widehat q= \bigg(1 + \bigg|\bigg(\frac{\partial x}{\partial y}\bigg) w'\bigg|^2\bigg)^{1/2},\\
\label{eqA.6.8.1}
&
	\widehat P \cdot \widehat Q  =    \widehat p_0 \widehat q_0 -  w^T \bigg(\frac{\partial x}{\partial y}\bigg)^T \bigg(\frac{\partial x}{\partial y}\bigg)  w',\\
&
	p (y, w) \otimes q (y, w') = \bigg(\frac{\partial x}{\partial y}\bigg) w (w')^T  \bigg(\frac{\partial x}{\partial y}\bigg)^T,\notag\\
	&
\bigg(\frac{\partial y}{\partial x}\bigg)	p (y, w) \otimes q (y, w') \bigg(\frac{\partial y}{\partial x}\bigg)^T
=  w (w')^T, \notag\\
  & \label{eqA.6.9}
  \bigg(\frac{\partial y}{\partial x}\bigg)  S (\widehat P, \widehat Q) \bigg(\frac{\partial y}{\partial x}\bigg)^T \\
	&=  \big((\widehat P\cdot \widehat Q)^2 - 1\big) \bigg(\frac{\partial y}{\partial x}\bigg)  \bigg(\frac{\partial y}{\partial x}\bigg)^T\notag\\
	& -  (w-w') \otimes (w-w')\notag\\
	& +  (\widehat P\cdot \widehat Q - 1)  (w \otimes w' + w' \otimes w)\notag.
\end{align}
We first handle the factor $\Lambda (\widehat P, \widehat Q)$ (see \eqref{eq1.8}). By \eqref{eqA.6.2} in the assertion $(i)$, we have
\begin{equation}
\label{eqA.6.8}
    f (y, w, w') = f (y, \bm{R} w, \bm{R} w'), \, \, \text{for} \, \, f = \widehat p_0, \widehat q_0, \widehat P \cdot \widehat Q,   \quad \text{if} \, \, y_3 = 0,
\end{equation}
and, hence, \eqref{eqA.6.8} is also true for $f = \Lambda (\widehat P, \widehat Q)$.

We consider the remaining factor on the r.h.s of \eqref{eqA.6.14}. By   \eqref{eqA.6.2} in the assertion $(i)$, we conclude that \eqref{eqA.6.5} holds with $\Xi$ replaced with the left-hand side of \eqref{eqA.6.9}. Thus, \eqref{eqA.6.5} holds, and the desired identity \eqref{eqA.6.6} is valid.
\end{proof}

\begin{lemma}
		\label{lemma A.3}
Let 
$M$ be a nondegenerate  $3$ by $3$ matrix,
and denote
\begin{align}
    \label{A.3.1}
      	\mathfrak{W} (w) = \frac{w}{(1+ |M w|^2)^{1/2}} =: v.
\end{align}
For the sake of convenience, we also denote $v = \mathfrak{W} (w)$.

Then, the following assertions hold.

$(i)$ 
\begin{equation}
            \label{A.3.0}
\underbrace{	| D^{ j } \mathfrak{W}|}_{ = D_w^j  v }  \le N (M) (1+|w|^2)^{-j/2}, j = 1, 2.
\end{equation}

$(ii)$ Let $m \ge 1$ be a  number.
Then, $\mathfrak{W}: \{|w| < m\} \to \bR^3$ is a diffeomorphism onto its image, and
\begin{align}
	\label{A.3.8}
  &\sup_{   \mathfrak{W} (\{|w| < m\}) }  |  \underbrace{ D^j \mathfrak{W}^{-1} }_{ = D_v^j w }|  < N  m^{2j+1}, j = 1, 2, 3. \\
\label{A.3.11}
& 
    \sup_{   \mathfrak{W} (\{|w| < m\}) } |\underbrace{ D \big((D \mathfrak{W}) \circ \mathfrak{W}^{-1}\big)}_{  = D_v \big( (D_w v) (w (v))\big)   }    | \le N m,
\end{align}
where $N  = N (|M|)$.
\end{lemma}

\begin{proof}
$(i)$ Let $c_{i k}$ is the $i k$-th entry of the matrix $M^T M$. Then, by direct computations,
\begin{align}
\label{A.3.4}
	&\frac{\partial v_i}{\partial w_j}  = \frac{\delta_{i j}}{ (1+ |M w|^2)^{1/2} } - \frac{ c_{j l} w_l w_i }{(1 + |M w|^2)^{3/2}  }, \\
	&\frac{\partial^2 v_i}{\partial w_j \partial w_k}  = -\frac{\delta_{i j} c_{ k l }  w_l }{ (1+ |M w|^2)^{3/2} } -  \frac{ c_{j k}  w_i + \delta_{k i} c_{j l}  w_l }{(1 + |M w|^2)^{3/2}  }
	+ \frac{3  c_{j l} c_{k l'}  w_i w_l    w_{l'} }{ 2(1 + |M w|^2)^{5/2}  }. \notag
\end{align}
Combining the above identities with the fact that
$$
	|M w|^2 \ge N (M) |w|^2, 
$$
we prove the first assertion.

$(ii)$  Multiplying both sides of  \eqref{A.3.1} by $M$ gives
\begin{align}
			\label{A.3.7}
&	|M v|^2 =  \frac{|M w|^2}{1+ |M w|^2},  \quad 1 - |M v|^2 =  \frac{1}{1+ |M w|^2},\\
			\label{eqA.3.2}
&	|M w|^2 = \frac{|M v|^2}{1- |M v|^2}, \quad w = \frac{v}{(1- |M v|^2)^{1/2}}.
\end{align}
Note that $1 - |M v|^2$ is bounded away from $0$ on $\mathfrak{W} (\{|w| < m\})$, and hence $\mathfrak{W} : \{|w| < m\} \to \bR^3$ is a  diffeomorphism.

Next, differentiating the second identity in \eqref{eqA.3.2}, we get
\begin{align}
&\label{A.3.5}
	\frac{\partial w_i}{\partial v_j} = \frac{\delta_{i j}}{ (1- |M v|^2)^{1/2} } + \frac{ c_{j l} v_l v_i }{(1- |M v|^2)^{3/2}  },\\
&	\frac{\partial^2 w_i}{\partial v_j \partial v_k}  =  \frac{ P_3 (v) }{ (1 - |M v|^2)^{5/2}  }, \quad
\frac{\partial^3 w_i}{\partial v_j \partial v_k  \partial v_r} =    \frac{P_5 (v)}{ (1- |M v|^2)^{7/2} }, \notag
\end{align}
where $P_3 (v)$ and $P_5 (v)$ are certain polynomials of orders $3$ and $5$ with coefficients bounded by $N (|M|)$.
Then, by \eqref{A.3.7} and \eqref{A.3.5}, and the fact that
\begin{equation}
				\label{A.3.6}
	|v|  = |\mathfrak{W} (w)| \le N (M),
\end{equation}
 for $v \in  \mathfrak{W} (\{|w| < m\})$, we have
\begin{align*}
& \bigg|\frac{\partial w_i}{\partial v_j}\bigg|  \le  N  (1+ |M w|^2)^{3/2} 
	 \le  N  m^3, 
\end{align*}
where $N = N (M)$.
Similarly, we prove the estimates of the second and third-order derivatives.

Finally, to prove the bound \eqref{A.3.11}, we note that by \eqref{A.3.4} and \eqref{A.3.7} -\eqref{eqA.3.2}, 
$$
    \frac{\partial v_i}{\partial w_j}  = (1-|M v|^2)^{1/2} (\delta_{i j} - c_{j l} v_i v_j).
$$    
Differentiating the above expression and using \eqref{A.3.6}, we conclude 
$$
    \big|D_v \frac{\partial v_i}{\partial w_j}\big| \le N   (1-|M v|^2)^{-1/2} \le N m,
$$
so that \eqref{A.3.11} is true.
\end{proof}

\begin{lemma}
			\label{lemma A.5}
Let $n \ge 0$, $G \subset \bR^3$ be the even extension of $\psi (\Omega_{r_0} (x_0))$ across the plane $y_3 = 0$ (see Step 3 in the proof of Lemma \ref{lemma 7.4}).
Let $\mathcal{W}$ and   $\Upsilon_n$ be the mappings given by \eqref{eq2.11.3} and \eqref{eq2.52}, respectively.
Then, the following assertions hold.

$(i)$ The mapping $\Upsilon_n: G \times \{|w| < 2^{n+2}\}$ is  a bi-Lipschitz  homeomorphism onto its image, and
\begin{equation}
            \label{eqA.5.3}
	\Upsilon^{-1}_n (y, v) = (y,  \mathsf{W} (y, v)),
\end{equation}
where
$$
	\mathsf{W} (y, v) = 
\begin{cases}
 \frac{v}{ \big(1 -  \big|\mathsf{M} (y) v\big|^2\big)^{1/2}}, \, \, (y, v) \in \Upsilon_n \big(\psi (\Omega_{r_0} (x_0)) \times \{|w| < 2^{n+2}\}\big), \\   \frac{v}{ \big(1 -  \big|[\mathsf{M} (\bm{R} y)] \bm{R} v\big|^2\big)^{1/2}}, \,\,  (y, v) \in \Upsilon_n \big(G \cap \bR^3_{+} \times \{|w| < 2^{n+2}\}\big), 
\end{cases}
$$
and
$$
    \mathsf{M} (y) = \bigg(\frac{\partial x}{\partial y}\bigg) (y).
$$
Furthermore, for the sake of convenience, we denote
$$
    v = \mathcal{W} (y, w), \quad w = \mathsf{W} (y, v).
$$

$(ii)$ One has
\begin{align}
    	\label{A.5.0.1}
	& \|   \underbrace{  \nabla_y  \mathsf{W}}_{ = \displaystyle \frac{\partial w}{\partial y} }  \|_{  L_{\infty} \big((0, T) \times \Upsilon_n (G \times \{|w| < 2^{n+2}\})\big)   } \le N 2^{  n},	\\
			\label{A.5.0}
	&\|  \underbrace{  \nabla_y \nabla_w \mathcal{W} }_{ = \displaystyle \frac{\partial^2 v}{\partial y \partial w}  }  \|_{  L_{\infty} ((0, T) \times G \times \{|w| < 2^{n+2}\})   } \le N 2^{- n},\\
    	\label{A.5.0.0}
	& \||   \underbrace{ \nabla_y  \nabla_v  \mathsf{W} }_{ =  \displaystyle \frac{\partial^2 w}{\partial y \partial v} }|
 + |  \underbrace{ \nabla_v  \nabla_v  \mathsf{W} }_{ =  \displaystyle  \frac{\partial^2 w}{\partial v^2 } }    |\|_{  L_{\infty} \big((0, T) \times \Upsilon_n (G \times \{|w| < 2^{n+2}\})\big)   } \le N 2^{5 n},
	\end{align}
where  $N = N (\Omega)$.
\end{lemma}

\begin{proof}
$(i)$ First, note that \eqref{eqA.5.3} follows from  \eqref{eqA.3.2}.
We now show that $\Upsilon_n$ is bi-Lipschitz. Since $\Omega$ is a $C^{1, 1}$ domain, we only need to show that
the functions $\mathcal{W}, \mathsf{W}$ 
are continuous across the boundary $\{y_3 = 0\} \times \bR^3$. 
To this end, it suffices to demonstrate that
\begin{align*}
	\bigg|\bigg(\frac{\partial x}{\partial y}\bigg) w\bigg| =  \bigg|\bigg(\frac{\partial x}{\partial y}\bigg) \bm{R} w\bigg| \quad \text{whenever} \, \, y_3 = 0.
\end{align*}
The latter is true thanks to \eqref{eqA.6.3} in Lemma \ref{lemma A.6} $(ii)$.

$(ii)$ Invoke the notation of Lemma \ref{lemma A.3}. 
Let $M (y)$ be either 
$$
	\bigg(\frac{\partial x}{\partial y}\bigg) (y) \, \, \text{or} \, \,   \bigg(\big(\frac{\partial x}{\partial y}\big)(\bm{R} y)\bigg)  \bm{R},
$$
and $C (y) = (c^{i j}, i, j  = 1, 2, 3) := M^T M$, and 
$$
    v (y, w) = \frac{w}{(1+|M w|^2)^{1/2}}.
$$

First, we claim that the functions 
$$
    \bigg(\frac{\partial v }{\partial w}\bigg), \bigg(\frac{\partial w}{\partial v}\bigg)
$$
are continuous across the set $\{y_3 = 0\} \times \bR^3$.
This assertion follows from the explicit expressions of these functions (see \eqref{A.3.4} and \eqref{A.3.5}) and the identity \eqref{eqA.6.3}. Hence, we only need to prove \eqref{A.5.0} - \eqref{A.5.0.0}  away from $\{y_3 = 0\} \times \bR^3$.

Next, by \eqref{eqA.3.2} and \eqref{A.3.7}, whenever $y_3 \neq 0$, we have
$$
    \frac{\partial w_i}{\partial y_{r}} =   \frac{ (\partial_{y_r} c_{l l'}) v_l v_{l'} v_i }{(1-|M v|^2)^{1/2}}
     = (\partial_{y_r} c_{l l'}) v_l v_{l'} v_i (1+|M w|^2)^{1/2},
$$
and this  implies \eqref{A.5.0.1}.
Furthermore, by \eqref{A.3.4} and \eqref{A.3.5}, away from $\{y_3 = 0\}$, the following identities hold:  
\begin{align*}
		&
	\frac{\partial^2 v_i}{\partial w_j \partial y_{r}}  = (\partial_{y_r} c_{l l'}) w_{l} w_{l'} \bigg(-\frac{1}{2}\frac{\delta_{i j}  }{ (1+ |M w|^2)^{3/2} } + \frac{3}{2} \frac{ c_{j l} w_l w_i }{(1 + |M w|^2)^{5/2}  }\bigg),\notag\\
& \quad \quad  \quad \quad \quad -\frac{ (\partial_{y_r} c_{j l}) w_l w_i }{ (1+ |M w|^2)^{3/2} },\\
	& 
\frac{\partial^2 w_i}{\partial v_j \partial y_r} = (\partial_{y_r} c_{l l'} v_{l} v_{l'}) \bigg(\frac{1}{2} \frac{\delta_{i j}}{ (1- |M v|^2)^{3/2} } -  \frac{3}{2}  \frac{ c_{j k} v_k v_i }{(1- |M v|^2)^{5/2}  }\bigg)\\
&\quad \quad  \quad \quad \quad + \frac{ (\partial_{y_r} c_{j l}) v_l v_i }{(1- |M v|^2)^{3/2}  }.
\end{align*}
The first identity implies \eqref{A.5.0}.
Furthermore, the second identity combined with \eqref{A.3.7}, and \eqref{A.3.6} yield
$$
	\bigg|\frac{\partial^2 w_i}{\partial v_j \partial y_r}\bigg| 
	\le N  (1+|M w|^2)^{5/2} \le N (\Omega)  2^{5 n}.
$$
 The bound of 
$\nabla_v  \nabla_v  \mathsf{W}$ follows from \eqref{A.3.8} with $j=  2$.
The assertion $(ii)$ is proved.
\end{proof}

\begin{lemma}
			\label{lemma A.4}
Let $G$ be an  bounded domain and $\psi: G \to \bR^3$ be a diffeomorphism
such that
$$
	\|D \psi\|_{  C (G) } \le N_0, \quad	 \|D (\psi)^{-1}\|_{  C (\psi (G)) } \le N_1.
$$
Let $a$ be a bounded matrix-valued function such that
$$
	\delta_1 |\xi|^2 \le	a^{i j} (v) \xi_i \xi_j \le \delta_2 |\xi|^2, \,\forall v \in G, \xi \in \bR^3.
$$
Denote
$$
	\tilde a  =  (D \psi) (a \circ \psi^{-1}) (D \psi)^T.
$$
Then, $\tilde a$ satisfies 
\begin{equation}
                \label{eqA.4.1}
	\tilde \delta_1 |\xi|^2 \le 	\tilde a^{i j} (w) \xi_i \xi_j \le \tilde \delta_2 |\xi|^2, \, w \in \psi (G), \xi \in \bR^3
\end{equation}
with
$$
	\tilde \delta_1 =  c  \, \delta_1 N_1^{-2}, \quad  \tilde \delta_2 = c^{-1}  \delta_2 N_0^2,
$$
where $c \in (0, 1)$.
\end{lemma}

\begin{proof}
To prove the lower bound, note that 
$$
	\xi^T \tilde a \xi  = ((D \psi)^T \xi) (a \circ \psi^{-1}) ((D \psi)^T \xi)^T \ge \delta_1 |(D \psi)^T \xi|^2 \ge  c \delta_1 N_1^{-2}.
$$
The upper bound follows from the same argument.
\end{proof}

\section{Auxiliary results about the relativistic Landau equation near J\"uttner's solution}

\begin{lemma}[Lemma 6 in \cite{GS_03}]	
				\label{lemma G.1}
				For sufficiently regular functions $f = (f^{+}, f^{-})$, $g = (g^{+}, g^{-})$, $h = (h^{+}, h^{-})$ on $\bR^3$, the following formulas hold:
 \begin{align}
             \label{eqG.1.1}
&	A f  = 2 \nabla_p \cdot (\sigma  \nabla_{p} f)    -  \frac{1}{2} (v(p))^T \sigma v (p)   f + \nabla_p \cdot \big(\sigma v (p)\big) f, \\
     \label{eqG.1.2}
&	K f = - J^{-1/2} (p) \partial_{p_i}  \bigg(J (p) \int \Phi^{ i j } (P, Q)  J^{1/2} (q) \big(\partial_{q_j} f (q)  \\
&+ \frac{ q_j}{2 q_0} f (q)\big) \cdot  \bm{\xi}_0 \, dq\bigg) \bm{\xi}_0, \notag\\
         \label{eqG.1.3}
&	\Gamma_{\pm} (g, h) =   \big(\partial_{p_i} - \frac{ p_i}{2 p_0}\big)   \partial_{p_j} g_{\pm} (p) \int \Phi^{i j} (P, Q) J^{1/2} (q)  h (q) \cdot \bm{\xi}_0 \, dq\\
&	 -   \big(\partial_{p_i} - \frac{ p_i}{2 p_0}\big) g_{\pm} (p) \int \Phi^{i j} (P, Q) J^{1/2} (q) \partial_{q_j}  h (q) \cdot \bm{\xi}_0 \, dq,\notag
      \end{align}
    where $\sigma$ is defined in \eqref{eq6.1.1}.
\end{lemma}

\begin{lemma}[Corollary 4.5 with $\alpha = -3$ in \cite{L_00})] 
				\label{lemma G.2}
Let $\sigma$ be the function defined in \eqref{eq6.1.1}.	Then, the following assertions hold.

$(i)$ There exist constants $N_1, N_2 > 0$ such that for any $\xi \in \bR^3$
\begin{align}
        \label{eqG.2.1}
	N_1 |\xi|^2	\le 	\sigma^{i j} (p) \xi_i \xi_j \le N_2 |\xi|^2.
\end{align}

$(ii)$ For any multi-index $\beta$,
\begin{align}
    \label{eqG.2.2}
	|D^{\beta}_p \sigma (p)| \le N (\beta) p_0^{-|\beta|}.
\end{align}
\end{lemma}

\begin{lemma}
        \label{lemma B.2}
Let $k \ge 0$ be an integer, $r \in (3/2, \infty]$, and  $g \in W^k_r (\bR^3)$.
Then, for
\begin{equation}
        \label{eqB.2.0}
  I (p) =  \int  \Phi^{i j} (P, Q) J^{1/2} (q)   g (q)  \, dq,
\end{equation}
we have
\begin{equation}
            \label{eqB.2.1}
    \|D^k_p I\|_{ L_{\infty} (\bR^3) } \lesssim  \|g\|_{W^k_r (\bR^3) }.
\end{equation}
\end{lemma}

\begin{proof}
    By Theorem 3 in \cite{GS_03} (see p. 281 therein), for any multi-index $\beta = (\beta_1, \beta_2, \beta_3)$,  
\begin{align*}
    &D^{\beta}_p  \int \Phi^{i j} (P, Q) J^{1/2} (q)    g (q) \, dq \\
    &  = \sum_{\beta^1+\beta^2  \le \beta} \int \Theta_{\beta_1} (p, q) \Phi^{i j} (P, Q) J^{1/2} (q) \partial_{\beta_2}  g (q) \phi^{\beta}_{\beta_1, \beta_2} (p, q) \, dq,
\end{align*}
where 
 $$
	\Theta_{\beta_1} (p, q) =  (\partial_{p_1} + \frac{q_0}{p_0} \partial_{q_1})^{\beta_1^1} (\partial_{p_2} + \frac{q_0}{p_0} \partial_{q_2})^{\beta_2^1} (\partial_{p_3} + \frac{q_0}{p_0} \partial_{q_3})^{\beta_3^1},
$$
and $\phi^{\beta}_{\beta_1, \beta_2}$ is a smooth function satisfying the bound
$$
    |\phi^{\beta}_{\beta_1, \beta_2, \beta_3} (p, q)| \lesssim q_0^{|\beta|} p_0^{|\beta_1|- |\beta|}.
$$
By using the above identity, the estimate 
$$
    |\Theta_{\beta_1} (p, q) \Phi (P, Q)|  \le N  p_0^{-|\beta_1|} q_0^7 (1+ |p-q|^{-1})
$$
(see Lemma 2 on p. 277 in \cite{GS_03}),
and  H\"older's inequality with $r \in (3/2, \infty]$ and $r' = r/(r-1) \in [1, 3)$, 
we obtain \eqref{eqB.2.1}.
\end{proof}

The following lemma follows directly from  Lemma 4 on p. 287 in \cite{GS_03}
\begin{lemma}
        \label{lemma B.5}
    For $r \in (3/2, \infty]$, $g \in W^{  1 }_r (\bR^3)$, the following identity holds in the sense of distributions:
   \begin{align}
        \label{eqB.5.1}
     &  \partial_{p_i} \int    \Phi^{ i j } (P, Q)  J^{1/2} (q)   \partial_{q_j}  g ( q) \, dq  \\
     & = -  \partial_{p_i}   \int    \Phi^{ i j } (P, Q)  J^{1/2} (q)  \frac{ q_j}{2 q_0}   g (q)  \, dq  \notag\\
    &  - 4  \int  \frac{P \cdot Q}{p_0 q_0} \bigg((P \cdot Q)^2 - 1\bigg)^{-1/2}  J^{1/2} (q)   g (q) \, dq - \kappa (p) J^{1/2} (p) g (p), \notag
     \end{align}
     where $\kappa (p) = 2^{7/2} \pi p_0 \int_0^{\pi} (1+|p|^2 \sin^2 \theta)^{-3/2} \sin (\theta) \, d\theta$.
\end{lemma}

\begin{lemma}
        \label{lemma A.1}
Let $r \in (3/2, \infty]$, $g = (g^{+}, g^{-}) \in W^{  1 }_r (\bR^3)$ and $a_g$, $C_g$ and $K g$ be given by \eqref{eq6.1} - \eqref{eq7.6}, respectively.
    Then, one has
        \begin{align}
        \label{eqB.1.5}
                  &  \|a_g\|_{ L_{\infty } (\bR^3) } \le     \|g\|_{ W^1_{r} (\bR^3) },  \\
                             \label{eqB.1.3}
          & \|C_g\|_{ L_{\infty } (\bR^3) } \le N   +     N \|g\|_{ W^1_{r} (\bR^3) }, \\
               \label{eqB.1.1}
             &   |K g| (p) \le N J^{1/4} (p) \|g\|_{W^1_r (\bR^3)},
          \end{align}
          where $N  = N (r)$.
\end{lemma}

\begin{proof}
\textit{Estimate of $a_g$.}
The estimate  follows from the definition of $a_g$ (see \eqref{eq6.1}) and \eqref{eqB.2.1} with $k = 0, 1$ (see Lemma \ref{lemma B.2}). 

\textit{Estimate of $C_g$.}
 By the estimates of $\sigma$ in \eqref{eqG.2.2},
  we only need to handle the integral term in \eqref{eq6.2}, 
 which we decompose as follows:
  \begin{align*}
   &\partial_{p_i} \int    \Phi^{ i j } (P, Q)  J^{1/2} (q)   \partial_{q_j}  g ( q) \cdot \bm{\xi}_0 \, dq \\
   &  - \frac{p_i}{2 p_0}  \int  	   \Phi^{ i j } (P, Q)  J^{1/2} (q)  \partial_{q_j}  g (q) \cdot \bm{\xi}_0 \, dq =: C_{g, 1} + C_{g, 2}.
\end{align*}
Next, by the identity \eqref{eqB.5.1}, 
  \begin{align*}
& C_{g, 1} =  - \partial_{p_i} \int    \Phi^{ i j } (P, Q)  J^{1/2} (q)  \frac{ q_j}{2 q_0}   g (q) \cdot \bm{\xi}_0 \, dq \\
    &  - 4  \int  \frac{P \cdot Q}{p_0 q_0} \bigg((P \cdot Q)^2 - 1\bigg)^{-1/2}  J^{1/2} (q)   g (q) \cdot \bm{\xi}_0 \, dq \\
 &    - \kappa (p) J^{1/2} (p) g (p) \cdot \bm{\xi}_0  =:   C_{g, 1, 1}  +   C_{g, 1, 2}  +   C_{g, 1, 3} .
\end{align*}
Applying
  the estimate \eqref{eqB.2.1} with $k = 0, 1$ to the terms $C_{g, 1, 1}$ and $C_{g, 2}$,
  we get
$$
     |C_{g, 1, 1}| + |C_{g,  2}| \leq N \|g\|_{W^1_r (\bR^3)}.
$$
By a simple bound
$$
    P \cdot Q -  1  \ge N_1 \big(\frac{|p-q|^2}{q_0^2} 1_{ |p-q| < (|p|+1)/2 } + \frac{p_0}{q_0}  1_{ |p-q| \ge (|p|+1)/2 }\big)
$$
and  H\"older's inequality,
$$
    |C_{g, 1, 2}| \le  N \|g\|_{ L_r (\bR^3)}.
$$
Finally, we note that the last inequality also holds for $C_{g, 1, 3}$ 
since $\kappa$ is a bounded function. Thus, \eqref{eqB.1.3} is valid.

\textit{Estimate of $K g$.}
First, we split  the integral in \eqref{eqG.1.2} as follows: 
	\begin{align*}
	&	K g  = (\partial_{p_i} p_0) J^{1/2} (p) \int \Phi^{i j} (P, Q) J^{1/2} (q) (\partial_{q_j} g (q) + \frac{q_j}{2} g (q)) \cdot \bm{\xi}_0 \, dq \,  \bm{\xi}_0\\
	&
		 -  J^{1/2} (p) \partial_{p_i}  \int \Phi^{i j} (P, Q) J^{1/2} (q)   \frac{q_j}{2 q_0} g (q) \cdot \bm{\xi}_0 \, dq\,  \bm{\xi}_0  \\
   & -  J^{1/2} (p) \partial_{p_i} \int  \Phi^{i j} (P, Q) J^{1/2} (q)  \partial_{q_j} g (q) \cdot \bm{\xi}_0 \, dq\,  \bm{\xi}_0 =: K_1 + K_2 + K_3.
	\end{align*}
 We observe that the following terms are similar:
 \begin{itemize}
     \item[--] $K_1$ and $C_{g, 2}$, 
     \item[--] $K_2$ and $C_{g, 1, 1}$, 
     \item[--] $K_3$ and $C_{g, 1}$. 
 \end{itemize}
  Hence, the estimate \eqref{eqB.1.1} is proved by repeating the above argument.
  \end{proof}

\begin{lemma}
			\label{lemma A.2}
Let $g$ be a function  satisfying Assumption \ref{assumption 1.4} (see \eqref{eq1.4.1} - \eqref{eq1.4.2}).
Then, for $\sigma_g$ defined in \eqref{eq6.0}, one has
\begin{align}
    \label{eqA.2.1}
    \|\nabla_p \sigma_g\|_{ L_{\infty} (\Sigma^T) } +	\|\sigma_g\|_{   L_{\infty} 
 ((0, T)) C^{\varkappa/3, \varkappa}_{x, p} (\Omega \times \bR^3 )   } \le N (K).
\end{align}
\end{lemma}

\begin{proof}
First, note that the estimate of $\nabla_p \sigma_g$ follows directly from \eqref{eqB.2.1} with $k=1$ and the assumption \eqref{eq1.4.2}. 
Furthermore, for any $t \ge 0$ and $x_1, x_2 \in \Omega$, $p \in \bR^3$,
$$
	\sigma_g (t, x_1, p) - \sigma_g (t, x_2, p) = \int \Phi (P, Q) J^{1/2} \big(g (t, x_1, q) - g (t, x_2, q)\big) \, dq.
$$
Then, by \eqref{eqB.2.1} with $k=0$ and the assumption  \eqref{eq1.4.1}, 
\begin{align*}
	|\sigma_g (t, x_1, p) - \sigma_g (t, x_2, p)|
&  \le N \sup_{ p \in \bR^3}  |g (t, x_1, p) - g (t, x_2, p)|\\
&	\le N |x_1 - x_2|^{\varkappa/3}. 
\end{align*}
Now the assertion follows from the above inequality, the $L_{\infty}$ estimate of $|\nabla_p \sigma_g|$, and the interpolation inequality for H\"older spaces.
\end{proof}

\begin{lemma}[cf. Lemma 7 of \cite{GS_03}]
				\label{lemma 2.4}
For any $\theta \ge 0$,  there exists $\kappa > 0$ such that for any and 
any $g = (g^{+}, g^{-}) \in  W^1_{2, \theta} (\bR^3), h = (h^{+}, h^{-}) \in  L_{2, \theta} (\bR^3)$,
\begin{align}
    \label{eq2.4.1}
    -  \langle  A g,  g \, p_0^{2\theta} \rangle \, \ge  \kappa \|\nabla_p g\|^2_{L_{2, \theta} (\bR^3) }  -  N (\theta) \|g\|^2_{L_{2} (\bR^3)}.
\end{align}
Furthermore, for any $\varepsilon \in (0, 1)$,
\begin{align}
    \label{eq2.4.2}
	\bigg|\int_{\bR^3} (K g) \cdot h \, p_0^{2\theta} dp\bigg|
\le  \varepsilon \|g\|^2_{ W^1_{2}(\bR^3) } + N (\theta) \varepsilon^{-1}  \|h\|^2_{ L_{2} (\bR^3) }.
\end{align}
\end{lemma}

\begin{proof}
In the case when $\theta = 0$, the estimate \eqref{eq2.4.1} is proved in Lemma 7 in \cite{GS_03}. The case $\theta > 0$ is handled by the same argument, and hence, we omit the proof. The bound \eqref{eq2.4.2} follows from \eqref{eqB.1.1} in Lemma \ref{lemma A.1}.
\end{proof}

\begin{lemma}
        \label{lemma G.4}
 For  sufficiently regular functions $f_j = (f^{+}_j, f^{-}_j)$, $j = 1, 2, 3,$ on $\bR^3$ and any  $r \in (3/2, \infty]$ and $\theta \ge 0$, we have
\begin{align}
    \label{eqG.4.1}
    &  \big|\langle   \Gamma (f_1, f_2),  f_3  p_0^{2 \theta} \rangle\big| \\
   & \lesssim_{\theta} \|\nabla_p f_1 \|_{  L_{2, \theta} (\bR^3) } \|f_2\|_{  L_{r} (\bR^3) } \|f_3\|_{  W^1_{2, \theta} (\bR^3) }\notag\\
      & + \|f_1 \|_{  L_{2, \theta} (\bR^3) } \|\nabla_p f_2\|_{  L_{r} (\bR^3) } \|f_3\|_{  W^1_{2, \theta} (\bR^3) }. \notag
\end{align}
\end{lemma}

\begin{proof}
Invoke the explicit expression of $\Gamma (f, g)$ in \eqref{eqG.1.3}. For the sake of simplicity, we assume that $g$ and $h$ are scalar functions, and we estimate a simplified integral  given by (cf. \eqref{eqG.4.1})
\begin{align*}
&     I =    \langle  \big(\partial_{p_i} - \frac{ p_i}{2 p_0}\big)    \int \Phi^{i j} (P, Q) J^{1/2} (q) \partial_{p_j} f_1 (p)  f_2 (q)  \, dq, f_3 p_0^{2 \theta} \rangle\\
& - \langle  \big(\partial_{p_i} - \frac{ p_i}{2 p_0}\big)    \int \Phi^{i j} (P, Q) J^{1/2} (q) f_1 (p)  \partial_{q_j} f_2 (q)  \, dq, f_3 p_0^{2 \theta}\rangle. \notag
\end{align*}
Integrating by parts in $p$ gives
\begin{align*}
  &  I =     \langle \partial_{p_j} f_1    \int \Phi^{i j} (P, Q) J^{1/2} (q)  f_2 (q)  \, dq, (-\partial_{p_i} - \frac{ p_i}{2 p_0}) (f_3 p_0^{2 \theta}) \rangle \\
 & + \langle  f_1    \int \Phi^{i j} (P, Q) J^{1/2} (q)   \partial_{q_j} f_2 (q)  \, dq, (\partial_{p_i}  + \frac{ p_i}{2 p_0}) (f_3  p_0^{2 \theta})\rangle 
  =: I_1+I_2. \notag
\end{align*}
Finally, applying the $L_2-L_{\infty}-L_2$ H\"older's inequality to $I_1$ and $I_2$ using the bound \eqref{eqB.2.1} with $k=0$, we obtain \eqref{eqG.4.1}.
\end{proof}

\section{Verification of estimates $(\ref{eq2.58})$ - $(\ref{eq2.47})$}
                                \label{section }
\begin{lemma}
            \label{lemma A.7}
Estimates \eqref{eq2.58} - \eqref{eq2.47.2} are true.            
\end{lemma}
\begin{proof}
\textit{Ellipticity and boundedness of the leading coefficients.} 
\begin{enumerate}

\item Bounds of $\mathcal{A}$ (see \eqref{eq2.11}). By Lemma  \ref{lemma 6.1} and \eqref{eqA.4.1}, for sufficiently small $r_0$,  
$$
	(\delta_0/4)  |\xi|^2 \le	A^{i j} (z) \xi_i \xi_j \le (4 \delta^{-1}_0) |\xi|^2, \,  \forall z \in \bR^7_T, \xi \in \bR^3, 
$$
where $A$ is defined in \eqref{eq2.0}, 
and hence, the same estimate also holds for $\mathcal{A}$. 

\item Bounds of $\mathfrak{A}$ (see \eqref{eq2.10}). First, we estimate $\mathbb{A}$ via  Lemma \ref{lemma A.4} with $\psi = \mathcal{W}_y$. By \eqref{A.3.0} - \eqref{A.3.8} with  $m = 2^{n+2}$, 
the assumptions of  Lemma \ref{lemma A.4} hold with $N_0 = N_0' 2^{-n}$, $N_1 = N_1' 2^{3 n}$, and $\delta_1 = \delta_0/4$. Then, by \eqref{eqA.4.1}), we conclude that for any $\xi \in \bR^3$ and $z \in (0, T) \times \Upsilon_n (G \times \{|p| < 2^{n+2}\}$, 
\begin{equation}
			\label{eqA.7.1}
	N' 2^{- 6 n}    \le  \mathbb{A}^{i j}  (z) \xi_i \xi_j \le N'',
\end{equation}
and hence, the same bounds (with, perhaps, different constants $N'$ and $N''$) are true for $\mathfrak{A}$.
\end{enumerate}

\textit{Boundedness of $\nabla_v \mathfrak{A}$.}
\begin{enumerate}

\item Estimate of $\nabla_w \mathcal{A}$. By \eqref{eqA.2.1} in Lemma \ref{lemma A.2} and the construction of the coefficients $A$ (see \eqref{eq2.0}), 
\begin{equation*}
	\|\nabla_w A\|_{   L_{\infty} ((0, T) \times \psi (\Omega_{r_0} (x_0)) \times \bR^3)   } \le N (K, \Omega).
\end{equation*}
  Then, by \eqref{eq2.11} we have 
\begin{equation}
						\label{eqA.7.2}
	\|\nabla_w \mathcal{A}\|_{   L_{\infty} ((0, T) \times  G \times \bR^3)   } \le N (K, \Omega).
\end{equation}

\item Estimate of $\mathfrak{A}$. First, we estimate $\mathcal{\hathat A}$, which is given \eqref{eq2.22}. By   \eqref{A.3.11}) with $m=2^{n+2}$,
\begin{equation}
            \label{eq2.16.1}
	\|\nabla_v \big(\frac{\partial v }{\partial w}\big) (y, w (y, v))\|_{   L_{\infty} (\Upsilon (G \times \{|w| < 2^{n+2}\}))  } \le N 2^{  n},
\end{equation}
 Then, by   the chain rule and \eqref{eqA.7.2} - \eqref{eq2.16.1}, one has
$$
\|\nabla_v \mathcal{\hathat A}\|_{   L_{\infty} ((0, T) \times \Upsilon (G \times \{|w| < 2^{n+2}\}))  } \le N 2^{ n}.
$$

Next,  recall that $\mathbb{A}$ is defined in \eqref{eq2.9}. Combining \eqref{eq2.15} (with $j=1$) and \eqref{eq2.16.1} with the last inequality, we get
\begin{align*}
	\|\nabla_v \mathbb{A}\|_{   L_{\infty} ((0, T) \times \Upsilon (G \times \{|w| < 2^{n+2}\}))  } \le N.
\end{align*}
Then, by the definition of $\mathfrak{A}$ (see \eqref{eq2.10}) and the last inequality    gives the desired bound \eqref{eq2.18}, that is
$$
    \|\nabla_v \mathfrak{A}\|_{   L_{\infty}  ((0, T) \times  \bR^6)  } \le N.
$$
\end{enumerate}

\textit{H\"older continuity   of the leading coefficients.}
Here we verify \eqref{eq2.23}.

\begin{enumerate}

\item Estimate of $\mathcal{A}$ (see \eqref{eq2.11}). First,  by  the definition of $A$ in \eqref{eq2.0} and \eqref{eqA.2.1} in Lemma \ref{lemma A.2}, we have
$$
	\|A\|_{   L_{\infty} ((0, T)) C^{\varkappa/3, \varkappa}_{x, p} (\psi (\Omega_{r_0} (x_0)) \times \bR^3)   } \le N (K, \Omega, \varkappa).
$$
To show that $\mathcal{A}$  is H\"older continuous, that is, 
\begin{equation}
			\label{eq2.20}
		\|\mathcal{A}\|_{   L_{\infty} ((0, T)) C^{\varkappa/3, \varkappa}_{x, p} (G \times \bR^3)   } \le N (K, \Omega, \varkappa),
\end{equation}
it suffices to check that 
\begin{equation}
			\label{eq2.6}
	\mathcal{A} \quad \text{is continuous across} \, \,  \{y_3 = 0\} \times \bR^3.
\end{equation}
Note that for any arbitrary $3$ by $3$ symmetric matrix $M = (m^{i j}, i, j = 1, 2, 3)$, 
$$
	\bm{R} M \bm{R} = \begin{pmatrix}
		m^{11} & m^{1 2} & -m^{1 3}\\
		m^{12} & m^{2 2} & -m^{2 3}\\
		-m^{13} & -m^{ 2 3} & m^{3 3}
		\end{pmatrix}.
$$
Then, by the definition of $\mathcal{A}$ (see \eqref{eq2.11}), if  the identity
\begin{equation}
			\label{eq2.19}
		 A^{ i 3} (t, y_1, y_2, 0, w) =  -  A^{ i 3} (t, y_1, y_2, 0,  \bm{R} w),  i  \in  \{1, 2\}
\end{equation}
is valid, then, \eqref{eq2.6} is also true. The identity \eqref{eq2.19} follows from Lemma \ref{lemma A.6} because  Assumption \ref{assumption 1.5} is valid as $g$ satisfies the SRBC (see \eqref{eq1.5.1}).

\item Estimate of $\mathfrak{A}$. First, we estimate $\mathcal{\hathat A}$ (see \eqref{eq2.22}).  By  \eqref{A.3.8} with $m = 2^{n+2}$ and \eqref{A.5.0.1}, we have
\begin{equation}
\begin{aligned}
			\label{eq2.21}
   & 	\|\frac{\partial w}{\partial v}\|_{  L_{\infty} (\Upsilon (G \times \{|w| < 2^{n+2}\}))   }  \le N (\Omega) 2^{ 3 n}, \\
   &    \|\frac{\partial w}{\partial y}\|_{  L_{\infty} (\Upsilon (G \times \{|w| < 2^{n+2}\}))   } \le N (\Omega) 2^{  n}.
     \end{aligned}
\end{equation}
By the definition of $\mathcal{\hathat A}$ in \eqref{eq2.22} and \eqref{eq2.20}, and \eqref{eq2.21},  we have
\begin{equation}
						\label{eq2.24}
		\|\mathcal{\hathat A} \|_{   L_{\infty} ((0, T)) C^{\varkappa/3, \varkappa}_{y, v} (\Upsilon (G \times \{ |w| < 2^{n+2}\}))   } \le  N (K, \Omega, \varkappa)  2^{ 3 n}.
\end{equation}

Next, we estimate the H\"older norm of $\mathbb{A}$. 
First, we  need to bound the H\"older norm of 
$$
    \bigg(\frac{\partial v}{\partial w}\bigg) (y, w (y, v)).
$$
By  the chain rule, for any $(y, v) \in \Upsilon (G \times \{ |w| < 2^{n+2}\})$,
\begin{equation}
            \label{eq2.24.1}
\begin{aligned}
   & |\nabla_y \big(\frac{\partial v}{\partial w}\big) (y, w (y, v))| \le  \|\nabla_y \frac{\partial v}{\partial w}\|_{  L_{\infty} (G \times \{|w| < 2^{n+2}\}) } \\
   & + \|\nabla_w \frac{\partial v}{\partial w}\|_{ L_{\infty} (G \times \{|w| < 2^{n+2}\}) } \|\frac{\partial w}{\partial y}\|_{ L_{\infty} (\Upsilon_n (G \times \{|w| < 2^{n+2}\})) }
\end{aligned}
\end{equation}
Hence, by \eqref{A.5.0} in Lemma  \ref{lemma A.5}  $(ii)$, the first term on the l.h.s. of \eqref{eq2.24.1} is bounded by $N (\Omega) 2^{-n}$. Furthermore, \eqref{A.3.0}) and  \eqref{A.5.0.1}, we conclude that the second term on the r.h.s of \eqref{eq2.24.1} is also bounded by $N 2^{-n}$, and hence, 
\begin{equation}
			\label{eq2.25}
	\|\nabla_y \big(\frac{\partial v}{\partial w}\big) (y, w (y, v))\|_{  L_{\infty} (\Upsilon (G \times \{|w| < 2^{n+2}\}))   }   \le N (\Omega) 2^{ - n}.
\end{equation}
Furthermore, by the definition of $\mathbb{A}$ in \eqref{eq2.9},
 \eqref{eq2.24}, and  \eqref{eq2.16.1} and \eqref{eq2.25}, we obtain 
\begin{equation}
			\label{eq2.25.2}
		\|\mathbb{A}\|_{   L_{\infty} ((0, T)) C^{\varkappa/3, \varkappa}_{y, v} (\Upsilon (\bR^3 \times \{ |w| < 2^{n+2}\}))   } \le  N (K, \Omega, \varkappa)  2^{   n}.
\end{equation}
Finally, by the definition of $\mathfrak{A}$ (see \eqref{eq2.10}), the bound  \eqref{eq2.25.2},  and our choice of the cutoff function $\zeta_n$ (see \eqref{eq2.17}),  we obtain 
$$
\|\mathfrak{A}\|_{   L_{\infty} ((0, T)) C^{\varkappa/3, \varkappa}_{y, v} (\bR^6)} \le  N (K, \Omega, \varkappa)  2^{   n}
$$
(see \eqref{eq2.23}).

\textit{Estimates of the lower-order terms.} Invoke the definition $\mathbb{B}$ (see \eqref{eq2.0}, \eqref{eq2.11.2}, \eqref{eq2.22},  \eqref{eq2.9.2}). By the assumption \eqref{eq1.2.2} and \eqref{eq2.0}, and \eqref{eq2.11.2},   we have
\begin{align*}
    \|\mathcal{B}\|_{ L_{\infty} ((0, T) \times G \times \bR^3) },
\end{align*}
and, then, by the first inequality in \eqref{eq2.15}, we obtain \eqref{eq2.45}.

Next, recall the definition of $\mathbb{X}$ (see \eqref{eq2.0.1}, \eqref{eq2.11.1}, \eqref{eq2.22}, \eqref{eq2.9.1}).  Note that by \eqref{eq2.0.1},  for any $(y, w) \in  \psi (\Omega_{3r_0/4}) \times \{|w| <  2^{n+2} \}$,
\begin{equation}
                \label{eq2.51}
\begin{aligned}                
    &|X (y, w)| + |\nabla_w X (y, w)| \\
    & \le  N (\Omega) (1+|w|^2)^{1/2}  \le N 2^n.
\end{aligned}    
\end{equation}
Hence, by \eqref{eq2.11.1}, the same bound is true for $\mathcal{X}$. Furthermore, by the definition of $\mathbb{X}$ (see \eqref{eq2.9.1}) and  the first inequality in \eqref{eq2.15}, we get 
$$
    \|\mathbb{X}\|_{ L_{\infty}  (\Upsilon (G \times \{|w| < 2^{n+2}\})) }  \le N (\Omega).
$$
Next, recall the definition of $\mathcal{\hathat X}$ (see \eqref{eq2.22}). 
By the chain rule,  \eqref{eq2.51} and \eqref{eq2.21} (cf. \eqref{eq2.24}), we get
\begin{equation}
                \label{eq2.51.1}
    \|\nabla_v \mathcal{\hathat X}\|_{ L_{\infty} (\Upsilon (G \times \{|w| < 2^{n+2}\})) }  \le N 2^{ 3 n}.
\end{equation}
Finally, by the definition of $\mathbb{X}$ (see \eqref{eq2.9.1}),  \eqref{eq2.51.1} and \eqref{eq2.16.1}, and the first inequality in \eqref{eq2.15}, we conclude that the bound \eqref{eq2.47} is valid, that is 
$$
    \|\nabla_v \mathbb{X}\|_{ L_{\infty} (\Upsilon (G \times \{|w| < 2^{n+2}\})) }   \le N 
 2^{ 2  n}.
$$

Finally, we estimate the second `geometric' coefficient $\mathbb{G}$ (see \eqref{eq2.9.1.1}). 
By \eqref{eq2.21} and the first inequality in  \eqref{eq2.15},
$$
    \|\mathbb{G}\|_{ L_{\infty} (\Upsilon (G \times \{|w| < 2^{n+2}\})) } \le N.
$$
Furthermore, by differentiating \eqref{eq2.9.1.1} and using the estimates \eqref{eq2.16.1} and \eqref{eq2.21}, and \eqref{A.5.0.0} combined with the 
the first inequality in  \eqref{eq2.15}, we conclude
$$
    \|\nabla_v \mathbb{G}\|_{ L_{\infty} (\Upsilon (G \times \{|w| < 2^{n+2}\})) } \le N 2^{ 4 n}.
$$
\end{enumerate}
\end{proof}

\section{Relativistic kinetic transport equation in a domain}
                                                                    \label{section 9}
All the assertions here are either contained in \cite{BP_87} or \cite{U_86}, or can be easily proved by adapting the arguments therein.
We start by introducing the relativistic counterpart of the set of test functions in \cite{BP_87} (see Definition \ref{definition C.2}).

\begin{definition}
			\label{definition C.1}
We say that $G \subset \Sigma^T \cup \Sigma^T_{\pm}$ is a good set if there is a positive lower bound of the length of the characteristic lines  $(t+s, x + v (p) s, p)$ inside $\Sigma^T \cup \Sigma^T_{\pm}$ that intersect  $G$.
\end{definition}

\begin{definition}
                \label{definition C.2}
Let $\mathsf{\Phi}$ be the set of functions $\phi$ on $\Sigma^T$ such that 
\begin{itemize}[--]
    \item $\phi$ is continuously differentiable along the  characteristic lines $(t+s, x + v (p) s, p)$,

    \item $\phi$, $Y \phi$ are bounded functions on $\Sigma^T$,

    \item the support of $\phi$ is a  bounded  good set.
\end{itemize}
\end{definition}

\begin{remark}
		\label{remark 2.2.4}
By following the argument of Lemma 2.1 in \cite{G_93}, one can show that
$$
	C^1_0 \bigg(\overline{\Sigma^T}
\setminus \big(((0, T) \times \gamma_0) \cup (\{0\} \times \partial \Omega \times \bR^3)   \cup (\{T\} \times \partial \Omega \times \bR^3)\big)\bigg) \subset \mathsf{\Phi}.
$$
\end{remark}

\begin{definition}
                \label{definition C.3}
For $r \in [1, \infty)$, we say that $\xi \in L_{r,  \theta, \text{loc} } (\Sigma^T_{\pm})$ if for any good set $G$,
one has $\xi 1_{G} \in L_{r, \theta} (\Sigma^T_{\pm})$. 
\end{definition}

To define the traces of functions on $\Sigma^T$, we need the following assertion, which is similar to Proposition 1 in \cite{BP_87}.
\begin{proposition}
			\label{proposition C.1}
Let $r \in [1, \infty)$,  $\theta \ge 0$ be numbers,
and $u \in L_{r, \theta} (\Sigma^T)$ be a function 
such that $Y u \in L_{r, \theta} (\Sigma^T)$.
Then, there exist unique functions $u_{\pm}, u_T, u_0$ on $\Sigma^T_{\pm}$ and $\Omega \times \bR^3$, respectively, such that
\begin{itemize}[--] 
\item $u_{\pm} \in L_{r, \theta, \text{loc} } (\Sigma^T)$, $u_T, u_0 \in L_{r, \theta} (\Omega \times \bR^3),$

\item the following Green's identity holds for any $\phi \in \Phi$:
 	\begin{equation}
				\label{eqC.1.1}
 	\begin{aligned}
 	    &		\int_{\Sigma^T}  (Y u)  \phi
  +   (Y \phi)    u \, dz
	\\
	  &
		 =  \int_{\Omega \times \bR^3 }       u_T (x, p) \phi (T, x, p) \, dx dp
		 -
		 \int_{\Omega \times \bR^3 }     u_0 (x, p)  \phi (0, x, p) \, dx dp
	  \\
 	   &
		\quad +  \int_{\Sigma^T_{+}}    u_{+} \phi \, |v (p) \cdot n_x| \,  dS_x dp  dt
	-  \int_{\Sigma^T_{-}}    u_{-} \phi \, |v (p) \cdot n_x| \,  dS_x dp  dt.
	\end{aligned}
	   	\end{equation}
\end{itemize}
\end{proposition}

\begin{definition}
			\label{definition C.4}
\index{$f_{\pm}$} Such functions $u_{\pm}, u_T, u_0$ are called the traces of a function $u$.
\end{definition}

The next lemma shows that $u_{\pm}$ belongs to a certain weighted Lebesgue space (see \cite{U_86}).
\begin{lemma}[Ukai's trace lemma]
    \label{lemma C.6}
For $r \ge 1$ and $u$ be such that $u, Y u \in L_r (\Sigma^T)$. Then, we have $u_{\pm} \in L_r (\Sigma^T_{\pm}, w |v (p) \cdot n_x|)$, where  $w (z) = \min \{1, l (z)\}$, and $l (z)$ is the length of the characteristic line $(t + s, x + v (p) s, p)$ inside $\Sigma^T \cup \Sigma^T_{\pm}$, and, in addition, 
\begin{equation}
        \label{eqC.6.1}
    \|u\|_{ L_r (\Sigma^T_{\pm}, w |v (p) \cdot n_x|) } \le N  \|Y u\|_{ L_r (\Sigma^T)}
    + N \|u\|_{ L_r (\Sigma^T)},
\end{equation}
where $N = N (r, T)$, and the weighted Lebesgue space on the l.h.s. is defined in \eqref{1.2.0.0}.
\end{lemma}




\begin{proposition}[see Theorem 5.1.2 in \cite{U_86}]
			\label{proposition C.2}
Let $r \in [1, \infty), \theta \ge 0$ be numbers  and $u$ and $\phi$ be the functions in the following class: 
\begin{itemize}
    \item $u,  Y u \in L_{r, \theta} (\Sigma^T)$, 
    \item either $u_0$ or $u_T$ belongs to $L_{2, \theta} (\Omega \times \bR^3)$, 
    \item  either $u_{+}$ or $u_{-}$ belongs to $L_{2, \theta} (\Sigma^T_{\pm}, |v (p) \cdot n_x|)$.
\end{itemize}
Then, we have 
\begin{equation}
		\label{eqC.2.1}
\begin{aligned}
    &    \int_{\Omega \times \bR^3} (u_T \phi_T (x, p)  - u_0 \phi_0 (x, p)) \, p_0^{\theta r} dx dp\\
&
   +  \int_{\Sigma^T_{+}}    u_{+} \phi_{+}   p_0^{\theta r}\,  |v (p) \cdot n_x|  \,  dS_x dp dt
    -  \int_{ \Sigma^T_{-} }     u_{-} \phi_{-}  p_0^{\theta r}\,  |v (p) \cdot n_x|  \,  dS_x dp dt\\
 &
  =
     \int_{\Sigma^T} \big((Y f) \phi + (Y \phi) f\big)  p_0^{\theta r} \, dz.
 \end{aligned}
\end{equation}
\end{proposition}

The following lemma shows that one can drop the `strong' integrability conditions on the traces $u$ and $\phi$ on $\Sigma^T_{\pm}$ in Proposition \ref{proposition C.2} if $u$ and $\phi$ satisfy the SRBC.
\begin{lemma}
            \label{lemma C.5}
We assume that 
\begin{itemize}
    \item[--]  $u, \phi, Y u, Y \phi \in L_{2, \theta} (\Sigma^T)$, 
    \item[--]  either $u_0, \phi_0  \in L_{2, \theta} (\Omega \times \bR^3)$ or the same holds for $u_T, \phi_T$, 
    \item[--] $u$ and $\phi$  satisfy the SRBC.
  \end{itemize}
Then, the following variant of the energy identity holds:
\begin{align}
    \label{eqC.5.1}
   & \int_{ \Omega \times \bR^3 }  \big((u \phi) (T, x, p) -   (u \phi) (0, x, p)\big) \, p_0^{\theta r} dxdp \\
   & =  \int_{\Sigma^T} \big(u (Y \phi) + \phi (Y u)\big) \, p_0^{\theta r} dz. \notag
\end{align}
In addition, $u, \phi \in C ([0, T]) L_2 (\Omega \times \bR^3)$.
\end{lemma}

\begin{proof}
We repeat the argument of Lemma 3.7 in \cite{DGY_21}. The key idea is to cut off away from the grazing set so that the traces on $\Sigma^T_{\pm}$ of the regularized function $\phi_{\varepsilon}$ are of class $L_{2} (\Sigma^T, p_0^{2 \theta} |v \cdot n_x|)$ and they satisfy the SRBC. Then, the  Green's identity \eqref{eqC.2.1} is applicable. We list below a few minor modifications in the argument of the aforementioned lemma.
\begin{itemize}
    \item[--]  We note that the integrals over $\Sigma^T_{\pm}$ cancel out thanks to the SRBC.

    \item[--] One needs modify   integrals $I_2$ and $I_3$:
     \begin{align*}
&	I_2 =   I_{2, 1} + I_{2, 2} : =   \int_{\bH^T_{-}}  (W \cdot \nabla_y \widehat \phi) \xi_{\varepsilon} (y, w) \xi \bigg(\frac{t}{\varepsilon}\bigg) \xi\bigg(\frac{T-t}{\varepsilon}\bigg) \widehat u \, dydwdt  \\
	&\quad \quad  +  2\int_{   \bH^T_{-}:  \varepsilon^2 < y_3^2 + w_3^2 < 2 \varepsilon^2 }  \widehat u  \widehat \phi   \xi \bigg(\frac{t}{\varepsilon}\bigg) \xi \bigg(\frac{T-t}{\varepsilon}\bigg)   \frac{ W_3  y_3}{\varepsilon^2}
	\, \xi' \bigg(\frac{y_3^2 + w^2_3}{\varepsilon^2}\bigg) \,  dydwdt,\\
&
    I_3 =  I_{3, 1} + I_{3, 2} : =   - \int_{ \mathbb{H}^T_{-} }  (X \cdot \nabla_w \widehat \phi)\xi_{\varepsilon} (y, w)   \xi \bigg(\frac{t}{\varepsilon}\bigg) \xi \bigg(\frac{T-t}{\varepsilon}\bigg)    \widehat u \,  dydwdt\\
& \quad \quad - 2 
    \int_{  \bH^T_{-}: \varepsilon^2 < y_3^2 + w_3^2 < 2 \varepsilon^2 }
     \widehat u  \widehat \phi \xi \bigg(\frac{t}{\varepsilon}\bigg) \xi \bigg(\frac{T-t}{\varepsilon}\bigg)
   \frac{ X_3  w_3}{\varepsilon^2} \,  \xi' \bigg(\frac{ y_3^2 + w_3^2}{\varepsilon^2}\bigg) \,  dydwdt,
\end{align*}
where $W$ and $X$ are defined by  \eqref{eqD.1}. Repeating the argument on p. 489 in \cite{DGY_21}, we conclude
\begin{align*}
	\lim_{\varepsilon \to 0} (I_{2, 2} + I_{3, 2}) =0
\end{align*}
\end{itemize}
 The rest of the proof is the same as in Lemma 3.7 in \cite{DGY_21}.
\end{proof}

\section{Verification of the identities $\ref{eq2.1}$ and $\ref{eq2.14}$}
			\label{Appendix E}
\subsection{Identity  $\ref{eq2.1}$}
                    \label{section E.1}
Let $\Omega$ be a $C^2$ bounded domain, and
$$ 
	\psi: \Omega_{r_0} (x_0) \to \bR^3
$$ 
be a local $C^2$ diffeomorphism, and  
$\mathsf{\Psi}: (x, p) \to (y, w)$ be a mapping given by
$$
	y = \mathsf{\psi} (x), \quad w = (D \mathsf{\psi})\,   p.
$$
For a function $f$ vanishing outside $(0, T) \times \Omega_{r_0} (x_0)  \cap \bR^3$, we set 
$$
	\widehat f (y, w)   = f (\mathsf{\Psi}^{-1} (y, w)) = f (x (y), p (y, w)).
$$

We compute the transport term $Y$ in the $(t, y, w)$ variables. 
We repeat the calculations of Appendix A in \cite{DGY_21} with minor modifications.

Let $f \in L_{1, \text{loc}} (\bR^{7}_T), \phi  \in C^{0, 1}_0 (\bR^7_T)$ be functions such that $f (\cdot, x, \cdot), \phi (\cdot, x, \cdot) = 0$ for $x \not \in \Omega_{r_0}$.
 Using the chain rule gives
\begin{align*}
  &   (\nabla_p \phi) (t, x (y), p (y, w))    = \bigg(\frac{\partial y}{ \partial x}\bigg)^T \nabla_w \widehat \phi (t, y, w),\\
  &     \big(\nabla_x \phi\big) (t, x (y), p (y, w))\\
 &   = 	\bigg(\frac{\partial y}{ \partial x}\bigg)^T   \big[\nabla_y  	\widehat \phi (t, y, w)
	  -  \bigg(\frac{\partial p}{\partial y}\bigg)^T (\nabla_p \phi) (t, x (y), p (y, w))\big] \\
&
     = \bigg(\frac{\partial y}{ \partial x}\bigg)^T  \nabla_y  	\widehat \phi (t, y, w)
	  - \bigg(\frac{\partial y}{ \partial x}\bigg)^T \bigg(\frac{\partial p}{\partial y}\bigg)^T \bigg(\frac{\partial y}{ \partial x}\bigg)^T \nabla_w \widehat \phi  (t, y, w).
 \end{align*}
Therefore, 
\begin{align*}
    &   \bigg(\frac{p^T}{ \sqrt{1+|p|^2} }\bigg) (y, w)  \big(\nabla_x \phi\big) (t, x (y), p (y, w))  \\
	& =  \frac{w^T}{ \sqrt{1+|\frac{\partial x}{ \partial y} w|^2} }  \nabla_y  	\widehat \phi (t, y, w)   -  \frac{w^T}{ \sqrt{1+|\frac{\partial x}{ \partial y} w|^2} }   \bigg(\frac{\partial p}{\partial y}\bigg)^T \bigg(\frac{\partial y}{ \partial x}\bigg)^T  \nabla_w \widehat \phi  (t, y, w)\\
	& = W  \cdot \nabla_y  	\widehat \phi (t, y, w) - X \cdot \nabla_w \widehat \phi  (t, y, w),
\end{align*}
where
\begin{align}
\label{eqD.1}
	& W = \frac{w}{ \sqrt{1+|\frac{\partial x}{ \partial y} w|^2} }, \, \, \,  X =   \bigg(\frac{\partial y}{ \partial x}\bigg) \bigg(\frac{\partial p}{\partial y}\bigg) W =    \bigg(\frac{\partial y}{ \partial x}\bigg) \bigg(\frac{\partial (\frac{\partial x}{ \partial y} w)  }{\partial y}\bigg) W.
\end{align}
Thus, by the above computation, we have
\begin{equation}
\begin{aligned}
			\label{eqC.2}
&	\int_{\Sigma^T} (Y \phi) f \, dz \\
	&=\int_{ \bH^T_{-} }  (\partial_t \widehat \phi +  W \cdot \nabla_y \widehat \phi) \widehat f  \, \mathsf{J} \, dy dw dt  - \int_{ \bH^T_{-} }  (X \cdot \nabla_w  \widehat \phi) \widehat f   \mathsf{J}\, dy dw dt. 
\end{aligned}
\end{equation}
where $\bH^T_{-}$ is defined in \eqref{eq1.2.0}, and
$$
	\mathsf{J} =  \bigg(\text{det} \bigg(\frac{\partial x}{ \partial y}\bigg)\bigg)^2.
$$


The following lemma is a consequence of \eqref{eqC.2}.
\begin{lemma}
            \label{lemma D.1}
Let $\Psi$ be the local diffeomorphism given by \eqref{eq2.55} - \eqref{eq2.56},
  $r > 1$ be a number,
and $u \in L_r (\Sigma^T)$ be a function such that
\begin{itemize}
\item[--] $\nabla_p u \in L_{2, 1} (\Sigma^T)$,  $Y u \in L_2 (\Sigma^T)$,
\item[--]  $u (t, \cdot, p) = 0$ for $x \not \in \Omega_{ r } (x_0)$.
\end{itemize}
Then, one has
\begin{align}
    \label{eqD.1.1}
 \|Y  u\|_{ L_r (\Sigma^T ) } \le N \|(\partial_t + W \cdot \nabla_y)\widehat u\|_{ L_r ( \bH^T_{-}   ) }
+ N \|\nabla_p  u\|_{ L_{r, 1} ( \Sigma^T   ) },
\end{align}
where $N  =  N (\Omega, r)$.
\end{lemma}

\begin{proof}
The estimate follows from \eqref{eqC.2} and the fact that for any $y \in \psi (\Omega_{r_0} (x_0))$ and $w \in \bR^3$,
$$
    |W (y, w)| \le N (\Omega) \max\{|w|, 1\}.
$$

\subsection{Identity  $\ref{eq2.14}$}
                \label{section E.2}
Invoke the notation of Step 4 in the proof of Lemma  \ref{lemma 7.4} (see p. \pageref{eq2.53}).
Proceeding as in Section \ref{section E.1}, for any test function $\phi$, we have
\begin{align*}
    &(\nabla_w  \phi) (t, y, w (y, v)) = \big(\frac{\partial v}{\partial w}\big)^T \nabla_v \hathat \phi (t, y, v)\\
    &  (\nabla_y  \phi) (t, y, w (y, v)) = \nabla_y \hathat \phi (t, y, v) - \big(\frac{\partial w}{\partial y}\big)^T  \big(\frac{\partial v}{\partial w}\big)^T \nabla_v  \hathat \phi (t, y, v).
\end{align*}
By this computation,  we conclude
\begin{align}
    \label{eqC.3}
    \int_{\bR^7_T} (\mathcal{W} \cdot \nabla_y \phi) \,  \mathcal{U} \, dy dw dt
     = \int_{\bR^7_T} (v \cdot  \nabla_y \hathat \phi)\,  \mathcal{\hathat U}  \, dy dv dt- \int_{\bR^7_T} (\mathbb{G}  \cdot \nabla_v \hathat \phi) \,  \mathcal{\hathat U}  \, dy dv dt,
\end{align}
where
$$
    \mathbb{G}  = (\frac{\partial v}{\partial w}\big)\big(\frac{\partial w}{\partial y}\big) v.
$$
\end{proof}


\section{$S_r$ theory for the KFP equation on the whole space}
                            \label{Appendix F}
\begin{assumption}$(\gamma_{\star})$
                  \label{assumption D.1}
There exists $R_0 > 0$ such that for any $z_0 = (t_0, x_0, v_0)$ satisfying $t_0 < T$  and $r \in (0, R_0]$,
\begin{equation}
			\label{eqD.0}
     \text{osc}_{x, p} (a, Q_r (z_0)) \le \gamma_{  \star  },
\end{equation}
where
\begin{align}
	\label{D.1.1}
  & \text{osc}_{x, p} (a, Q_r (z_0))\\
&= r^{- 14 } \int_{t_0 - r^2}^{t_0} \int_{  D_r (z_0, t) \times D_r (z_0, t)}
   |a (t, x_1, p_1) - a (t, x_2, p_2)| \, dx_1 dp_1 dx_2 dp_2 \, dt\notag,
\end{align}
  and
  $$
    D_r (z_0, t) =  \{(x, p): |x  - x_0 - (t-t_0) p_0|^{1/3} < r, |p-p_0| < r      \}.
  $$
\end{assumption}


\begin{remark}
			\label{remark D.1}
Note that, if $a \in  L_{\infty} ((-\infty, T)) C^{\varkappa/3, \varkappa}_{x, v}  (\bR^6)$ for some $\varkappa \in (0, 1]$, then, for any $\gamma_{\star} \in (0, 1)$, Assumption \ref{assumption D.1} $(\gamma_{\star})$ holds
with $R_0  = ([a]^{-1}_{ L_{\infty} ((-\infty, T)) C^{\varkappa/3, \varkappa}_{x, v}  (\bR^6) } \gamma_{\star})^{1/\varkappa}$.
\end{remark}

The following theorem is a simplified version of Theorem 2.4 of \cite{DY_21a}.
\begin{theorem}
            \label{theorem D.1}
Let 
\begin{itemize}
\item[--] $r > 1$, $K > 0$, $\lambda \ge 0$, $-\infty<S < T<\infty$,  be numbers, 
\item[--] $a, b, c$ satisfy the assumptions  \eqref{eq1.1} and \eqref{eq1.2.2},
\item[--] $[a]_{ L_{\infty} ((0, T)) C^{\varkappa/3, \varkappa}_{x, p} (\bR^6) } \le L$ for some $\varkappa \in (0, 1]$ and $L > 0$. 
\end{itemize}
Then, for any
$f \in L_r ((S, T) \times \bR^{6})$,
the equation
$$
	   (\partial_t + p \cdot \nabla_x) u  - a^{i j} \partial_{p_i p_j} u +  b \cdot \nabla_p u + c u + \lambda u = f, \quad u (0, \cdot) = 0
$$
 has a unique solution
$u  \in S_r^{N} ((S, T) \times \bR^{6})$ (see \eqref{eq1.2.4.0}).
In addition,
\begin{align*}
	 & \|u\| + \|\nabla_p u\| + \|D^2_p u\| +  \|(\partial_t + p \cdot \nabla_x) u\|  \\
	&
	+  \|(-\Delta_x)^{1/3} u\|
	+\|\nabla_p (-\Delta_x)^{1/6} u\|   \le N \|f\|,
\end{align*}
where $\|\,\cdot\,\|=\|\,\cdot\,\|_{L_r ((S, T) \times \bR^{6})  }$ and $N = N (\delta, \varkappa, r, K,  T-S, L)$.
\end{theorem}

\begin{theorem}[Corollary 2.6 of \cite{DY_21a}]
                                                \label{theorem D.2}
Invoke the  assumptions  of Theorem \ref{theorem D.1}
and drop the H\"older continuity of $a$ assumption.
Then, there exist constants
$$
\kappa = \kappa (r) > 0, \quad
\beta = \beta (r) > 0,
\quad	\gamma_{  \star  }  =   \delta^{ \kappa  }  \widetilde \gamma_{  \star } (r)  > 0,
$$
such that if the condition \eqref{eqD.0} in Assumption \ref{assumption D.1} $(\gamma_{ \star })$ holds,
then for any $u \in S_{r}^{N} ((-\infty, T) \times \bR^6 )$  and $\lambda \ge 0$,
\begin{equation}
				\label{eqD.2.1}
\begin{aligned}
&	\|u\|_{S_{r}^{N} ((-\infty, T) \times \bR^6) }\\
 &  \le N \delta^{-\beta} ( \|(\partial_t + p \cdot \nabla_x) u  - a^{i j}\partial_{p_i p_j} u + b \cdot \nabla_p u + c u + \lambda u \|_{ L_{r} ((-\infty, T) \times \bR^6) }\\
&\quad + N R_0^{-2}\|u\|_{ L_{r} ((-\infty, T) \times \bR^6) }),
\end{aligned}
\end{equation}
where      $N  = N (r,  K)$, and $R_0 \in (0, 1)$ is the constant in Assumption \ref{assumption D.1} $(\gamma_{\star})$.
\end{theorem}


\begin{lemma}[see Lemma D.6 in \cite{DGY_21}]
			\label{lemma D.3}
Invoke the assumptions of Theorem \ref{theorem D.1} and let 
\begin{itemize}[--]
\item $T > 0$, $\lambda \ge 0$,   $1 <  q < r$ be numbers,  
\item $u \in S_q^{N} (\bR^7_T)$
be a function such that $u (0, \cdot) \equiv 0$, 
and
$$
	h: = (\partial_t + p \cdot \nabla_x) u  - a^{i j} \partial_{p_i p_j} u + b \cdot \nabla_p u + (c + \lambda) u  \in L_r  (\bR^7_T).
$$
\end{itemize}
Then, $u \in S_r^{N}  (\bR^7_T)$.
\end{lemma}

\begin{lemma}
			\label{lemma 4.1}
Let 
\begin{itemize}[--]
\item $\delta \in (0, 1)$, $\varkappa \in (0, 1]$, $K > 0$, $M > 0, \lambda \ge 0$ be numbers,
\item
$u \in L_2 ((0, T) \times \bR^3_x) W^1_2 (\bR^3_p)$,
\item $\delta |\xi|^2 < a^{i j} (z) \xi_i \xi_j < \delta^{-1} |\xi|^2$ for any $z \in \bR^7_T$ $\xi \in \bR^3$, 
\item \begin{equation}
            \label{eq4.1.0}
    \|a\|_{ L_{\infty} ((0, T))  C^{\varkappa/3, \varkappa}_{x, v} (\bR^6)  }, \|\nabla_p a\|_{ L_{\infty} (\bR^7_T)  } \le K \delta^{-1}
\end{equation}
for some $K > 1$,
\item $h \in L_2 (\bR^7_T)$, 
\item for any $\phi \in C^{1}_0 ([0, T] \times \bR^6)$ such that $\phi (0, \cdot), \phi (T, \cdot) \equiv 0$, one has
\begin{align*}
  \int (-(\partial_t \phi  + p \cdot \nabla_x \phi) u + (\nabla_p \phi)^T a \nabla_p u  + \lambda u \phi - h \phi) \, dz  = 0.
\end{align*}
\end{itemize}
Then,  the following assertions hold.

$(i)$ 
 One has  $u \in S_2^{N} (\bR^7_T)$ and,  there exists 
 $\beta  = \beta (\varkappa) > 0$ such that
\begin{equation}
            \label{eq4.1.3}
	\|u\|_{  S_2^{N}   (\bR^7_T) } \le N \delta^{-\beta} \big(\|h\|_{  L_2   (\bR^7_T) } +  \||u|+|\nabla_p u|\|_{  L_2   (\bR^7_T) }\big),
\end{equation}
where $N = N (\varkappa, K) > 0$.

$(ii)$ If, additionally, $h, \nabla_p u \in L_r (\bR^7_T)$, for some $r \in (2, \infty)$,
then,  $u \in S_r (\bR^7_T)$, and
\begin{equation}
            \label{eq4.1.4}
            \|u\|_{   S_r^{N} (\bR^7_T)  } \le N \delta^{-\beta} \big(\|h\|_{  L_r   (\bR^7_T) } + \||u|+|\nabla_p u|\|_{  L_r   (\bR^7_T) }\big),
\end{equation}
where $\beta = \beta (r, \varkappa) > 0$,  $N  = N (r, \varkappa, K) > 0$.
\end{lemma}

\begin{proof}
In this proof, $N$ is a constant independent of $\delta$.

$(i)$
\textit{Step 1: $u  \in  S_2^{N} (\bR^7_T)$.}
For $t < 0$ we set $u$ and $h$ to be $0$, and $a$ to be $\bm{1}_3$.
Then, for any smooth function $\phi$ with compact support in  $(-\infty, T) \times \bR^6$,   we have 
$$
	 \int \bigg(-(\partial_t \phi  + p \cdot \nabla_x \phi) u +  (\nabla_p \phi)^T a \nabla_p u - \lambda u \phi  - h \phi\bigg) \, dz  = 0.
$$
 This implies that 
$
 (\partial_t + p \cdot \nabla_x) u \in L_2 ((-\infty, T)\times \bR^3_x) W^{-1}_2 (\bR^3_p).
$

Next, by Theorem \ref{theorem D.1}, the equation 
\begin{equation}
			\label{eq4.1.1}
	\partial_t u_1 + p \cdot \nabla_x u_1 - a^{i j} \partial_{p_i p_j} u_1 + \lambda u_1  = h - \partial_{p_i} a^{i j}  \partial_{p_j} u
\end{equation}
has a unique solution $u_1 \in S_2^{N} ((-\infty, T)\times \bR^6)$ such that $u_1 1_{t \le 0} \equiv 0$.
Then,  for $U = u - u_1$, we have
$$
    U \in  L_2 ((-\infty, T)\times \bR^3_x) W^{1}_2 (\bR^3_p), \quad (\partial_t + p \cdot \nabla_x) U \in  L_2 ((-\infty, T)\times \bR^3_x) W^{-1}_2 (\bR^3_p),
$$
 and the identity
$$
	  (\partial_t  + p \cdot \nabla_x)  U  - \nabla_p \cdot (a \nabla_p U) + \lambda U = 0
$$
holds in $L_2 ((-\infty, T)\times \bR^3_x) W^{-1}_2 (\bR^3_p)$.
Then, by `testing' the above identity with $u$ in the sense of the duality pairing between $L_2 ((-\infty, T)\times \bR^3) W^{k}_2 (\bR^3)$, $k =  \pm 1$, and integrating by parts in  $p$ ,  we get for a.e.  $s \in (-\infty, T)$,
$$
		\int_{  \bR^6}  U^2 (s, x, p) \, dx dp  +   \int_{  (-\infty, s) \times \bR^6 }  (\delta |\nabla_p U|^2 + \lambda U^2) \, dz   = 0.
$$
We conclude $U \equiv 0$, and hence, $u \in S_2^{N} ((-\infty, T) \times \bR^6)$.

\textit{Step 2: $S_2$ estimate.}
By the assumption \eqref{eq4.1.0} and Remark \ref{remark D.1}, 
for any $\gamma_{\star} \in (0, 1)$, the condition \eqref{eqD.0} in Assumption  \ref{assumption D.1} $(\gamma_{\star})$ holds with
\begin{equation}
			\label{eq4.8.2}
	R_0  =  K^{-1/\varkappa} \delta^{1/\varkappa} \gamma^{1/\varkappa}_{\star}.
\end{equation}
Furthermore, let
$$
	\beta   > 0, \quad \kappa   > 0, \quad \gamma_{   \star  } = \delta^{\kappa}   \widetilde \gamma_{\star} (\varkappa)  > 0
$$
be the numbers in Theorem \ref{theorem D.2} with $r = 2$.
Then, by the estimate \eqref{eqD.2.1}  applied to Eq. \eqref{eq4.1.1} and the assumption \eqref{eq4.1.0}, we conclude that there exists $N = N (\varkappa)$ such that
\begin{align*}
	&\|u\|_{  S_2^{N}   (\bR^7_T) }  =  \|u_1\|_{  S_2^{N}   (\bR^7_T) } \\
&	\le  N \delta^{-\beta} \||h| + |\partial_{p_i} a^{i j}  \partial_{p_j} u|\|_{  L_2 (\bR^7_T)  } + N  K^{-2/\varkappa} \delta^{- 2\kappa/\varkappa} \|u\|_{  L_2 (\bR^7_T)  }\\
&	\le   N \delta^{-\beta} \|h\|_{  L_2 (\bR^7_T)  }  + N K   \delta^{-\beta}  \|\nabla_p u\|_{  L_2 (\bR^7_T)  }   + N  \delta^{- 2\kappa/\varkappa} \|u\|_{  L_2 (\bR^7_T)  },
\end{align*}
where we used the fact that $K > 1$.

$(ii)$ Applying Lemma \ref{lemma D.3} to Eq. \eqref{eq4.1.1}, we conclude that  $u \in S_r^{N} (\bR^7_T)$.  The desired estimate \eqref{eq4.1.4}  is obtained via  Theorem \ref{theorem D.2} in the same way as \eqref{eq4.1.3}.
\end{proof}

\begin{lemma}[embedding for the steady $S_p^{N} (\bR^{2d})$ space]
        \label{lemma 14.C.7}
    Let $d \ge 1$,  $p \in (1, \infty)$, and $u \in S_p^{N} (\bR^{2d})$ (see \eqref{eq1.2.4.0}).
    Then, the following assertions hold.

    $(i)$ For any $p \in (1, 2d)$ and $q > 1$ satisfying
    $$
        \frac{1}{q} > \frac{1}{p} -  \frac{1}{2d},
    $$
    we have
    \begin{align}
        \label{eq14.C.7.1}
        \|u\|_{L_q (\bR^{2d})} \lesssim_{d, p, q} \|u\|_{ S_p^{N} (\bR^{2d}) }.
        \end{align}

    $(ii)$   For any $p \in (1, 4d)$ and $q > 1$ satisfying
$$
        \frac{1}{q} >   \frac{1}{p} -  \frac{1}{4d},
$$
    one has
        \begin{align}
             \label{eq14.C.7.2}
        \|\nabla_v u\|_{L_q (\bR^{2d})} \lesssim_{d, p, q} \|u\|_{ S_p^{N} (\bR^{2d}) }.
        \end{align}

    $(iii)$ For $p > 2d$,
     \begin{align}
             \label{eq14.C.7.3}
        \|u\|_{ L_{\infty} (\bR^{2d}) } \lesssim_{d, p} \|u\|_{ S_p^{N} (\bR^{2d}) }.
            \end{align}
            Furthermore, if $p > 4d$ and $\alpha \in \big(0,  1 -\frac{4d}{p}\big)$,
            \begin{align}
                   \label{eq14.C.7.4}
                   \|[u, \nabla_v u]\|_{ C^{\alpha/3, \alpha}_{x, v} (\bR^{2d}) } \lesssim_{d, p, \alpha} \|u\|_{ S_p^{N} (\bR^{2d}) }.
            \end{align}
\end{lemma}

\begin{proof}
$(i) - (ii)$ We denote
$$
    f = v \cdot \nabla_x u - \Delta_v u + u.
$$
Let $\Gamma (t, x, v; t',  x', v')$ be the fundamental solution of the operator $(\partial_t + v \cdot \nabla_x) - \Delta_v$.
It is well known that (see, for example, \cite{PR_98})
$$
    \Gamma (t, x, v; t', x', v') = (t - t')^{- 2 d} \mathfrak{p} \bigg(\frac{x - x' - (t-t')v'}{(t-t')^{3/2}}, \frac{v-v'}{(t-t')^{1/2}}\bigg),
$$
where $\mathfrak{p}$ is a certain Gaussian function.
Then, we have
\begin{align*}
      &  u (x, v)  = \int_0^{\infty} \int_{\bR^{2d}} e^{-t} \, \Gamma (t, x, v; 0, x', v') f (x', v') \, dx' dv'dt \\
      &= \int_0^{\infty} \int_{\bR^{2d}} t^{-2d} e^{-t} \,\mathfrak{p} \bigg(\frac{x - x' }{t^{3/2}}, \frac{v-v'}{t^{1/2}}\bigg)  \tilde f (x', v') \, dx' dv' dt,
\end{align*}
where $\tilde f (x, v) = f (x + t v, v)$.

Next, let $r$ be the number defined by the relation
$$
    \frac{1}{r} + \frac{1}{p}= 1 + \frac{1}{q}.
$$
Then, by Minkowski and Young's inequalities,
\begin{align*}
     &   \|u\|_{ L_q (\bR^{2d})} \le \|f\|_{ L_p (\bR^{2d})} \int_0^{\infty} e^{-t} t^{-2d} \|\mathfrak{p} \big(\frac{\cdot }{t^{3/2}}, \frac{\cdot}{t^{1/2}}\big)\|_{L_r (\bR^{2d}) } \, dt.
     \end{align*}
Since $1 - 1/r < 1/(2d)$,  the second factor on the right-hand side is bounded by
$$
    N (d) \int_0^{\infty} e^{-t}  t^{-2d (1 - 1/r)} \, dt < \infty,
$$
 and hence, the estimate \eqref{eq14.C.7.1} is valid.  The second assertion \eqref{eq14.C.7.2} is proved in the same way.

$(iii)$ A simple application of  H\"older's inequality gives
\begin{align}
    \label{eq14.C.7.3.1}
   |u (x, v)| \lesssim_{d} \|f\|_{ L_p (\bR^{2d})}  \int_0^{\infty} e^{-t}  t^{-2d/p } \, dt
   \lesssim_{d, p} \|f\|_{ L_p (\bR^{2d})},
\end{align}
and hence,  \eqref{eq14.C.7.3} is true.
The proof of \eqref{eq14.C.7.4} follows from the identity
\begin{align*}
      &  \nabla_v u (x, v)  = \int_0^{\infty} \int_{\bR^{2d}} e^{-t} t^{-2d-1/2} \, (\nabla_v \mathfrak{p}) \bigg(\frac{x - x' - t v' }{t^{3/2}}, \frac{v-v'}{t^{1/2}}\bigg)  f (x', v') \, dx' dv' dt
\end{align*}
and the argument in  \eqref{eq14.C.7.3.1}. We omit the technical details.
\end{proof}

\section{Proof of Proposition \ref{proposition 7.2}}
    \label{section 7.2}

\begin{lemma}
            \label{lemma 6.8}
We invoke the assumptions of Proposition  \ref{proposition 27.4}   and let  $f$ be the finite energy solution to  \eqref{eq1.0}-\eqref{eq1.0.0}. We assume, additionally, that for some $0 \le  \theta_1 \le \theta$,
\begin{align}
\label{eq6.8.6}
&\partial_t [g,  \nabla_p g,   b, \nabla_p \cdot b,  c] \in L_{\infty} (\Sigma^T),\\ 
\label{eq6.8.7}
&\partial_t \eta \in   L_2 ((0, T) \times \Omega) W^{-1}_{2, \theta_1  } (\bR^3), 
\end{align}
and, for 
\begin{equation}
                \label{eq6.8.0}
\begin{aligned}
   & f_1 (x, p):= - v (p) \cdot \nabla_x f_0 (x, p) + \nabla_p \cdot (\sigma_{f_0} (x, p) \nabla_p f_0 (x, p))\\
    & - b (0, x, p) \cdot \nabla_p f_0 (x, p) - c (0, x, p) f_0 (x, p) + \eta (0, x, p)
\end{aligned}
\end{equation}
(understood in the sense of distributions)
one has 
$$
    f_1 \in L_{2, \theta_1 }  (\Omega \times \bR^3),
$$ 
where $\sigma_{f_0}$ is given by \eqref{eq6.0} with $g$ replaced with $f_0$.
In addition, we assume that $f_0 \in L_2 (\Omega)  W^1_{2} (\bR^3)$ is a finite energy  solution to the steady equation  \eqref{eq6.8.0} with SRBC (see Definition \ref{definition 27.1}), where $f_1$ is viewed as the r.h.s. 
Then,  
\begin{align*}
    \partial_t f \in  C ([0, T])  L_{2, \theta_1 } (\Omega \times \bR^3) \cap L_2 ((0, T) \times \Omega) W^1_{2, \theta_1 } (\bR^3),
 \end{align*}   
 and, furthermore,   $u = \partial_t f$  is a finite energy  solution to
 \begin{align}
\label{eq6.8.5}
&Y u - \nabla_p \cdot (\sigma_g \nabla_p u) + b \cdot \nabla_p u + c u = \eta_1, \,  \, z \in \Sigma^T, \\
& u (t, x, p) = u (t, x, R_x p), (t, x, p) \in \Sigma^T_{-}, \quad u (0, x, p) = f_1 (x, p), (x, p) \in \Omega \times \bR^3,\notag
\end{align}
where 
\begin{align}
\label{eq6.8.8}
    \eta_1  = \partial_t \eta - \big(-\nabla_p \cdot ((\partial_t \sigma_g) \nabla_p f) + (\partial_t b) \cdot \nabla_p f + (\partial_t c) f\big). 
\end{align}
\end{lemma}

\begin{proof}[Proof of Lemma \ref{lemma 6.8}]
For the sake of clarity, we consider the case when $\theta, \theta_1 = 0$. The argument in the remaining case is the same as the one presented here.
 Let us first consider the equation \eqref{eq6.8.5}.
By the definition of $\eta_1$ (see \eqref{eq6.8.8}), the assumptions of the present lemma, and the fact that $f \in  L_2 ((0, T) \times \Omega) W^{1}_{2} (\bR^3)$, we conclude
$$
    \eta_1 \in  L_2 ((0, T) \times \Omega) W^{-1}_{2} (\bR^3).
$$
Then, by Proposition  \ref{proposition 27.4},   the problem \eqref{eq6.8.5}  has a unique finite energy solution (see Definition \ref{definition 27.1}).
Furthermore, we denote
$$
    \widetilde f (t, x, p) = \int_0^t u (s, x, p) \, ds + f_0 (x, p).
$$
To prove the lemma, it suffices to show that $\widetilde f \equiv f$. 


Next, by using  a simple identity
$$
    (\xi_1 \xi_2) (t) = (\xi_1 \xi_2) (0) + \int_0^t [\xi_1' (s) \xi_2 (s) + \xi_1 (s) \xi_2' (s)]  \, ds
$$
with $\xi_1 = \sigma_g, b, c$ and $\xi_2 = \tilde f, \nabla_p \tilde f$, 
and \eqref{eq6.8.0} - \eqref{eq6.8.5},  we formally conclude that $\tilde f$ is a finite energy solution to the equation 
\begin{align*}
&Y \widetilde f (z) - \nabla_p \cdot (\sigma_g (z) \nabla_p \widetilde f (z))  +  b (z) \cdot  \nabla_{p} \widetilde f (z) 
+  c (z) \widetilde f (z) -  \eta (z) \\
& =    - \eta (0, x, p) + v (p) \cdot  \nabla_x f_0 (x, p) -  \nabla_p \cdot (\sigma_{f_0} (x, p) \nabla_p  f_0 (x, p)) \\
& + b (0, x, p) \cdot \nabla_p f_0 (x, p)  + c (0, x, p)  f_0 (x, p) + u (0, x, p)  \\
&+  \int_0^t \bigg( -\nabla_p \cdot  \big((\partial_t \sigma_g (s, x, p)) \nabla_p (\widetilde f - f) (s, x, p)\big)\\
&+ (\partial_t b (s, x, p)) \cdot \nabla_p (\widetilde f - f) (s, x, p)  + (\partial_t  c (s, x, p))  (\widetilde f - f) (s, x, p) \bigg) \, ds,
\end{align*}
with SRBC and the initial data $\tilde f (0, \cdot) \equiv f_0 (\cdot)$. 
We note that the sum of the non-integral terms on the right-hand side of the above identity equals $0$ due to \eqref{eq6.8.0}.
Hence, the function $w = \widetilde f - f$ is a finite energy  solution to
\begin{align}
    \label{eq6.8.10}
&Y  w (z) - \nabla_p \cdot (\sigma_g (z) \nabla_p w (z))  +  b (z) \cdot  \nabla_{p} w (z) 
+  c (z) w (z)\\
&=\int_0^t \bigg( -\nabla_p \cdot  \big((\partial_t \sigma_g (s, x, p)) \nabla_p w (s, x, p)\big) \notag \\
&+ (\partial_t b (s, x, p)) \cdot \nabla_p w (s, x, p)  + (\partial_t  c (s, x, p))  w (s, x, p) \bigg) \, ds \notag
\end{align}
with the SRBC and the initial data $w (0, \cdot) \equiv 0$.
To make the above argument rigorous, one needs to work with the weak formulations of Eqs. \eqref{eq27.2} and \eqref{eq6.8.5}, and use the fact that $f_0$ is a finite energy  solution to \eqref{eq6.8.0} with the SRBC.  

Finally, by applying an `energy' type identity \eqref{eq27.4.12} to Eq. \eqref{eq6.8.10}, using  
integration by parts in $p$, and the Cauchy-Schwarz inequality, we get
\begin{align*}
\|w (t, \cdot)\|^2_{ L_2 (\Omega \times \bR^3)} + \|\nabla_p w \|^2_{ L_2 (\Sigma^t) } 
\le N \,  t  (\|w \|^2_{ L_2 (\Sigma^t) } + \|\nabla_p w \|^2_{ L_2 (\Sigma^t) }), t \in [0, T],
\end{align*}
where $N$ is independent of $t$.
Taking $t \le (2N)^{-1}$ and using the Gronwall's inequality, we conclude that $w =  0$ on $\Sigma^{T_1}$ where $T_1 = \min\{(2N)^{-1}, T\}$. Similarly, we show that $w = 0$ on $\Sigma^t$ for $t \in [T_1,  \min\{T_1+(2N)^{-1}, T\}]$ and so on.  Thus, $f \equiv \widetilde f$. 
\end{proof}

\begin{proof}[Proof of Proposition \ref{proposition 7.2}]
The uniqueness follows from the  estimate \eqref{eq7.60} with vanishing `initial data' $f_{0, k}, \bE_{0, k}, \bB_{0, k}$.

To show the existence, we  consider the iteration scheme $[h_{(n)}, \bE_{(n)}, \bB_{(n)}], n \ge 0,$ such that $[h_{(0)}, \bE_{(0)}, \bB_{(0)}]  = [f_{0, 0}, \bE_{0, 0}, \bB_{0, 0}]$, and given $[h_1, \bE_1, \bB_1] = [h_{(n)}, \bE_{(n)}, \bB_{(n)}]$, 
the next iteration   $[h_2, \bE_2, \bB_2] = [h_{(n+1)}, \bE_{(n+1)}, \bB_{(n+1)}]$ is defined as the  strong solution to the system
\begin{align}
    \label{eqF.1}
	&Y h_2 + \bm{\xi} (\bE_g + v (p) \times \bB_g) \cdot \nabla_p h_2 -  \frac{\bm{\xi}}{2}  (v (p) \cdot \bE_g) h_2  -  A  h_2  \\
	&=   \bm{\xi}_1 (v (p) \cdot  \bE_1)  J^{1/2}  +  K h_1  + \Gamma (h_2, g) \notag\\
    \label{eqF.2}
 & h_2 (t, x, p) = h_2 (t, x, R_x p), \, \, z \in \Sigma^T_{-}, \quad  h_2 (0, \cdot) \equiv f_{0, 0},\\
    \label{eqF.3}
&	\partial_t \bE_2 - \nabla_x \times \bB_2 = - \int v (p) J^{1/2} (p) h_1 (p) \cdot \bm{\xi} \, dp,\\
    \label{eqF.4}
 &
	\partial_t \bB_2  + \nabla_x \times \bE_2 = 0, \\
    \label{eqF.5}
  &
	\nabla_x \cdot \bE_2 = \int J^{1/2}  h_1 (p) \cdot \bm{\xi} \, dp, \quad \nabla_x \cdot \bB_2 = 0,\\ 
    \label{eqF.6}
& (\bE_2 \times n_x)_{|\partial \Omega} = 0, \quad (\mathbf{B}_2 \cdot n_x)_{|\partial \Omega} = 0,\\
    \label{eqF.7}
    & \bE_2 (0, \cdot) \equiv \bE_{0, 0} (\cdot),  \quad  \bB_2 (0, \cdot) \equiv \bB_{0, 0} (\cdot).
 \end{align}
 We assume that $[h_1, \bE_1, \bB_1]$ satisfies 
    \begin{align}
    \label{eqF.8}
& \partial_t^k h_1 \in C ([0, T]) L_2 (\Omega \times \bR^3) \cap L_2 ((0, T)) W^1_2 (\Omega \times \bR^3), k \le m, \\
    \label{eqF.9}
& \partial_t^k [\bE_1, \bB_1] \in C ([0, T]) L_2 (\Omega), k \le m, \\
    \label{eqF.10}
&
    \partial_t^k h_1 (0, \cdot) = f_{0, k} (\cdot), \, (\text{see} \, \eqref{eq3.3.2}) \,  k \le m, \\
    \label{eqF.11}
    &\partial_t^k  [\bE_1, \bB_1] (0, \cdot) = [\bE_{0, k} \bB_{0, k}] (\cdot) \, \,  (\text{see} \, \eqref{eq3.3.3} - \eqref{eq3.3.4}), k \le m,\\
    \label{eqF.12}
& \partial_t \rho_k + \nabla_x \cdot \mathbf{j}_k = 0 \, \text{(in the sense of  distributions)}, k \le m, \\
& 
\text{where} \,  \rho_k (t, x) = \int_{\bR^3} J^{1/2} (p) \partial_t^k h_1 (t, x, p) \cdot \bm{\xi} \, dp,  \notag \\
&\mathbf{j}_k  (t, x) =  \int_{\bR^3} J^{1/2} (p) v (p)  \partial_t^k h_1 (t, x, p) \cdot \bm{\xi} \, dp, \notag \\
    \label{eqF.13}
& y^{(\lambda)}  (h_1, \bE_1, \bB_1) \le N_1, \quad \sum_{k=0}^{m-1} \|e^{-\lambda \tau} \partial_t^k [\bE_1, \bB_1]\|^2_{ L_{\infty} ((0, T)) W^1_2 (\Omega) } \\
& +  \sum_{k=0}^{m-8} \|e^{-\lambda \tau} \partial_t^k [\bE_1, \bB_1]\|^2_{ L_{\infty} ((0, T) \times \Omega) } \le N_1 N_2, \notag
\end{align}
where $N_1, \lambda > 1$ are constants depending only on $r_1, \ldots, r_4, \Omega, \theta, f_{0, k}, \bE_{0, k}, \bB_{0, k}, k \le m$,  $N_2 = N_2 (\Omega) > 1$,  and 
\begin{align}
    \label{eqF.14}
   & y^{(\lambda)}  (h_1, \bE_1, \bB_1) =    \sup_{\tau \le T} \cI^{(\lambda)} (h_1, \bE_1, \bB_1, \tau) + \int_0^{T} \cD^{(\lambda)}   (h_1, \bE_1, \bB_1, \tau) \, d\tau, \\
    \label{eqF.15}
   & \cD^{(\lambda)} (h_1, \bE_1, \bB_1, \tau)  =   \sum_{k=0}^m  \bigg(\| e^{-\lambda \tau} (\sqrt{\lambda}  |\partial_t^k f (\tau, \cdot)| + |\nabla_p \partial_t^k f (\tau, \cdot)|)\|^2_{  L_2 (\Omega \times \bR^3) }\\ 
 &  +  \lambda  \|e^{-\lambda \tau} \partial_t^k [\bE_1, \bB_1] (\tau, \cdot)\|^2_{ L_2 (\Omega) }\bigg) \notag\\
   &  + \sum_{k=0}^{ m - 4 }  \| e^{-\lambda \tau} (\sqrt{\lambda}  |\partial_t^k f (\tau, \cdot)| + |\nabla_p \partial_t^k f (\tau, \cdot)|)\|^2_{  L_{2, \theta/2^{k}  } (\Omega \times \bR^3) }, \notag\\
    \label{eqF.16}
   & \cE^{(\lambda)} (h_1, \bE_1, \bB_1, \tau)  =  \sum_{k=0}^m    \bigg(\|e^{-\lambda \tau} \partial_t^k  f (\tau, \cdot) \|^2_{ L_2  (\Omega \times \bR^3) } 
   + \|e^{-\lambda \tau}\partial_t^k [\bE_1, \bB_1] (\tau, \cdot)\|^2_{ L_2 (\Omega) }\bigg) \notag  \\
   & +  \sum_{k=0}^{ m - 4 } \|e^{-\lambda \tau} \partial_t^k  f (\tau, \cdot)\|^2_{ L_{2, \theta/2^{k} } (\Omega \times \bR^3) },\\
    \label{eqF.17}
  & \cI^{(\lambda)} (h_1, \bE_1, \bB_1, \tau)  = \cE^{(\lambda)} (h_1, \bE_1, \bB_1, \tau) \\
  & +  \sum_{k=0}^{m-8}   \bigg(\|e^{-\lambda \tau} \partial_t^k  h_1 (\tau, \cdot) \|^2_{    L_{\infty} (\Omega) W^1_{\theta/2^{k+9}} (\bR^3)       }
     + \sum_{s \in \{2, r_4\} } \|e^{-\lambda \tau} D_p^2 \partial_t^k  h_1\|^2_{ L_s (\Sigma^{\tau}) } \bigg). \notag
\end{align}

We will show that the following assertions are true.
\begin{itemize}
     \item[i)] 
    $\partial_t^k [\bE_2, \bB_2], k \le m,$ is a weak solution to  Maxwell's equations \eqref{eqF.3} - \eqref{eqF.4} formally differentiated $k$ times with respect to $t$  with the perfect conductor BC and $\partial_t^k [\bE_2, \bB_2] (0, \cdot) = [\bE_{0, k}, \bB_{0, k}], k \le m$. For $k \le m-1$, the same pair is a strong solution. In addition, the identities in \eqref{eqF.5} formally differentiated $k$ times in $t$ are  valid.
    
    \item[ii)] $\partial_t^k h_2, k \le m,$ is a finite energy solution to the equation \eqref{eqF.1} differentiated formally $k$ times with respect to $t$ with 
    the initial conditions $\partial_t^k h_2 (0, \cdot) = f_{0, k} (\cdot)$ and with the SRBC.
        \item[iii)] the assumptions \eqref{eqF.8} - \eqref{eqF.12} hold with  $[h_2, \bE_2, \bB_2]$  in place of  $[h_1, \bE_1, \bB_1]$,
     \item[iv)]
     \begin{align}
        \label{eqF.18}
       &  y^{(\lambda)}  (h_2, \bE_2, \bB_2) \le N_1, \quad \sum_{k=0}^{m-1} \|e^{-\lambda \tau} \partial_t^k [\bE_2, \bB_2]\|^2_{ L_{\infty} ((0, T)) W^1_2 (\Omega) } \\
         &+ \sum_{k=0}^{m-8} \|e^{-\lambda \tau} \partial_t^k [\bE_2, \bB_2]\|^2_{ L_{\infty} ((0, T) \times \Omega) } \le N_1 N_2. \notag
     \end{align}
\end{itemize}

The  weight  $e^{-\lambda \tau}$ enables us to close the $L_{\infty, \theta/2^{k+9}} (\Sigma^T)$ estimate via the unsteady $S_p$ a priori estimates (see \eqref{eq2}, \eqref{eq2.3.1} -\eqref{eq2.3.1.5} in Propositions \ref{proposition 2.1}  and \ref{proposition 2.3}) by leveraging the  $L_{\infty}$ and $L_2$ control in the estimates of the `free' terms  $K (\partial_t^k h_1)$ and  $\bm{\xi}_1 (v (p) \cdot  \partial_t^k  \bE_1)  J^{1/2}$. See \eqref{eqF.23} and the paragraph below.
Furthermore,  the control of the last term in \eqref{eqF.17} is needed to estimate the  $L_{r_i}$-norm   of the free term $K (\partial_t^k h_1)$, which appears in the unsteady $S_p$ estimate.

If the assertions $i) - (iv)$ are true, then, by repeating a limiting argument in the proof of Theorem \ref{theorem 5.1} in Section \ref{section 8}, we conclude that there exist $[f, \bE_f, \bB_f]$ such that the desired assertions $a)- e)$ in Proposition \ref{proposition 7.2} are valid (see p. \pageref{eq7.90}). 

\textbf{Proof of $(i) - (iv)$.}
We will prove assertions in the following order: $(i)$, $(iii)$, $(iv)$, and $(ii)$.

 $(i)$ We use the standard existence/uniqueness results for weak/strong solutions to Maxwell's equations with the perfect conductor boundary conditions (see, for example, Chapter $(vii)$, Section 4 in \cite{DL_76}).
   In particular, 
   the differentiated in $t$ equations in \eqref{eqF.5} are satisfied 
   due to the continuity equations for $\partial_t^k h_1, k \le m,$ (see \eqref{eqF.12}) and the compatibility conditions  
   \eqref{eq3.3.8} - \eqref{eq3.10.1}  on the initial data $[\bE_{0, k}, \bB_{0, k}]$ combined with the fact that $\partial_t^k h_1 (0, \cdot) \equiv f_{0, k}$ (see \eqref{eqF.10}).
   Thus, the assertion $(i)$ is valid.

$(iv)$ In this argument, $N  = N (r_1, \ldots, r_4, \theta, \Omega, m, \alpha)$. First, we prove the estimates \eqref{eqF.18}, assuming that $(ii) - (iii)$ are true. 
We modify the proof of the estimate \eqref{eq5.1.4} given in Section \ref{section 7.1}.

\textit{$L_{\infty}$ estimate of $\partial_t^k [\bE_2, \bB_2], k \le m-8$.}
We establish the second estimate in \eqref{eqF.18} and specify the constant $N_2$. In this argument, $N_2  = N_2 (\Omega)$ is a constant which might change from line to line.
By applying $W^1_2$ div-curl estimate in \eqref{eq3.0.0} to Maxwell's equations differentiated $k$ times in $t$ and rewritten as div-curl systems (see \eqref{eq10.4} - \eqref{eq10.5}), and using the first bound in \eqref{eqF.13},   we have
\begin{align}
    \label{eqF.20}
  &  \sum_{k=0}^{m-1} \|e^{-\lambda \tau} \partial_t^k [\bE_2, \bB_2]\|^2_{ L_{\infty} ((0, T)) W^1_2 (\Omega) } \\
  & \le N_2  \sum_{k=0}^{m} \|e^{-\lambda \tau} \partial_t^k [\bE_2, \bB_2]\|^2_{ L_{\infty} ((0, T)) L_2 (\Omega) }
  + N_2 \sum_{k=0}^{m-1} \|e^{-\lambda \tau} \partial_t^k h_1\|^2_{ L_{\infty} ((0, T)) L_2 (\Omega \times \bR^3) } \le N_2 N_1 \notag
\end{align}
(recall $N_1 > 1$),  which gives the bound of the first term in the second estimate in \eqref{eqF.18}.

Next,  using the $W^1_6$ div-curl estimate, the Sobolev embedding $W^1_2 \subset L_6$, and \eqref{eqF.20}, and the $L_{\infty}^{t, x, p}$ bound of $\partial_t^k h_1, k \le m-8,$ in \eqref{eqF.13}, we conclude
\begin{align*}
  &  \sum_{k=0}^{m-8} \|e^{-\lambda \tau} \partial_t^k [\bE_2, \bB_2]\|^2_{ L_{\infty} ((0, T)) W^1_6 (\Omega) } \\
  & \le N_2  \sum_{k=0}^{m-  7 } \|e^{-\lambda \tau} \partial_t^k [\bE_2, \bB_2]\|^2_{ L_{\infty} ((0, T)) L_6 (\Omega) }
  + N_2 \sum_{k=0}^{m-8} \|e^{-\lambda \tau} \partial_t^k h_1\|^2_{ L_{\infty} (\Sigma^T) } \le N_2 N_1.
\end{align*}
Finally, thanks to the embedding $W^1_6 \subset L_{\infty}$, we obtain the desired estimate of the second term in the second estimate in \eqref{eqF.18}.

\textit{Total energy estimate.}
First, we derive  an estimate of the total instant energy and dissipation
\begin{align}
    \label{eqF.25}
  &  \sup_{\tau \le T} \cE^{(\lambda)} (h_2, \bE_2, \bB_2, \tau) + \int_0^T  \cD^{(\lambda)} (h_2, \bE_2, \bB_2, \tau) \, d\tau \\
  & \le N \sqrt{\varepsilon_0} y^{(\lambda)} (h_2, \bE_2, \bB_2) +  N \lambda^{-1}  N_1 N_2 
+  N \cE^{(\lambda)} (h_2, \bE_2, \bB_2, 0), \notag
\end{align}
where  $N_1$ and $N_2$ are the constants in \eqref{eqF.13}.
We follow the argument of Step 1 in Section \ref{section 7.1} (see p. \pageref{eq7.80}) by making the following minor modifications: 
\begin{itemize}
\item  we add the weight $e^{-2 \lambda t}$ to all the terms  therein, 
    \item we modify the integrals $I_1$ and $I_2$. 
\end{itemize}
In particular, in $I_1$ (see \eqref{eq7.11}), one needs to replace $L f$ with $-A h_2 - K h_1$. We then apply the estimate \eqref{eq2.4.1}  in Lemma  \ref{lemma 2.4}:
\begin{align}
    \label{eqF.21}
  &  \int_{\Sigma^{\tau}}   \langle - A (\partial^k_t h_2), p_0^{2 \theta_k} \partial^k_t h_2  \rangle e^{-2 \lambda t}   \, dz \\
  &  \ge \kappa \| e^{- \lambda t} \nabla_p \partial^k_t h_2\|^2_{ L_{2,  \theta_k     }   (\Sigma^{\tau}) }  - N \|e^{- \lambda t}  \partial^k_t h_2\|^2_{L_{2  } (\Sigma^{\tau})}. \notag
\end{align}
Furthermore, using the symmetry of the operator $K$ and the bound \eqref{eq2.4.2}, we get 
\begin{align}
    \label{eqF.21.1}
   & -\int_{\Sigma^{\tau}}  \big(K (\partial^k_t h_1)) (\partial_t^{k_1} h_2) e^{- 2\lambda t} \, dz\\
   & \ge -  (\kappa/2) \| e^{- \lambda t}  \partial^k_t h_2\|^2_{ L_{2}  ((0, \tau) \times \Omega) W^1_2 (\bR^3)  }  
     - N \|e^{- \lambda t}  \partial^k_t h_1\|^2_{L_{2  } (\Sigma^{\tau})}. \notag
\end{align}
 We note that by the assumption \eqref{eqF.13} and the presence of the factor $\lambda$ in the definition of $\cD^{(\lambda)} (h_1, \bE_1, \bB_1, \tau)$ (see \eqref{eqF.15}),  the last term on the r.h.s. of \eqref{eqF.21.1} can be replaced with $- N  N_1 \lambda^{-1}$.
Furthermore,  in the term $I_2$ in \eqref{eq7.11}, we replace $\bE_f$ with $\bE_1$ and proceed  as in  \eqref{eqF.21}.

\textit{Unsteady $S_p$ estimate.} 
Here, we estimate the remaining term in \eqref{eqF.17}, which is the sum of squares of weighted $L_{\infty}^{t, x} W^1_{\infty}$ norms. This will be done via the \textit{unsteady} $S_p$ estimate.
 We first note that $u = e^{-\lambda t} \partial_t^k h_2, k \le m-8,$  formally satisfies the identities
\begin{align}
    \label{eqF.22}
    &Y  u  -  \nabla_p \cdot (\sigma_{g^{+}  + g^{-} }   \nabla_p u) 
    +  \bm{\xi} (\bE_g + v (p) \times \bB_g - a_g) \cdot \nabla_p u  \\
    & +(\lambda + C_g - \frac{\bm{\xi}_1}{2}  v (p) \cdot \bE_g) u \notag \\
    &
        =  e^{-\lambda t}  \bigg(K (\partial_t^k   h_1) + \bm{\xi} (v (p) \cdot \partial_t^k \bE_1) J^{1/2}  + 1_{k > 0} \sum_{j=1}^3 \sum_{k_1+k_2 = k, k_1 \ge 1 }\eta_{k_1, k_2}^j\bigg),  \notag\\
    & u (t, x, p) = u (t, x, R_x p), z \in \Sigma^T_{-},  \quad u (0, \cdot)= f_{0, k} (\cdot) \, \text{(see \eqref{eq3.3.2})},  \notag\\
     & \eta^1_{k_1, k_2} = -\bm{\xi} \partial_t^{k_1}(\bE_g + v (p) \times  \bB_g ) \cdot \nabla_p (\partial_t^{k_2} h_2) + \frac{\bm{\xi}}{2}  (v (p) \cdot \partial_t^{k_1}  \bE_g) \partial_t^{k_2} h_2, \notag\\
     & \eta^2_{k_1, k_2} = \big(\partial_{p_i } \partial_t^{k_1}  \sigma_{g^{+}  + g^{-} }^{i j}
     - \partial_t^{k_1} a_g^i\big) (\partial_{ p_i} \partial_t^{k_2}  h_2) + (\partial_t^{k_1} C_g) \partial_t^{k_2} h_2,  \notag\\
          & \eta^3_{k_1, k_2} =  (\partial_t^{k_1}   \sigma_{g^{+}  + g^{-} }^{i j}) (\partial_{p_i p_j} \partial_t^{k_2}  h_2). \notag
 \end{align}
For $i \le 4$ and $k \le m-8$, we apply the unsteady $S_p$ estimates in \eqref{eq2}, \eqref{eq2.3.1}, and \eqref{eq2.3.1.5} with $\theta/2^{k+2i}$ in place of $\theta$ and $\kappa = \frac{1}{2}$, and we get
\begin{align}
    \label{eqF.24}
    &\|e^{-\lambda t}\partial_t^k h_2\|^2_{ S_{r_i, \theta/2^{k+2i+1}} (\Sigma^T)  } + 1_{i=4} \|e^{-\lambda t} \partial_t^k h_2\|^2_{ L_{\infty} ((0, T) \times \Omega) W^1_{\infty, \theta/2^{k+   2 i  + 1    } } (\bR^3)  } \\
    & \le  N \sum_{s \in \{2, r_i\} } \bigg(\|e^{-\lambda t}(\text{r.h.s of  }  \eqref{eqF.22})\|^2_{ L_{s, \theta/2^{k+2i}} (\Sigma^T)  } + \lambda^2 \|f_{0, k}\|^2_{ S_{s, \theta/2^{k+2i}} (\Omega \times \bR^3)   }\bigg) \notag\\
  & + N \|e^{-\lambda t}\partial_t^k h_2\|^2_{ L_{2,  \theta/2^{k+2i}  } (\Sigma^T)  }.  \notag
\end{align}
We follow the argument in the proof of \eqref{eq7.96} in Proposition \ref{proposition 7.3}  with minor modifications:
\begin{itemize}
    \item The loss of decay in the $p$ variable is different than that  in the $L_{\infty}^t S_p$ estimate in \eqref{eq7.96}   since the term $\partial_t^{k+1} f$ is on the l.h.s. in the present argument (cf. $b)$ on p. \pageref{b)}).
    \item The main difference is the estimate of the `free' terms $e^{-\lambda t} K \partial_t^k h_1$ and $v (p) \cdot \partial_t^k \bE_1 J^{1/2}$, as the rest of the terms on the r.h.s. of \eqref{eqF.22} are handled in the same way as in the proof of \eqref{eq7.96} (cf. \eqref{eq7.40} and \eqref{eq7.41}).
\end{itemize}
Let us consider the first two terms on the r.h.s. of \eqref{eqF.22}. 
By interpolating between $L_2$ and $L_{\infty}$, exploiting the presence of the factor $\lambda$ in front of the $L_2$ norm of $\bE_1$ in $\cD^{(\lambda)} (h_1, \bE_1, \bB_1, \tau)$, and using the $L_{\infty}^{t, x}$  bound of $\bE_1$ in \eqref{eqF.13}, we get
\begin{align}
    \label{eqF.23}
   & \|e^{-\lambda t} v (p) \cdot \partial_t^k \bE_1 J^{1/2}\|^2_{ L_{r_i} ((0, T) \times \Omega) }  \\
   & \le   \|e^{-\lambda t} \partial_t^k \bE_1\|^{4/r_i}_{L_2 ((0, T) \times \Omega)} \|e^{-\lambda t} \partial_t^k \bE_1\|^{2-4/r_i}_{L_{\infty} ((0, T) \times \Omega)} \le N_2^{ 1-2/r_i  } N_1 \lambda^{-2/r_i}. \notag
\end{align}

Next, by \eqref{eqG.1.2} Lemma \ref{lemma G.1} and interpolation inequality, we have (cf. \eqref{eq7.97})  
\begin{align*}
 &   \|e^{-\lambda t} K (\partial_t^k h_1)\|^2_{  L_{r_i, \theta/2^{k+2i}} (\Sigma^T)  }  \le N \|e^{-\lambda t} \partial_t^k h_1\|^2_{ L_{r_i} ((0, T) \times \Omega) W^1_{r_i} (\bR^3)       } \\
 & 
 \le \lambda^{-1/r_i} \|e^{-\lambda t} D^2_p \partial_t^k h_1\|^2_{ L_{r_i} (\Sigma^T)  }
 + N \lambda^{1/r_i}  \|e^{-\lambda t}  \partial_t^k h_1\|^2_{ L_{r_i} (\Sigma^T)  }.
  \notag
\end{align*}
Furthermore, since $k \le m-8$, 
by interpolating between $L_2$ and $L_{\infty}$ and using the bounds of $h_1$ in \eqref{eqF.13} (cf. \eqref{eqF.23}), the last term is bounded by
\begin{align*}
 N    N_1 \lambda^{1/r_i} \lambda^{-2/r_i}  = N N_1 \lambda^{-1/r_i}.
\end{align*}
By the above argument, H\"older's  inequality, and the fact that  the last term on the r.h.s. in \eqref{eqF.17} is bounded by $N_1$ (see \eqref{eqF.13}), we get
\begin{align}
    \label{eqF.28}
 &  \sum_{ s \in \{2, r_i\} }  \|e^{-\lambda t} K (\partial_t^k h_1)\|^2_{  L_{s, \theta/2^{k+2i} } (\Sigma^T) }    \\
 &  \le 2 \lambda^{-1/r_i} \sum_{ s \in \{2, r_i\} }    \|e^{-\lambda t}  D^2_p \partial_t^k h_1\|^2_{ L_{s }   (\Sigma^T)  } 
   +  N    N_1 \lambda^{-1/r_i} \le  N    N_1 \lambda^{-1/r_i}. \notag
\end{align}

Thus, combining \eqref{eqF.24} with the estimates of the `free terms' \eqref{eqF.23} - \eqref{eqF.28} and with the bounds of nonlinear terms (cf. \eqref{eq7.40} - \eqref{eq7.41}),    we obtain
\begin{align}
    \label{eqF.26}
   & 
   \sum_{k=0}^{m-8} \bigg(\|e^{-\lambda t} \partial_t^k h_2\|^2_{ L_{\infty} ((0, T) \times \Omega) W^1_{\infty, \theta/2^{k+9} } (\bR^3)  }
   + \sum_{i=1}^4 \|e^{-\lambda t}\partial_t^k h_2\|^2_{ S_{r_i, \theta/2^{k+2i+1}} (\Sigma^T)  } \bigg) \\
 &  \le N \varepsilon_0 y^{(\lambda)} (h_2, \bE_2, \bB_2)  
   +   N N_1 N_2 \lambda^{ -1/r_4 } +   N \lambda^2  \mathcal{S}_f (0),\notag
\end{align}
where
$$
    \mathcal{S}_f (0): = \sum_{k=0}^{m-8} \sum_{ s \in \{2, r_4\} } \|f_{0, k}\|^2_{ S_{s, \theta/2^{k+8} } (\Omega \times \bR^3) }.
$$

Finally, gathering \eqref{eqF.25} and \eqref{eqF.26} gives
\begin{align*}
& y^{(\lambda)} (h_2, \bE_2, \bB_2) \le N \sqrt{\varepsilon_0} y^{(\lambda)} (h_2, \bE_2, \bB_2) +   N  \lambda^{  -1/r_4  } N_1  \\
& +  N \cE_f (0) +  N \lambda^2 \mathcal{S}_f (0),
\end{align*}
where $\cE_f (0)$ is defined in \eqref{eq5.1.1}.
Choosing $\varepsilon_0 < (2N)^{-2}$ gives
$$
 y^{(\lambda)} (h_2, \bE_2, \bB_2) \le N (N_1 \lambda^{-1/r_4} + \cE_f (0) +  \lambda^2 \mathcal{S}_f (0)).
$$
Furthermore, choosing $\lambda > (4 N)^{r_4} + \lambda_0$ gives  $N_1 N \lambda^{-1/r_4} < N_1/4$.
Finally, choosing 
$
    N_1 > (4/3) N  \big(\cE_f (0) +  \lambda^2 \cS_f (0)\big),
$ we obtain
$$
     y^{(\lambda)} (h_2, \bE_2, \bB_2)  \le N_1,
$$
as desired.

$(ii)$ First, we note that by the estimates of the free terms \eqref{eqF.23} - \eqref{eqF.28}, the assumption on $f_{0, 0}$ in the statement of Theorem \ref{theorem 5.1},  and Propositions \ref{proposition 2.1} - \ref{proposition 2.3}, the problem \eqref{eqF.22} with $k = 0$ has a unique strong  solution $h_2$, and, in addition,
\begin{align}
     \label{eqF.29}
      &  h_2 \in C ([0, T]) L_{2, \theta} (\Omega \times \bR^3) \cap L_2 ((0, T) \times \Omega) W^1_{2, \theta} (\bR^3) \\
   & \cap S_{r_i, \theta/2^{2i+1}} (\Sigma^T) \cap L_{\infty} ((0, T) \times \Omega) W^{1}_{\infty, \theta/2^{k+9}} (\bR^3), i = 1, \ldots, 4. \notag
\end{align}
Next, we use an induction argument. 

\textbf{Claim 1.} We assume that for some $k_0 \in \{1, \ldots, m-8\}$,   and all $k \le k_0 -1$, one has
\begin{align}
  \label{eq7.2.11}
& \partial_t^k h_2 \in C ([0, T]) L_{2, \theta/2^k} (\Omega \times \bR^3) \cap L_2 ((0, T) \times \Omega) W^1_{2, \theta/2^k} (\bR^3),   \\
 \label{eq7.2.11.1}
& \partial_t^k h_2 \in S_{r_i, \theta/2^{k+2i+1}} (\Sigma^T)  \\
& \cap L_{\infty} ((0, T) \times \Omega) W^{1}_{\infty, \theta/2^{k+9}} (\bR^3),  \, i = 1, \ldots, 4,  \notag\\
\label{eq7.2.13} 
 &   u = e^{-\lambda t} \partial_t^k f \, \text{is a  strong solution to \eqref{eqF.22}. } 
\end{align}
 Then,  we claim that  \eqref{eq7.2.11}  - \eqref{eq7.2.13} hold for all $k \le k_0$.

\textbf{Claim 2.} Invoke the definition of $\theta_k$ in \eqref{eq7.98}.  We assume that for some $k_0 \in \{m - 7, \ldots, m\}$  and all $k \le k_0-1$, one has
\begin{align}
    \label{eqF.31}
 & \partial_t^k h_2 \in C ([0, T]) L_{2, \theta_k} (\Omega \times \bR^3) \cap L_2 ((0, T) \times \Omega) W^1_{2, \theta_k} (\bR^3), \\
    \label{eqF.32}
 &   u = e^{-\lambda t} \partial_t^k f \, \text{is a finite energy  solution to \eqref{eqF.22}}. 
\end{align}
Then, \eqref{eqF.31}  - \eqref{eqF.32} hold for all $k \le k_0$.  

\textbf{Proof of Claim 1.}
 To justify the differentiation with respect to $t$ and \eqref{eq7.2.11}    with $k_0$ in place of $k$, we use Lemma \ref{lemma 6.8} with $\partial_t^{k_0-1} h_2$ in place of $f$,  $f_0$ and $f_1$ replaced with $f_{0, k_0-1} \in L_{2, \theta/2^{k_0-1} } (\Omega \times \bR^3)$ and $f_{0, k_0} \in L_{2, \theta/2^{k_0}  } (\Omega \times \bR^3)$, respectively, and
\begin{align}
  &   b =  \pm (\bE_g + v (p) \times \bB_g) - a_g,  c = (C_g \mp \frac{1}{2}  v (p) \cdot \bE_g),  \notag \\
  \label{eqF.34}
  & \eta = \text{r.h.s. of \eqref{eqF.22} with $k$ replaced with $k_0-1$}.
\end{align}
We check the conditions of Lemma \ref{lemma 6.8}. First,  it follows from the argument of \eqref{eqF.25} that 
\begin{align}
    \label{eqF.35}
    \eta \in  L_2 ((0, T) \times \Omega) W^{-1}_{2, \theta/2^{k_0-1}} (\bR^3), 
    \quad  \partial_t \eta \in  L_2 ((0, T) \times \Omega) W^{-1}_{2, \theta/2^{k_0}} (\bR^3).
\end{align}
Finally, we check the condition $\nabla_p \cdot b \in L_{\infty} (\Sigma^T)$ with $b = a_g$, where $a_g$ defined in \eqref{eq6.1}.
We note that
\begin{align*}
    & \partial_{p_i} a_g^i (t, x, p) =  -  \partial_{p_i}  \int \Phi^{ i j } (P, Q)   J^{1/2} (q)    \frac{p_i}{2 p_0}  g (t, x, q) \cdot (1, 1) \, dq \\
 &	+  \partial_{p_i} \int \Phi^{ i j } (P, Q)   J^{1/2} (q)  \partial_{q_j}  g (t, x, q) \cdot (1, 1)\, dq =  I_1 + I_2.
\end{align*}
By the estimate \eqref{eqB.2.1} with $k = 1$,
\begin{align}
    \label{eqF.33}
        \|I_1\|_{L_{\infty} (\Sigma^T)  } \le  N \|g\|_{ L_{\infty} ((0, T) \times \Omega) W^1_2 (\bR^3) }.
\end{align}
Next, to handle $I_2$, we will use the identity \eqref{eqB.5.1}:
   \begin{align}
        \label{eqB.5.1.1}
     &  \partial_{p_i} \int    \Phi^{ i j } (P, Q)  J^{1/2} (q)   \partial_{q_j}  g ( q) \, dq  \\
     & =   \partial_{p_i}   \int    \Phi^{ i j } (P, Q)  J^{1/2} (q)  \frac{ q_j}{2 q_0}   g (q)  \, dq  \notag\\
    &  - 4  \int  \frac{P \cdot Q}{p_0 q_0} \bigg((P \cdot Q)^2 - 1\bigg)^{-1/2}  J^{1/2} (q)   g (q) \, dq - \kappa (p) J^{1/2} (p) g (p), \notag
     \end{align}
     where $\kappa (p) = 2^{7/2} \pi p_0 \int_0^{\pi} (1+|p|^2 \sin^2 \theta)^{-3/2} \sin (\theta) \, d\theta$.
By \eqref{eqB.2.1} with $k = 1$,  the first term on the r.h.s. of \eqref{eqB.5.1.1} is bounded by the r.h.s. of \eqref{eqF.33}.
The remaining terms are handled similarly.
Thus, $\|\nabla_p \cdot a_g\|_{ L_{\infty} ((0, T) \times \Omega \times \bR^3) }$ is bounded by the r.h.s. of \eqref{eqF.33}.
Hence, by Lemma \ref{lemma 6.8}, $\partial_t^{k_0} h_2$ is a finite energy solution to \eqref{eqF.22}, and \eqref{eq7.2.11}   holds with $k$ replaced with $k_0$, as claimed.

Next, to deduce that $\partial_t^{k_0} h_2$ is a strong solution that satisfies the desired $S_{r_i}$ regularity in \eqref{eq7.2.11.1}, we use Propositions \ref{proposition 2.1}  - \ref{proposition 2.3}   combined with the argument of \eqref{eqF.26}.   Thus, \textbf{Claim 1} is proved.

\textbf{Proof of Claim 2.} We repeat the proof of \textbf{Claim 1} with one minor modification. We note that to apply Lemma \ref{lemma 6.8}, we need \eqref{eqF.35} to hold, where $\eta$ is defined in \eqref{eqF.34}.
This estimate was established in the proof of the energy bound \eqref{eqF.25}. See the argument of \eqref{eq7.22}. In particular, to handle the cubic terms (see \eqref{eq7.15} - \eqref{eq7.16}),  we need to control certain weighted $L_{\infty}^{t, x, p}$ norms of $\partial_t^k [h_2, \nabla_p h_2], k \le m/2,$
which  was done in \textbf{Claim 1} (see \eqref{eq7.2.11.1}).

$(iii)$ Since $(ii)$ is valid, we only need to verify the continuity equation \eqref{eqF.12}.
To this end, we note that the functions $H_j = J + J^{1/2} h_j, j = 1, 2,$ and $[\bE_j, \bB_j], j = 1, 2$, satisfy  the identities
\begin{align}
  & \label{eq5.6}
    Y H_2^{+} +  (\mathbf{E}_g + v (p)   \times \mathbf{B}_g) \cdot \nabla_p H_2^{+} \\
    & = \cC (H_2^{+}, G^{+}+ G^{-})+ \cC \big(J, J^{1/2} (h_2^{+}+h_2^{-}-g^{+}-g^{-})\big) \notag \\
    &   - (\bE_1 - \bE_g) \cdot \nabla_p J  \notag \\
   & \label{eq5.7}
   Y H_2^{-} - (\mathbf{E}_g + v (p)   \times \mathbf{B}_g) \cdot \nabla_p H_2^{-} \notag  \\
   & = \cC (H_2^{-}, G^{-} +  G^{+}) + \cC \big(J, J^{1/2} (h_2^{+}+h_2^{-}-g^{+}-g^{-})\big)\\
    & +(\bE_1 - \bE_g) \cdot \nabla_p J.  \notag 
 \end{align}
Differentiating formally the above identities $k$ times in $t$ and integrating over $p \in \bR^3$, we obtain the continuity equation \eqref{eqF.12}.
\end{proof}

\printindex


\begin{thebibliography}{}

\bibitem{AS_13} Ch\'erif Amrouche,  Nour El Houda Seloula,  $L_p$-theory for vector potentials and Sobolev's inequalities for vector fields: application to the Stokes equations with pressure boundary conditions. Math. Models Methods Appl. Sci. 23 (2013), no. 1, 37--92.



\bibitem{BP_87}
\newblock  Richard W. Beals, Vladimir Protopopescu,
\newblock Abstract time-dependent transport equations,
\newblock J. Math. Anal. Appl. \textbf{121} (1987), no. 2, 370--405.





\bibitem{BCLP_13} Marco Bramanti, Giovanni Cupini, Ermanno Lanconelli, Enrico Priola,
 Global $L^p$ estimates for degenerate Ornstein-Uhlenbeck operators with variable coefficients. Math. Nachr. 286 (2013), no. 11-12, 1087--1101.


\bibitem{CDG_02} Jason Cantarella, Dennis DeTurck, Herman Gluck,  Vector calculus and the topology of domains in 3-space. Amer. Math. Monthly 109 (2002), no. 5, 409--442.



\bibitem{CKD_19}
Yunbai Cao, Chanwoo Kim, Donghyun Lee, 
 Global strong solutions of the Vlasov-Poisson-Boltzmann system in bounded domains,
Arch. Ration. Mech. Anal. 233 (2019), no. 3, 1027--1130.


\bibitem{CK_23}
Yunbai Cao, Chanwoo Kim, Lipschitz Continuous Solutions of the Vlasov-Maxwell Systems with a Conductor Boundary Condition, Comm. Math. Phys. 403 (2023), no. 1, 529--625.















\bibitem{DGO_22} Hongjie Dong, Yan Guo, Zhimeng Ouyang,
The Vlasov-Poisson-Landau system with the specular-reflection boundary condition,
Arch. Ration. Mech. Anal. 246 (2022), no. 2-3, 333--396.



\bibitem{DGY_21}  Hongjie Dong, Yan Guo, Timur Yastrzhembskiy, Kinetic Fokker-Planck and Landau Equations with Specular Reflection Boundary Condition, 	
Kinet. Relat. Models 15 (2022), no. 3, 467--516.




\bibitem{GWP_23} Hongjie Dong, Yan Guo, Timur Yastrzhembskiy, 
 Asymptotic Stability for Relativistic Vlasov-Maxwell-Landau System in Bounded Domain, 	arXiv:2401.00554 (2023).



\bibitem{DY_21a} Hongjie Dong, Timur Yastrzhembskiy,
Global $L_p$ estimates for kinetic Kolmogorov-Fokker-Planck equations in nondivergence form,   Arch. Ration. Mech. Anal.  245 (2022), no. 1, 501--564 


\bibitem{DL_76}
Georges Duvaut,  Jacques L. Lions, 
 Inequalities in mechanics and physics.
Grundlehren der Mathematischen Wissenschaften, 219
Springer-Verlag, Berlin-New York, 1976, xvi+397 pp.


\bibitem{EGKM_13} Raffaele Esposito, Yan Guo, Chanwoo Kim, Rossana Marra,
Non-isothermal boundary in the Boltzmann theory and Fourier law,
Comm. Math. Phys. 323 (2013), no. 1, 177--239.


\bibitem{EGKM_18} Raffaele Esposito, Yan Guo, Chanwoo Kim, Rossana Marra, 
Stationary solutions to the Boltzmann equation in the hydrodynamic limit,
Ann. PDE 4 (2018), no. 1, Paper No. 1, 119 pp.



\bibitem{G_93} Yan Guo,
Global weak solutions of the Vlasov-Maxwell system with boundary conditions, Comm. Math. Phys. 154 (1993), no. 2, 245--263.


\bibitem{G_94} Yan Guo,
Regularity for the Vlasov equations in a half-space,
Indiana Univ. Math. J. 43 (1994), no. 1, 255--320.


\bibitem{G_95} Yan Guo,
 Singular solutions of the Vlasov-Maxwell system on a half line,
Arch. Ration. Mech. Anal. 131 (1995), no. 3, 241--304.




\bibitem{G_10QAM} Yan Guo, Bounded solutions for the Boltzmann equation, Quart. Appl. Math. 68 (2010), no. 1, 143--148.



\bibitem{G_10}
Yan Guo,
Decay and continuity of the Boltzmann equation in bounded domains, Arch. Ration. Mech. Anal. 197 (2010), no.3, 713--809.

\bibitem{GKTT_16}
Yan Guo, Chanwoo Kim, Daniela Tonon, Ariane Trescases,
BV-regularity of the Boltzmann equation in non-convex domains,
Arch. Ration. Mech. Anal. 220 (2016), no. 3, 1045--1093.




\bibitem{GHJO_20} Yan Guo, Hyung Ju Hwang, Jin Woo Jang, Zhimeng Ouyang, 
The Landau equation with the specular reflection boundary condition,
Arch. Ration. Mech. Anal. 236 (2020), no. 3, 1389--1454.





\bibitem{H_04}
Hyung Ju Hwang, 
 Regularity for the Vlasov-Poisson system in a convex domain,
SIAM J. Math. Anal. 36 (2004), no. 1, 121--171.

\bibitem{HV_10}
Hyung Ju Hwang, Juan J. L. Vel\'azquez, 
 Global existence for the Vlasov-Poisson system in bounded domains,
Arch. Ration. Mech. Anal. 195 (2010), no. 3, 763--796.

\bibitem{HJL_14} Hyung Ju Hwang, Juhi Jang, Juan J. L. Vel\'azquez,
 The Fokker-Planck equation with absorbing boundary conditions,
Arch. Ration. Mech. Anal. 214 (2014), no. 1, 183--233.




\bibitem{K_11} Chanwoo Kim,  Formation and propagation of discontinuity for Boltzmann equation in non-convex domains, Comm. Math. Phys. 308 (2011), no. 3, 641--701.


\bibitem{KGH_20} Jinoh Kim, Yan Guo, Hyung Ju Hwang, 
An $L^2$ to $L^{\infty}$ framework for the Landau equation,
Peking Math. J. 3 (2020), no. 2, 131--202.






\bibitem{L_00} Mohammed Lemou, Linearized Quantum and Relativistic
Fokker-Planck-Landau Equations,
Math. Meth. Appl. Sci., 23  (2000), 1093--1119.



\bibitem{LP_81}
Evgeny Lifshitz, Lev Pitaevskii,  Physical Kinetics. Vol. 10, (1st ed.), 1981, Pergamon Press.







\bibitem{LZ_14}
Shuangqian Liu,  Huijiang Zhao, 
Optimal large-time decay of the relativistic Landau-Maxwell system,
J. Differential Equations 256 (2014), no. 2, 832--857.





\bibitem{NZ_20} Lukas Niebel, Rico Zacher,
Kinetic maximal $L^p$-regularity with temporal weights and application to quasilinear kinetic diffusion equations,
J. Differential Equations 307 (2022), 29--82.


\bibitem{PR_98}
\newblock Sergio Polidoro, Maria A. Ragusa,
\newblock Sobolev-Morrey spaces related to an ultraparabolic equation,
\newblock \emph{Manuscripta Math.} \textbf{96} (1998), no. 3, 371--392.

\bibitem{S_22}
Luis Silvestre, H\"older estimates for kinetic Fokker-Planck equations up to the boundary,
Ars Inven. Anal.(2022), Paper No. 6, 29 pp.


\bibitem{GS_03}  Robert M. Strain,  Yan Guo, 
Stability of the Relativistic Maxwellian
in a Collisional Plasma,
Commun. Math. Phys. 251 (2004), 263--320.




\bibitem{U_86} Seiji Ukai, Solutions of the Boltzmann equation, Patterns and waves, 37--96.
Stud. Math. Appl., 18, North-Holland Publishing Co., Amsterdam, 1986.



\bibitem{vW_92} Wolf von Wahl,  Estimating $\nabla u$ by $\text{div} \, u$ and $\text{curl} \, u$, Math. Methods Appl. Sci. 15 (1992), no. 2, 123--143.




\bibitem{X_15} Qinghua Xiao, 
Large-time behavior of the two-species relativistic Landau-Maxwell system in $\bR^3$, J. Differential Equations, 259 (2015), no.8, 3520--3558.

\bibitem{YY_12} Tong Yang,   Hongjun Yu,
Global solutions to the relativistic Landau-Maxwell system in the whole space,
J. Math. Pures Appl. (9) 97 (2012), no. 6, 602--634.
\end{thebibliography}
\end{document}